\documentclass[11pt]{article}
 
\usepackage{custom_tex}

\usepackage{enumitem}
\allowdisplaybreaks
\usepackage[page]{appendix}
\newcommand{\titlename}{One-way Matching of Datasets with Low Rank Signals}

\newcommand{\x}{\mathtt{x}}
\newcommand{\y}{\mathtt{y}}
\newcommand{\gap}{\mathtt{gap}}
\newcommand{\glob}{\mathtt{glob}}
\newcommand{\unif}{\mathtt{unif}}
\newcommand{\loco}{\mathtt{loco}}
 
\newcommand{\act}{\mathtt{active}}
\newcommand{\sha}{\mathtt{shared}}
\newcommand{\vect}{\textnormal{Vec}}

\begin{document} 
\title{\titlename}

\author[a]{Shuxiao Chen\thanks{Email:
\href{mailto:shuxiaoc@wharton.upenn.edu}{shuxiaoc@wharton.upenn.edu}}}
\author[b]{Sizun Jiang}
\author[a]{Zongming Ma\thanks{Email: 
\href{mailto:zongming@wharton.upenn.edu}{zongming@wharton.upenn.edu}}}
\author[c]{Garry P. Nolan}
\author[c]{Bokai Zhu}
\affil[a]{\textit{University of Pennsylvania}}
\affil[b]{\textit{Beth Israel Deaconess Medical Center}}
\affil[c]{\textit{Stanford University}}

\date{}
\maketitle

\abstract{
We study one-way matching of a pair of datasets with low rank signals. 
Under a stylized model, we first derive information-theoretic limits of matching {under a mismatch proportion loss.} 
We then show that linear assignment with projected data achieves fast rates of convergence and sometimes even minimax rate optimality for this task.
The theoretical error bounds are corroborated by simulated examples.
Furthermore, we illustrate practical use of the matching procedure on two single-cell data examples.
}

\addtocontents{toc}{\protect\setcounter{tocdepth}{2}}
\tableofcontents 
 
 

\section{Introduction} 
Data matching, also referred to as data alignment or record linkage in some fields, has played an increasingly important role and an integrative part in cleaning, pre-processing, exploratory, and inference stages of many modern data analysis pipelines.
A major motivation of the present work is the prevalence of data matching in analyzing \emph{single-cell multi-omics data}.
In single-cell biology research, it is routine to compile datasets obtained in different batches but with similar measurement protocols or under similar experiment conditions.
When handling such datasets, matching similar cells in different datasets is often a critical step for the correction of technical variations and batch effects \cite{tran2020benchmark}.
As another common practice, 
cell biologists routinely integrate datasets with (partially) overlapping biological (e.g., transcriptomic and proteomic) information collected from different experiment conditions, profiling technologies, tissues, or species (e.g., \cite{stuart2019comprehensive,welch2019single,korsunsky2019fast}) to better understand and define cell states.
To achieve such goals, it is necessary to (identify and) align cells in comparable states across related datasets.
Yet another important application is the transfer and integration of complementary biological information across datasets:
for example,
if one dataset contains individual cells' spatial information within a tissue, matching it with a non-spatial single-cell dataset bears the potential of transferring spatial information to a different measurement modality 
(e.g., \cite{zhu2021robust,kriebel2022uinmf}).

The need to match entities in different datasets also arises in other fields.
In computer vision tasks such as motion tracking and object recognition, the processes usually involves a \emph{feature matching} step where features (local patches, features found by convolutional neural networks, etc.) are first computed for each image, and then a matching algorithm is applied to link features between two or across multiple images for downstream analyses.
See the survey \cite{ma2021image} and the references therein.
In health care system and business intelligence applications, \emph{record linkage} methods are routinely used for data cleaning and for generating insights to inform further medical and business decisions \cite{gu2003record,sayers2016probabilistic}. 
In these applications, algorithms are deployed to match identical or similar records in databases from different sources.

\begin{figure}[t] 
\centering
		\includegraphics[width=0.48\linewidth]{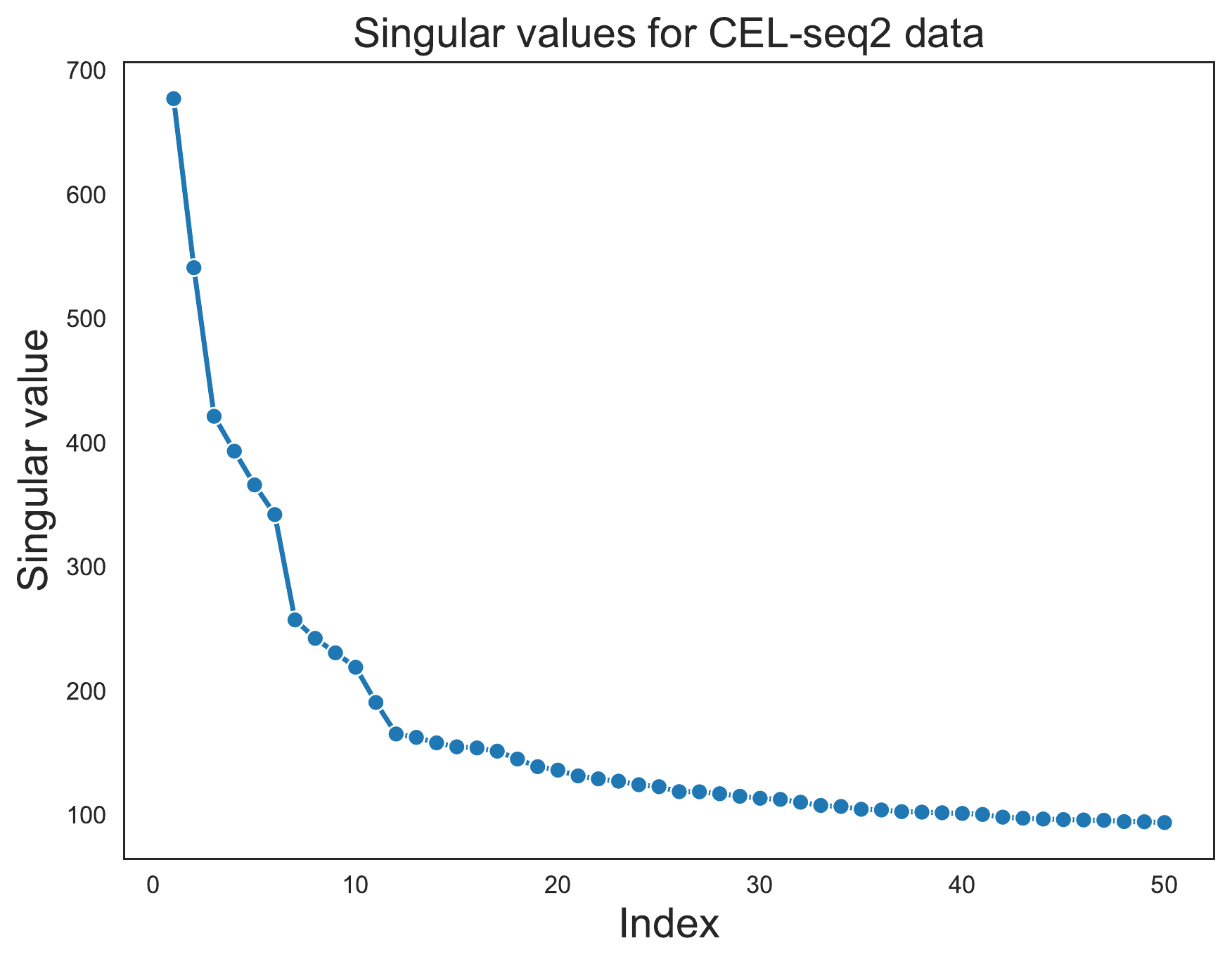}
       \caption{\small Leading singular values of a CEL-seq2 dataset \cite{grun2016novo}. 
       After pre-processing, it contains measurements on the expression levels of $2508$ RNAs in $1935$ cells. See Section \ref{subsec:rna} for details.
	   }
	\label{fig:scree_plot}
\end{figure}

Suppose we have two datasets, represented by two matrices $X\in \mathbb{R}^{n\times p}$ and $Y\in \mathbb{R}^{m\times p}$.
Without loss of generality, assume $n\leq m$.
In the scenario of matching single-cell datasets, rows of the two datasets correspond to cells and each column represents a shared feature (transcript, protein, etc.) measured by both datasets\footnote{In reality, either dataset may contain features that are not measured in the other. See, for instance, the spatial proteomics data example in Section \ref{sec:proteomics}. 
In such cases, we assume that common features shared by the two datasets have been identified and $X$ and $Y$ represent the datasets with only shared features.}. 
In feature matching, each row of $X$ and $Y$ corresponds to a $p$-dimensional feature. 
We assume that both data matrices admit a ``signal $+$ noise'' structure
\begin{equation*}
	X = M_\x + Z_\x 
	\quad \mbox{and}\quad 
	Y = M_\y + Z_y,
\end{equation*}
where $M_\x$, $M_\y$ are deterministic signals and $Z_\x$, $Z_\y$ are additive noise components.
In this paper,
we focus on the case where both $M_\x$ and $M_\y$ are of low rank.
This is usually the case for single-cell datasets. 
In Figure \ref{fig:scree_plot}, we plot leading singular values of a pre-processed single-cell transcriptomic dataset \cite{grun2016novo} obtained via CEL-seq2 technology.
The fast decay of leading singular values indicates that the signal in this dataset is (at least approximately) of low rank.
Similar patterns are also observed on other single-cell datasets with different measurement modalities and technologies.
In an ideal situation, $M_\y$ contains $M_\x$ as a submatrix up to an unknown permutation of rows. 
In other words, the rows in $X$ can be matched to a subset of rows in $Y$.
Our goal is to recover this unknown matching based on observing $X$ and $Y$ only.

\subsection{One-way matching as estimation}

In this paper, we further restrict our attention to the simplified case of $m=n$ and so both dimensions of $X$ and $Y$ agree.
As we shall see, investigation in this special case is already challenging and we leave theoretical study of the general case $n\leq m$ for future work.
In motivating single cell data examples, such an assumption is not overly restrictive if the two datasets are obtained on homogeneous samples from identical or comparable cell populations as we could first aggregate similar cells (or downsample) within respective datasets to align sample sizes before matching.

When $m=n$, we could specialize the foregoing general model to 
\begin{equation}
\label{eq:model}
\begin{aligned}
	X & = UDV^\top + \sigma_\x N_\x,\\
\Pi^\star Y & =	UDV^\top + \sigma_\y N_\y.
\end{aligned}
\end{equation}
Here, $U \in O_{n, r}, V \in O_{p, r}$ are two orthonormal matrices, $D = \diag(d_1, \hdots, d_r)$ is a diagonal matrix with descending positive diagonal entries, $N_\x, N_\y$ are two independent random matrices with independent standard Gaussian entries, and $\sigma_\x$ and $\sigma_\y$ are noise standard deviations.
Furthermore, $\Pi^\star$ is an unknown permutation matrix that encodes the matching between rows in two data matrices.
In other words, after row permutation by $\Pi^\star$, the rows in the signal component of $Y$ are identical to those in the signal component of $X$, while 
the noise components are independent.
For brevity, we write $(X,Y)\sim \calM(U,D,V, \sigma_\x,\sigma_\y, \Pi^\star)$ when $X$ and $Y$ are generated by model \eqref{eq:model}.

Each permutation matrix $\Pi$ {has an one-to-one correspondence to a vector} $\pi\in S_n$ where $S_n$ collects all permutations of the set $\{1,\dots, n\}$.
With slight abuse of notation, for any permutation matrix $\Pi$, we also write $\Pi\in S_n$.
Denote the representations of the true matrix $\Pi^\star$ and any estimator $\widehat{\Pi}$ by $\pi^\star$ and $\widehat{\pi}$, respectively. 
The rest of this manuscript focuses on minimax estimation of $\Pi^\star$ under the following normalized Hamming loss 
{
\begin{equation}
\label{eq:loss}
	\frac{1}{n}d(\hat\pi, \pi^\star) = \frac{1}{n}\sum_{i\in[n]} \indc{\hat\pi_i \neq \pi^\star_i}
	= \frac{1}{2n}\|\widehat\Pi - \Pi^\star\|_F^2 
	= \ell(\hat\Pi, \Pi^\star).
\end{equation}
}
{The loss function $\ell(\hat\Pi, \Pi^\star)$ can be understood as the mismatch proportion of an estimator $\hat\Pi$ for the ground truth permutation matrix $\Pi^\star$.}
Finally, to cast the estimation problem in a decision-theoretic framework, we shall focus on uniform error bounds over the following class of parameter spaces:
\begin{equation}
\label{eq:para-space}
\calP_n(U, D, V, \sigma_\x, \sigma_\y)	= \{
(X,Y)\sim \calM(U,D,V, \sigma_\x,\sigma_\y, \Pi^\star): \Pi^\star \in S_n
\}.
\end{equation} 
We consider the asymptotic regime where $n$ tends to infinity and all other parameters are allowed to scale with $n$.

\subsection{Main results}

The theoretical contribution of this paper has three aspects:
\begin{enumerate} 
\item Under the foregoing decision-theoretic framework, we derive minimax lower bounds for estimating $\Pi^\star$.
These lower bounds are governed by pairwise separation of rows in the signal component.

\item We consider a highly intuitive matching algorithm based on solving a linear assignment problem on projected data. 
The projection directions are estimated from data, and so the method is completely data-driven.
This method has been used as a critical step in a full pipeline for integrating single-cell datasets in the companion paper \cite{zhu2021robust}.
With minimal assumption, we derive a uniform high probability error rate that is polynomial in minimum pairwise separation of rows in the signal component.

\item Under a more stringent ``weak symmetry condition'', we further improve the error bounds of the same algorithm to have exponential decay in minimum pairwise separation of rows. 
Under an even stronger ``strong symmetry condition'', the constant in the exponent can match that in the lower bound and hence is sharp.

\end{enumerate} 

With the foregoing three aspects covered, we achieve the goal of showing that an empirically well-performing matching algorithm has guaranteed generality, at least under a class of stylized models. 
This shall facilitate further theoretical justification for more complex model classes and related matching methods.

\subsection{Related works}

\citet{collier2016minimax} studied estimation of $\Pi^\star$ under model \eqref{eq:model} with full rank signal components, i.e., when $r = \min(n,p)$.
The primary focus of \cite{collier2016minimax} is on minimax rate of separation for exact recovery of $\Pi^\star$ with full rank signals. 
See \cite{galstyan2021optimal} for a nontrivial extension to the case of unequal sizes.
In contrast, the present work focuses on minimax nearly-exact recovery rate of $\Pi^\star$ with low rank signals.
Under a different correlated Gaussian feature model, \citet{dai2019database} and \citet{dai2020achievability} studied information-theoretic limits for both exact and nearly-exact recovery and their achievability, which is built upon the investigation in \cite{cullina2018fundamental} on exact recovery when correlated features are randomly drawn from finite-alphabets.
See also \citet{kunisky2022strong} and \citet{wang2022random} for more refined analyses under such correlated Gaussian feature models.
In both lines of literature, linear assignment problem plays a critical role in achievability of respective information-theoretic limits. 

In single-cell data analysis literature, a majority of popular matching approaches rely on the concept of mutual nearest neighbor (MNN), e.g., \cite{haghverdi2018batch,barkas2019joint,stuart2019comprehensive}.
Although methods differ in details, the overall spirit is the same: 
One first finds for each data entry in one dataset its $k$-nearest neighbors in the other dataset under some distance measure, which is feasible as long as the two datasets have column correspondences. 
A pair of data entries from two datasets are mutual nearest neighbors if they appear in each other's neighborhood and are hence matched.
A major drawback of MNN-based approaches is easy to perceive: It is suitable only when signal-to-noise ratio is relatively high as otherwise noise can significantly blur local neighborhoods and there could exist very few or even no MNNs.
There exists an alternative approach based on non-negative matrix factorization \cite{welch2019single} which is closer in spirit to the approach in the present work. However, it lacks theoretical justification.

Estimation of permutation also appears in other contexts. 
For example, unlabeled linear regression (e.g., \cite{pananjady2017linear,pananjady2017denoising,zhang2020optimal}) and graph matching (e.g., \cite{ding2021efficient,fan2019spectral_I,fan2019spectral_II,ganassali2020tree,mao2021exact}). 
However, in these settings, the matching problem is of a different nature: there are unknown parameters governing both dimensions of the data matrices to be matched. Hence, the matching problem is ``two-way'' in nature whereas the present work focuses exclusively on the ``one-way'' setting: the columns of datasets are already aligned and only permutation of rows is to be estimated.

\subsection{Notation}
Throughout the paper,
for any positive integer $l$, let $[l] = \{1,\dots,l\}$.
For permutation $\pi$ any ordered sequence of distinct indices $(i_1,\dots,i_k)$ and $(\pi_{i_1},\dots,\pi_{i_k})$, we abbreviate them as $i_{1:k}$ and $\pi_{i_{1:k}}$, respectively.
For any real numbers $a$ and $b$, we let $a\wedge b = \min(a,b)$ and $a\vee b = \max(a,b)$.
For any two sequences of positive numbers $a_n$ and $b_n$, we write $a_n = \calO(b_n)$, $a_n\lesssim b_n$ or $b_n\gtrsim a_n$ if $\limsup a_n/b_n < \infty$ and $a_n = o(b_n)$, $a_n \ll b_n$ or $b_n \gg a_n$ if $\limsup a_n/b_n = 0$. We write $a_n = \Omega(b_n)$ when $b_n =\calO(a_n)$.
Furthermore, we write $a_n\asymp b_n$ if $a_n \lesssim b_n$ and $b_n\lesssim a_n$ hold simultaneously.
For any vector $v$, $\|v\|$ denotes its Euclidean norm.
For any matrix $A$, $\|A\|_F$ denotes it Frobenius norm, $\|A\|$ denotes its spectral norm, and $\mathrm{tr}(A) = \sum_i A_{ii}$ denotes its trace.
For a pair of matrices $A$ and $B$ with identical dimensions, $\langle A, B\rangle = \mathrm{tr}(A^\top B)$ denotes their trace inner product.
For any subset $I$ of row indices and $J$ of column indices of a $s\times t$ matrix $A$, we use $A_{I,J}$ to denote the submatrix indexed by $I$ and $J$.
We also let $A_{I,\cdot} = A_{I,[t]}$, $A_{\cdot,J} = A_{[s],J}$, $A_{-I,J} = A_{[s]\backslash I,J}$ and $A_{I,-J} = A_{I,[t]\backslash J}$.
For any integers $s\geq t>0$, $O_{s,t}$ denotes the collection of all $s\times t$ orthogonal matrices.
Additional notation will be defined at first occurrence.
 

\subsection{Paper organization}
We derive information-theoretic lower bounds in Section \ref{sec:lb}. The algorithm and its theoretical analysis are given in Section \ref{sec:ub}. Section \ref{sec:sim} presents simulation results that corroborate the theory developed. 
The proposed algorithm is applied to two single-cell data examples in Section \ref{sec:data}.

\section{Fundamental limits of one-way matching} 
\label{sec:lb}
In this section, we present a minimax lower bound under the model specified in \eqref{eq:model}. To start with, we present a lemma that decomposes the expected number of mismatches into errors coming from different sources.

\begin{lemma}[Cycle decomposition of expected mismatches]
\label{lemma:cyc_decomp}
Let $\hat\pi = \hat \pi(X, Y)$ be any estimator of $\pi^\star$. Then we have
\begin{equation}
	\label{eq:cyc_decomp}
	\bbE[d(\hat\pi, \pi^\star)] = \sum_{k=2}^n \sum_{i_1\neq i_2 \neq \cdots \neq i_k} 
	 \bbP(\hat\pi_{i_1} = \pi^\star_{i_2}, \hat\pi_{i_2} = \pi^\star_{i_3}, \hdots, \hat\pi_{i_{k-1}} = \pi^\star_{i_{k}}, \hat\pi_{i_k} = \pi^\star_{i_1}),
\end{equation} 
where the second summation on the right-hand side above is taken over all collections of $k$ distinct indices $1\leq i_1 \neq i_2 \neq \cdots \neq i_k \leq n$.
\end{lemma}
\begin{proof}
	See Appendix \ref{prf:lemma:cyc_decomp}.
\end{proof}
The above lemma is useful in that it makes clear where the error, as measured by expected mismatches, comes from: each summand in \eqref{eq:cyc_decomp} is precisely the error resulting from {mis-estimating} the matching on $\pi^\star_{i_{1:k}}$, and the error takes a cyclic form.
In graph-theoretic terminologies, the above lemma decomposes the expected mismatches into errors resulting from all possible cycles on a complete graph with $n$ nodes.

If we only retain the cycles of length two, we arrive at the following lower bound on minimax risk {by observing \eqref{eq:loss}}:
{
\begin{equation}
	\label{eq:cyc_decomp_two_cycle}
	\inf_{\hat \Pi(X, Y)} \sup_{\calP_n(U, D, V, \sigma_\x, \sigma_\y)}\bbE[\ell(\hat\Pi, \Pi^\star)] 
	\geq 
	\inf_{\hat \Pi(X, Y)} \sup_{\calP_n(U, D, V, \sigma_\x, \sigma_\y)}\frac{1}{n}\sum_{i \neq i'}
	 \bbP(\hat\pi_{i} = \pi^\star_{i'}, \hat\pi_{i'} = \pi^\star_{i}),
\end{equation} 
where we recall that $\hat\pi$ (resp. $\pi^\star$) is the vector representation of $\hat\Pi$ (resp. $\Pi^\star$).
}
Using some reduction arguments, one can further lower bound each summand on the right-hand side of \eqref{eq:cyc_decomp_two_cycle} by the optimal testing error of the following binary hypothesis testing problem:
\begin{equation}
	\label{eq:testing_problem}
	H_0: \pi^\star_i = j, \pi^\star_{i'} = j' 
	\qquad \textnormal{v.s.}
	\qquad
	H_1: \pi^\star_i = j', \pi^\star_{i'} = j,
\end{equation}
where one knows all the values of $\pi^\star$ other than $(\pi^\star_i, \pi^\star_{i'})$, and is trying to differentiate the two possibilities specified by $H_0$ and $H_1$ from the data $(X, Y)$ at hand. By Neyman--Pearson lemma, the optimal testing procedure is the likelihood ratio test, and a careful calculation gives the following minimax lower bound.

\begin{theorem}[Minimax lower bound]
Let $\sigma_{\max} := \sigma_\x \lor \sigma_\y$. We have
	\label{thm:lb}
	{
\begin{align}
	\label{eq:lb}
	\inf_{\hat\Pi(X , Y )} \sup_{\calP_n(U, D, V, \sigma_\x, \sigma_\y)} \bbE[\ell(\hat \Pi, \Pi^\star)] 
	\geq \frac{1}{n} \sum_{i\neq i'} \Phi\left(- \frac{\|(U_{i, \bigcdot} - U_{i', \bigcdot})D\|}{\sqrt{2}\cdot\sigma_{\max}}\right),
	\end{align}	
	}
where the infimum is taken over all randomized measurable functions of $(X, Y)$ and $\Phi(\cdot)$ is the cumulative distribution function of standard Gaussian random variables.
In particular, as long as the minimum pairwise separation
\begin{equation}
	\label{eq:beta}
	\beta^2:= \min_{i\neq i'} \frac{\|(U_{i, \bigcdot} - U_{i', \bigcdot})D\|^2}{\sigma_{\max}^2} \gg 1,
\end{equation}
there exists a sequence $\underline{\delta}_n = o(1)$ such that
{
	\begin{align}
		\label{eq:lb_clear_form}
		\inf_{\hat\Pi(X , Y )} \sup_{\calP_n(U, D, V, \sigma_\x, \sigma_\y)} \bbE[\ell(\hat \Pi, \Pi^\star)]
		\geq \frac{1}{n} \sum_{i\neq i'} \exp\left\{-\frac{(1+\underline{\delta}_n)\|(U_{i, \bigcdot} - U_{i', \bigcdot})D\|^2}{4\sigma_{\max}^2}\right\}.
	\end{align}
	}
\end{theorem}
\begin{proof}
	See Appendix \ref{prf:thm:lb}.
\end{proof}

The above minimax lower bound is derived by only retaining the errors from cycles of length two, and its tightness depends on whether such errors constitute the dominant part in the full cycle decomposition given in \eqref{eq:cyc_decomp}. In fact, we are to see in Section \ref{sec:ub} that this is indeed the case under certain symmetry conditions.

As long as the signal matrix $UD$ is configured such that there exist at least $\Omega(n)$ pairs of indices $i\neq i$ (among all $n(n-1)$ pairs) with $\|(U_{i, \bigcdot} - U_{i', \bigcdot})D\|^2/\sigma_{\max}^2 = \beta^2$, then the lower bound in \eqref{eq:lb} can be further lower bounded by
\begin{align}
	\label{eq:lb_in_beta}
	\inf_{\hat\Pi(X , Y )} \sup_{\calP_n(U, D, V, \sigma_\x, \sigma_\y)} \bbE[\ell(\hat \Pi, \Pi^\star)] 
		\geq \exp\left\{-\frac{-(1+o(1))\beta^2}{4}\right\}.
\end{align}		
In this sense, $\beta \gg 1$ becomes a necessary condition for consistent estimation of $\pi^\star$.



\section{Linear assignment with projected signals}
\label{sec:ub}

We describe an intuitive algorithm, \ul{l}inear \ul{a}ssignment with \ul{p}roject \ul{s}ignals (LAPS), that solves a linear assignment problem after projecting the potentially high-dimensional signals onto a lower-dimensional subspace. We then theoretically characterize the performance of LAPS under the model given in \eqref{eq:model}.

Suppose we know $r$ and we know which dataset is less noisy\footnote{If unknown, the knowledge could be relatively easily obtained by investigating the singular values of both datasets \cite{benaych2012singular}.
}, i.e., which one of $\sigma_\x$ and $\sigma_\y$ is smaller. 
In this case, we can perform a truncated singular value decomposition (SVD) on the less noisy dataset and collect its top $r$ right singular vectors in a matrix $\hat V\in O_{p, r}$. 
For example, if $\sigma_\x \leq \sigma_\y$, we perform SVD on $X$ and let columns of $\hat V$ be the top $r$ right singular vectors. 
The LAPS algorithm solves the following linear assignment problem after projecting the two datasets using $\hat V$: 
{
\begin{align}
	\label{eq:linear_assignment}
	\hat \Pi  \in \argmax_{\Pi \in S_n} \la X \hat V , \Pi Y \hat V \ra = \argmin_{\Pi\in S_n} \| X \hat V - \Pi Y \hat V \|_F^2.
\end{align}
}
This optimization problem can be solved in polynomial time by, e.g., the Hungarian algorithm \cite{kuhn1955hungarian}.

If we observe $(XV, YV)$, then maximizing the inner product between $XV$ and $\Pi YV$ would give rise to the maximum likelihood estimator. In this regard, LAPS can be interpreted as an approximate maximum likelihood estimator where the nuisance parameter $V$ is estimated from data.

\subsection{Polynomial rate and consistency}
In this subsection, we show LAPS achieves consistency (i.e., $o(1)$ mismatch proportion) with high probability under mild conditions.

To start with, we present a proposition that bounds the estimation error of the right singular space under a spectral gap assumption. 
Recall that $d_r$ is the smallest non-zero singular value of the signal component in model \eqref{eq:model}.

\begin{proposition}[Estimation error of right singular subspace]
	\label{prop:est_err_of_V}
	Let $\sigma_{\min} = \sigma_{\x} \land \sigma_{\y}$.
	Suppose
	\begin{equation}
		\label{eq:eigen_gap_assump}
		\frac{d_r}{\sigma_{\min}} \geq C_\gap  (\sqrt{n} + \sqrt{p}),
	\end{equation}
	for some sufficiently large constant $C_\gap > 0$ and let $c > 0$ be any constant.
	There exists another constant $C$ only depending on $C_\gap$ and $c$, such that with probability at least $1 - n^{-c}$, we have
	\begin{equation}
		\label{eq:est_err_of_V}
		\|\hat V \hat V^\top - V V^\top\|_F\leq C \cdot \calE_\unif
		\quad\mbox{with}\quad
		\calE_\unif :=  \frac{\sqrt{rp\log n}}{d_r/\sigma_{\min}},
	\end{equation}
	where we recall that $\hat V \in O_{p, r}$ collects the top $r$ right singular vectors of the less noisy data matrix.
\end{proposition}
\begin{proof}
	See Appendix \ref{prf:prop:est_err_of_V}.
\end{proof}

The above proposition is a consequence of the version of sin-theta theorem proved in \cite{cai2018rate}. The dependence of $\calE_\unif$ on $\sigma_{\min}$ demonstrates the benefit of estimating $V$ from the less noisy data matrix. As long as $V$ can be consistently estimated, one would expect that $\la X\hat V, \Pi Y \hat V \ra$ is a good proxy of the ground truth maximum likelihood objective $\la X V, \Pi Y V\ra$, the maximization of which would give a reasonable estimate of $\pi^\star$. The following theorem makes this point clear.

\begin{theorem}[Polynomial error rate of LAPS]
\label{thm:poly_rate}
Assume \eqref{eq:eigen_gap_assump} holds and $\calE_\unif \ll 1$.
Then \textcolor{black}{uniformly over $\calP_n(U, D, V, \sigma_\x, \sigma_\y)$} we have
{
\begin{equation}
	\label{eq:poly_rate}
	\ell(\hat \Pi, \Pi^\star) \lesssim \frac{p}{\beta^2} \cdot \calE_\unif \left(\frac{1}{\sqrt{n}}+ \frac{1}{\sqrt{p}}\right) + \frac{r}{\beta^2}.
\end{equation}
}
with probability at least $1-n^{-c}$, where $c>0$ is some absolute constant. In particular, LAPS achieves $o(1)$ error with high probability uniformly over the parameter space as long as 
\begin{equation}
	\label{eq:poly_rate_snr_assump}
	\frac{\beta^2}{r} \gg1 \qquad\textnormal{and}\qquad \frac{\beta^2}{p} \gg \calE_\unif \left(\frac{1}{\sqrt{n}}+ \frac{1}{\sqrt{p}}\right).
\end{equation}
\end{theorem}
\begin{proof}
	See Appendix \ref{prf:thm:poly_rate}. 
\end{proof}

The error bound given in the above theorem is inverse proportional to the minimum pairwise separation $\beta^2$. The proof is based on the basic inequality
\begin{equation}
	\label{eq:basic_ineq}
	\|X \hat V -  \hat \Pi Y \hat V\|_F^2 \leq \|X \hat V - \Pi^\star Y \hat V\|_F^2.
\end{equation}

\subsection{Exponential rate and minimax optimality}
Theorem \ref{thm:poly_rate} is not entirely satisfactory as the error rate is a polynomial function of the minimum pairwise separation $\beta^2$. In contrast, the lower bound in Theorem \ref{thm:lb} is exponential in the pairwise separations. In this section, we show LAPS is capable of achieving exponential error rate and sometimes even minimax optimality under additional symmetry conditions.

Instead of invoking the basic inequality \eqref{eq:basic_ineq}, the key step in establishing exponential error rate is to invoke the cycle decomposition in Lemma \ref{lemma:cyc_decomp} and carefully compute the errors coming from cycles of all possible lengths. One can show that each summand in \eqref{eq:cyc_decomp} can be upper bounded by 
\begin{equation}
	\label{eq:upper_bound_on_each_term_in_cyc_decomp}
	\bbP\left( \left\la \sigma_\x^{-1}{(X \hat V)_{i_{1:k}, \bigcdot}}, ~ \sigma_\y^{-1}{(I_k^\leftarrow - I_k) (Y \hat V)_{\pi^\star_{i_{1:k}},\bigcdot}} \right\ra \geq 0\right),
\end{equation}	
where $I_k^\leftarrow$ is a $k\times k$ matrix defined as
$$
	I_k^\leftarrow := \begin{pmatrix}
		0 & 1 & 0 & \cdots & 0 \\
		0 & 0 & 1 & \cdots & 0 \\
		\vdots & \vdots & \vdots & \ddots & \vdots \\
		0 & 0 & 0 & \hdots & 1\\
		1 & 0 & 0 & \hdots & 0
	\end{pmatrix},
	\qquad 
	\textnormal{so that } \forall v \in \bbR^{k},
		I_k^\leftarrow
	\begin{pmatrix}
		v_1 \\
		v_2 \\
		\vdots\\
		v_{k-1}\\
		v_k
	\end{pmatrix}
	= 
	\begin{pmatrix}
		v_2 \\
		v_3\\
		\vdots\\
		v_k\\
		v_1
	\end{pmatrix}.
$$
If $\hat V$ is independent of the data $(X, Y)$, then by Markov's inequality, upper bounding \eqref{eq:upper_bound_on_each_term_in_cyc_decomp} can be done by calculating moment generating function of a quadratic function of Gaussian matrices. 
However, the fact that $V$ is estimated from the data introduces non-trivial dependence structure that complicates the proof. 

One remedy to the issue of reusing the data is to argue that $\hat V$ is close to $V$ (e.g., by invoking Proposition \ref{prop:est_err_of_V}), so that \eqref{eq:upper_bound_on_each_term_in_cyc_decomp} remains close to the probability when $\hat V$ is replaced by $V$ uniformly over all possible choices of $1\leq i_1\neq i_2 \cdots \neq i_k\leq n$. Such a uniform strategy is clearly sub-optimal when the length of the cycle is small. 

Without loss of generality, let us assume $X$ is less noisy (i.e., $\sigma_\x \leq \sigma_\y$). When $k$ is small, one would expect that $\hat V$ to be close to $\hat V^{(-i_{1:k})}$, the \emph{leave-one-cycle-out} (LOCO) estimate of $V$ that collect top $r$ right singular vectors of $X^{(-i_{1:k})}\in\bbR^{n\times p}$, the matrix with the noise component $N_{i_{1:k}, \bigcdot}$ removed. That is,
\begin{equation}
	\label{eq:loco_data_matrix}
	X^{(-i_{1:k})} := X - N_\x^{(-i_{1:k})}, \qquad [N_{\x}^{(-i_{1:k})}]_{i,\bigcdot} := \indc{i\in i_{1:k}} \cdot (N_{\x})_{i,\bigcdot}.
\end{equation}	
Since $X_{i_{1:k},\bigcdot}$ is independent of $\hat V^{(-i_{1:k})}$, one can condition on the value of $\hat V^{(-i_{1:k})}$ and apply Markov's inequality and compute the moment generating function to upper bound \eqref{eq:upper_bound_on_each_term_in_cyc_decomp} with $\hat V$ replaced by $\hat V^{(-i_{1:k})}$. 
Such a LOCO strategy will yield better result when $\hat V$ is closer to $\hat V^{(-i_{1:k})}$ than $V$, which can occur when $k$ is relatively small. 
The LOCO strategy is in spirit similar to the leave-one-out strategy \cite{el2013robust} that has been successfully employed in many theoretical studies to address statistical dependence (see Section 4 of the monograph \cite{chen2021spectral} and references therein).

It turns out that the distance between $\hat V$ and $\hat V^{(-i_{1:k})}$ depends crucially on how the entries of $U$ are spread across its rows.
Let us introduce the incoherence parameter $\mu \in [1, n]$, defined as the smallest number that satisfies
\begin{equation}
	\max_{i\in[n]} \|U_{i, \bigcdot}\| \leq \sqrt{\frac{\mu}{n}} \|U\|_F = \sqrt{\frac{\mu r}{n}}.
\end{equation}
A small $\mu$ means the entries of $U$ are spread out, so that deleting a small fraction of rows will not significantly affect the right singular space.
The following proposition bounds the distance between $\hat V$ and $\hat V^{(-i_{1:k})}$.

\begin{proposition}[LOCO error of right singular subspace]
\label{prop:loco_err_of_V}
Let $\calC_k = \{i_1, ..., i_k\} \subseteq[n]$ be a collection of $k$ distinct indices 
and let columns of $\hat V^{(-\calC_k)}$ collect top $r$ right singular vectors of either $X^{(-\calC_k)}$ or $Y^{(-\calC_k)}$ defined in \eqref{eq:loco_data_matrix}, whichever has a lower noise level.
If \eqref{eq:eigen_gap_assump} holds
for some sufficiently large constant $C_\gap > 0$ 
then for any $k^\star \in[n]$ and $c \geq 0$, we have
\textcolor{black}{uniformly over $\calP_n(U, D, V, \sigma_\x, \sigma_\y)$},
\begin{align*}
	\bbP\left[ \max_{k\leq k^\star}\max_{\calC_k} \|\hat V^{(-\calC_k)} (\hat V^{(-\calC_k)})^\top - \hat V \hat V^\top \|_F \leq Ck^\star \cdot \calE_\loco \right] \geq 1- n^{-c},\\
\end{align*}
where
\begin{equation}
	\label{eq:loco_err_of_V}
	\calE_\loco := 
	\frac{\sqrt{\mu r} (\sqrt{p} + \sqrt{\log n}) d_1/d_r}{\sqrt{n} d_r/\sigma_{\min}}  + \frac{(\sqrt{p} + \sqrt{\log n})^2}{d_r^2/\sigma_{\min}^2}
\end{equation}
and $C$ is an absolute constant only depending on $C_\gap$ and $c$.
\end{proposition}
\begin{proof}
	See Appendix \ref{prf:prop:loco_err_of_V}.
\end{proof}

Now the strategy is clear: we seek to find the best cutoff $k^\star$, such that when $k> k^\star$, we use uniform arguments (which invokes Proposition \ref{prop:est_err_of_V}), and when $k\leq k^\star$, we use LOCO arguments (which invokes Proposition \ref{prop:loco_err_of_V}). 

Before proceeding, we pause to do some simple calculations to understand when LOCO bounds can improve upon uniform bounds, namely when $\calE_\loco \ll \calE_\unif$. Note that
$$
	\frac{\calE_\loco}{\calE_\unif} = \left(\frac{d_1}{d_r}\sqrt{\frac{\mu}{n}} + \frac{\sqrt{p} + \sqrt{\log n}}{\sqrt{r}d_r/\sigma_{\min}}\right) \left(\frac{1}{\sqrt{\log n}} + \frac{1}{\sqrt{p}}\right) \leq \frac{d_1}{d_r}\sqrt{\frac{\mu}{n}} \left(\frac{1}{\sqrt{\log n}} + \frac{1}{\sqrt{p}}\right) + o(1)
$$
where the inequality holds when $\calE_\unif\ll 1$.
Thus, $\calE_\loco \ll \calE_\unif$ holds provided
$$
	\mu \ll n(\log n \land p) \cdot \frac{d_r^2}{d_1^2}.
$$
Since the maximum possible value of $\mu$ is $n$, the above condition holds as long as $p$ tends to infinity and $d_1\asymp d_r$.
 
We are now ready to state the theorem that gives the exponential error rate for LAPS.
\begin{theorem}[Exponential rate of LAPS]
	\label{thm:ub_simplified}
	Recall the definition of $\calE_\unif$ in \eqref{eq:est_err_of_V} and $\calE_\loco$ in \eqref{eq:loco_err_of_V}.
	Assume \eqref{eq:eigen_gap_assump} holds for some sufficiently large $C_\gap > 0$ and $\sigma_\x \asymp \sigma_\y, d_1 \asymp d_r$.
	In addition, assume $\calE_\unif \ll 1, \beta^2 \gg r$ and 
		\begin{align*}
			\label{eq:ub_simplified_unif_and_loco_bound_condition}
			\frac{\beta^2}{p}
			& \gg
			r\log n \cdot \frac{\calE_\loco}{\calE_\unif} 
			+  (r\log n)^{3/5} \cdot \calE_\loco^{4/5} 
			+  (r\log n)^{1/3} \cdot \calE_\loco^{2/3}
			+ \calE_\unif^{1/2} \calE_\loco^{1/2}. \numberthis
		\end{align*}
	Now, if for any $o(1)$ sequence $\delta_n$, we have
	\begin{equation}
	\label{eq:ub_vanishing_err}
	\frac{1}{n} \sum_{k=2}^n \sum_{i_1\neq \cdots\neq i_k} \exp\bigg\{\frac{-(1-\delta_n)C_k\left\|(I_k^\leftarrow -I_k)U_{i_{1:k},\bigcdot} D\right\|_F^2}{\sigma_{\max}^2}\bigg\} = o(1),
	\end{equation}	
	where
	\begin{align}
		\label{eq:Ck}
		C_k = 
		\begin{cases}
			1/8 = 0.1250 & \textnormal{if }  k = 2\\
			1/12 \approx 0.0833 & \textnormal{if } k = 3\\
			1/14 \approx 0.0714 & \textnormal{if } k = 4 \\
			11/181 \approx 0.0608 & \textnormal{if } k = 5\\
			(3-2\sqrt{2})/4 \approx 0.0429 & \textnormal{if } k\geq 6,
		\end{cases}
	\end{align}
	then there exists two sequences $\overline{\delta}_n, \overline{\delta}_n' = o(1)$ such that \textcolor{black}{uniformly over $\calP_n(U, D, V, \sigma_\x, \sigma_\y)$} with probability $1-o(1)$, we have
	{
	\begin{align}
	\label{eq:ub_mismatch_proportion}
	\ell(\hat \Pi, \Pi^\star)  \leq \left(\frac{1}{n} \sum_{k=2}^n \sum_{i_1\neq \cdots\neq i_k} \exp\left\{\frac{-(1-\overline{\delta}_n)C_k\left\|(I_k^\leftarrow -I_k)U_{i_{1:k},\bigcdot} D\right\|_F^2}{\sigma_{\max}^2}\right\}\right)^{1-\overline{\delta}_n'}.
	\end{align}
	}
\end{theorem}
\begin{proof}
	See Appendix \ref{prf:thm:ub_simplified}.
\end{proof}

The bound in \eqref{eq:ub_mismatch_proportion} takes an exponential form, and each summand comes from the error incurred by a cycle $(i_1, \hdots, i_k)$.

In the above theorem, we have assumed $\sigma_\x \asymp \sigma_\y$ and $d_1 \asymp d_r$ for ease of exposition, and we refer the readers to Theorem \ref{thm:ub} for a general version with those two assumptions removed. The assumption $\calE_\unif \ll 1$ ensures the right singular subspace can be consistently estimated. The condition $\beta^2 \gg r$ is almost necessary for consistent estimation of {$\Pi^\star$} in view of \eqref{eq:lb_in_beta}.

The condition in \eqref{eq:ub_simplified_unif_and_loco_bound_condition} arises from the hybrid strategy that combines uniform arguments and LOCO arguments.
In fact, an adapted version of the proof will give the same exponential rate \eqref{eq:ub_mismatch_proportion_under_futher_assumption_2} for the naive algorithm (i.e., linear assignment without projection) provided $\beta^2 \gg p$, and this assumption is 
usually stronger than \eqref{eq:ub_simplified_unif_and_loco_bound_condition} when $r$ does not grow too fast with $n$ (and $p$).



The next corollary states that under some additional symmetry conditions, the upper bound in \eqref{eq:ub_mismatch_proportion} can be simplified to match the lower bound in Theorem \ref{thm:lb}.

\begin{corollary}[Minimax optimality of LAPS]
\label{cor:ub_mismatch_proportion_under_further_assumptions}
Let the assumptions of Theorem \ref{thm:ub_simplified} hold. We extend the definition of $C_k$ such that $C_1 = 1/4$ and $C_k$ is given in \eqref{eq:Ck} for $2\leq k \leq n$.
\begin{enumerate}
	\item Suppose the following \emph{weak symmetry condition} holds: 
	for any sequence $\delta_n = o(1)$, there exists another sequence $\delta_n' = o(1)$ such that
	\begin{align*}
		& \max_{i\in[n]} \sum_{i'\in[n]\setminus\{i\}} \exp\left\{ \frac{-(1-\delta_n)C_k \|(U_{i, \bigcdot} - U_{i', \bigcdot})D\|^2}{\sigma_{\max}^2} \right\} \\
		\label{eq:ub_further_assumption_1}
		& \qquad \leq
		\frac{1}{n}\sum_{i\neq i'}  \exp\left\{ \frac{-(1-\delta_n')C_k \|(U_{i, \bigcdot} - U_{i', \bigcdot})D\|^2}{\sigma_{\max}^2} \right\}, 
		\qquad \forall 1\leq k \leq n. \numberthis
	\end{align*}
	Now, if for any sequence $\delta_n=o(1)$, we have
	$
		\frac{1}{n} \sum_{i\neq i'}  \exp\left\{ \frac{-(1-\delta_n)C_6\|(U_{i, \bigcdot} - U_{i', \bigcdot})D\|^2}{\sigma_{\max}^2} \right\}= o(1),
	$
	then there exists two sequences $\overline{\delta}_n, \overline{\delta}_n' = o(1)$ such that \textcolor{black}{uniformly over $\calP_n(U, D, V, \sigma_\x, \sigma_\y)$} with probability $1-o(1)$, we have
	{
		\begin{align}
		\label{eq:ub_mismatch_proportion_under_futher_assumption_1}
		\ell(\hat\Pi, \Pi^\star)  \leq \left(\frac{1}{n} \sum_{i\neq i'}  \exp\left\{ \frac{-(1-\overline{\delta}_n)C_6\|(U_{i, \bigcdot} - U_{i', \bigcdot})D\|^2}{\sigma_{\max}^2} \right\} \right)^{1-\overline{\delta}_n'}.
		\end{align}
		}
	\item Suppose in addition to the weak symmetry condition, the following \emph{strong symmetry condition} also holds:
	for any sequence $\delta_n = o(1)$, there exists another sequence $\delta_n' = o(1)$, such that
	\begin{align*}
		& \left[\sum_{i'\in[n]\setminus\{i\}} \exp\left\{ \frac{-(1-\delta_n)C_k \|(U_{i, \bigcdot} - U_{i', \bigcdot})D\|^2}{\sigma_{\max}^2} \right\}\right]^k\\
		\label{eq:ub_further_assumption_2}
		& \qquad \leq 
		\sum_{i'\in[n]\setminus\{i\}} \exp\left\{ \frac{-(1-\delta_n')kC_k \|(U_{i, \bigcdot} - U_{i', \bigcdot})D\|^2}{\sigma_{\max}^2} \right\}, \qquad \forall i\in[n], 2\leq k \leq 6.\numberthis
	\end{align*}
	Now, if for any sequence $\delta_n= o(1)$, we have
	$
		\frac{1}{n} \sum_{i\neq i'}  \exp\left\{ \frac{-(1-\delta_n)\|(U_{i, \bigcdot} - U_{i', \bigcdot})D\|^2}{4\sigma_{\max}^2} \right\}= o(1),
	$
	then there exists two sequences $\overline{\delta}_n, \overline{\delta}_n' = o(1)$ such that \textcolor{black}{uniformly over $\calP_n(U, D, V, \sigma_\x, \sigma_\y)$} with probability $1-o(1)$, we have
	{
		\begin{align}
		\label{eq:ub_mismatch_proportion_under_futher_assumption_2}
		\ell(\hat\Pi, \Pi^\star) \leq \left(\frac{1}{n} \sum_{i\neq i'}  \exp\left\{ \frac{-(1-\overline{\delta}_n)\|(U_{i, \bigcdot} - U_{i', \bigcdot})D\|^2}{4\sigma_{\max}^2} \right\} \right)^{1-\overline{\delta}_n'}.
		\end{align}
		}
\end{enumerate}
\end{corollary}
\begin{proof}
	See Appendix \ref{prf:cor:ub_mismatch_proportion_under_further_assumptions}.
\end{proof}
Under weak symmetry condition, we get an exponential error rate \eqref{eq:ub_mismatch_proportion_under_futher_assumption_1} that nearly matches the lower bound in Theorem \ref{thm:lb}, but the constant on the exponent is $C_6 = (3-2\sqrt{2})/4\approx 0.0429$, which is not sharp. Under the additional strong symmetry condition, the error rate \eqref{eq:ub_mismatch_proportion_under_futher_assumption_2} exactly matches the lower bound in Theorem \ref{thm:lb} with a sharp constant $1/4 = 0.25$ on the exponent.

The weak symmetry condition states that a certain notion of energy, as measured by
$$
	E_{i, k} := \sum_{i'\in[n]\setminus\{i\}} \exp\left\{ \frac{-(1-\delta_n)C_k \|(U_{i, \bigcdot} - U_{i', \bigcdot})D\|^2}{\sigma_{\max}^2} \right\},
$$
is spread out across all indices $i\in[n]$ for each fixed $1\leq k\leq n$. 
The strong symmetry condition further asserts that for any fixed $i$ and any fixed $2\leq k\leq 6$, the summands in $E_{i,k}$, after proper reordering, decay at a sufficiently fast rate.

\begin{example}
\label{ex:sym-cond}
To conclude this section, we consider a specific configuration of the signal matrix $UD$ and work out sufficient conditions for the weak and strong symmetry conditions. 
Let 
$$
	\beta_i^2:=  \min_{i'\in[n]\setminus\{i\}} \frac{\|(U_{i, \bigcdot} - U_{i', \bigcdot})D\|^2}{\sigma_{\max}^2}
$$
be the minimum separation between the $i$-th row and the other rows.
Suppose rows of the signal component are well-separated in the sense that there exists some $\alpha\in (0,\infty)$ such that for each fixed $i$, after a potential proper relabeling of the other rows\footnote{This relabeling can change for different fixed row index $i$. For instance, given $i$, we could relabel all other rows according to their Euclidean distances from the $i$-th row.}, we have
\begin{equation}
	\label{eq:alpha_separation}
	\frac{\|(U_{i, \bigcdot} - U_{i', \bigcdot})D\|^2}{\sigma_{\max}^2} \geq \beta_i^2 |i - i'|^{1/\alpha}.
\end{equation}	

If $\alpha \in (0, 1]$, we have
\begin{align*}
	E_{i, k} & \leq \sum_{i'\in[n]\setminus\{i\}} \exp\left\{ -(1-\delta_n)C_k \beta_i^2 |i-i'|\right\} \\
	& \leq 2\sum_{m=1}^{\infty} \exp\{-(1-\delta_n)C_k \beta_i^2 m\} 
	= \exp\left\{ -(1-\delta_n') C_k\beta_i^2\right\},
\end{align*}
where $\delta_n'=o(1)$ and the last equality is by summing over geometric series and the assumption that $\beta_i^2\geq \beta^2\gg1$. 
If $\alpha > 1$, we proceed by
\begin{align*}
	E_{i, k} & \leq \sum_{i'\in[n]\setminus\{i\}} \exp\left\{ -(1-\delta_n)C_k \beta_i^2 |i-i'|^{1/\alpha}\right\} \leq 2\sum_{m=1}^{\infty} \exp\{-(1-\delta_n)C_k \beta_i^2 m^{1/\alpha}\}
	\\
	& =  2  \left[\exp\left\{-(1-\delta_n) C_k \beta_i^2 1^{1/\alpha}\right\} + \cdots + \exp\left\{-(1-\delta_n) C_k \beta_i^2 (\lceil 2^\alpha \rceil - 1)^{1/\alpha}\right\}\right] + \\
	& \quad + 2\left[\exp\left\{-(1-\delta_n) C_k \beta_i^2 (\lceil2^\alpha\rceil)^{1/\alpha}\right\} + \cdots + \exp\left\{-(1-\delta_n) C_k \beta_i^2 (\lceil 3^\alpha \rceil - 1)^{1/\alpha}\right\}\right] + \cdots\\ 
	& \leq 2\sum_{m=1}^\infty \left(\lceil (m+1)^\alpha\rceil - \lceil m^\alpha\rceil\right) \cdot \exp\left\{-(1-\delta_n)C_k \beta_i^2 (\lceil m^\alpha \rceil)^{1/\alpha}\right\} \\
	& \leq 2 \sum_{m=1}^\infty (m+1)^\alpha \cdot \exp\left\{-(1-\delta_n)C_k \beta_i^2 m\right\} \\
	& = \sum_{m=1}^\infty \exp\left\{-(1-\delta_n)C_k \beta_i^2 m + \alpha \log(m+1) + \log 2\right\}.
\end{align*}
Since $\beta_i^2 \geq \beta^2 \gg 1$, as long as $\alpha \ll \beta^2$, we get
$$
	E_{i, k} \leq \sum_{m=1}^\infty \exp\left\{-(1-\delta_n') C_k \beta_i^2 m \right\}
	\leq \exp\left\{-(1-\delta_n'') C_k \beta_i^2 \right\},
$$
where $\delta_n', \delta_n'' = o(1)$ and the last inequality is by summing over geometric series. In summary, as long as $\alpha \ll \beta^2$, we have
\begin{equation}
	\label{eq:eik-upper-example}
	E_{i, k} \leq \exp\left\{-(1-\overline{\delta}_n) C_k \beta_i^2\right\}	
\end{equation}
for some $\overline{\delta}_{n} = o(1)$.
On the other hand, we have from the definition of $\beta_i^2$ that
\begin{equation}
	\label{eq:eik-lower-example}
	E_{i, k} \geq \exp\left\{-(1-\delta_n)C_k \beta_i^2\right\}.
\end{equation}
Thus, the weak symmetry condition would hold if
\begin{equation}
	\label{eq:weak_symmetry_under_alpha_separation}
	\max_{i\in[n]}\exp\{-(1-\overline{\delta}_{n})C_k \beta_i^2\} \leq \frac{1}{n} \sum_{i\in[n]} \exp\{-(1-\overline{\delta}_{n}')C_k \beta_i^2\}
\end{equation}	
for some $\overline{\delta}_{n}'=o(1)$.
In particular, \eqref{eq:weak_symmetry_under_alpha_separation} holds when there is a positive fraction of rows with $\beta_i = \beta$. 
Fix a proportion $a\in (0,1]$, this is achievable by using the coordinates of all elements in the intersection of the integer lattice in $\mathbb{R}^r$ with an $r$-dimensional ball with radius $c_r(an)^{1/r}$ as the first $an$ rows of a matrix $L\in \mathbb{R}^{n\times r}$. 
Here $c_r$ is a constant depending only on $r$.
We could fill the remaining rows of $L$ sequentially under the constraint that all elements are integers and that the Euclidean distance of a new row is at least $1$ away from all existing rows.
Finally, we take $U$ and $D$ as the left singular vectors and singular values of $\beta \sigma_{\max} L$.
In this case, $\beta_i = \beta$ for $1\leq i\leq an$ and $\beta_i\geq \beta$ for all other rows.
In addition, \eqref{eq:alpha_separation} holds with $\alpha \asymp r$ for this particular configuration.

On the other hand, \eqref{eq:eik-upper-example} and \eqref{eq:eik-lower-example} imply that the strong symmetry condition would hold if for all $i$ and for $2\leq k\leq 6$,
$$
	\left(\exp\{-(1-\overline{\delta}_{n})C_k \beta_i^2\}\right)^k \leq \exp\{-(1-\overline{\delta}_{n}'')kC_k \beta_i^2\},
$$
for some $\overline{\delta}_{n}''=o(1)$, which is trivially true.
In summary, under the condition in \eqref{eq:alpha_separation}, the weak and strong symmetry conditions both hold as long as $\beta^2\gg \alpha \lor 1$ and \eqref{eq:weak_symmetry_under_alpha_separation} holds.
\end{example}

\section{Simulation}\label{sec:sim}

\subsection{Effects of signal strength}

\begin{figure}[t!]
\centering
	\begin{subfigure}{.45\textwidth}
		\centering
		\includegraphics[width=1\linewidth]{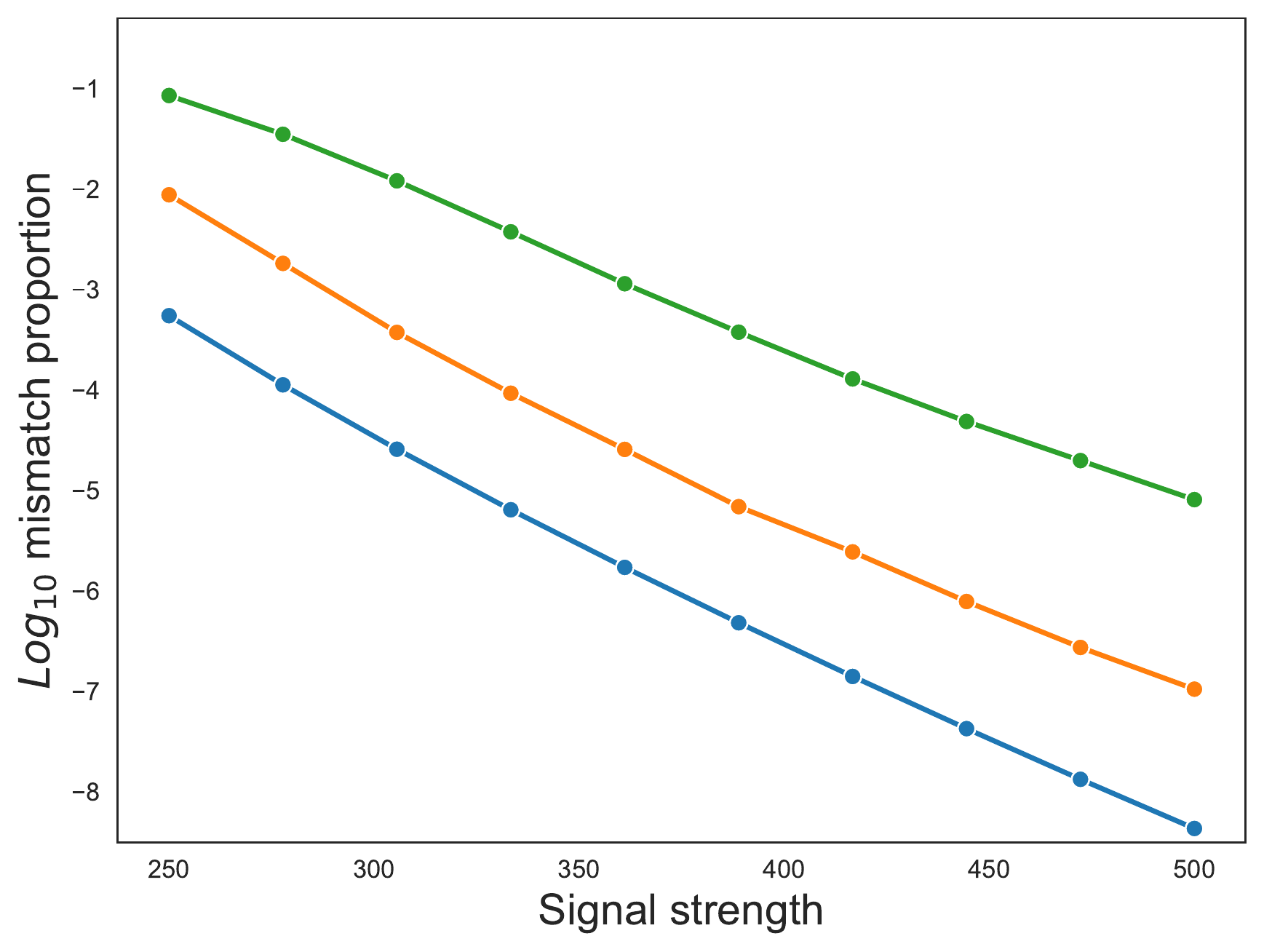}
	\end{subfigure}
	\begin{subfigure}{.45\textwidth}
		\centering 
		\includegraphics[width=1\linewidth]{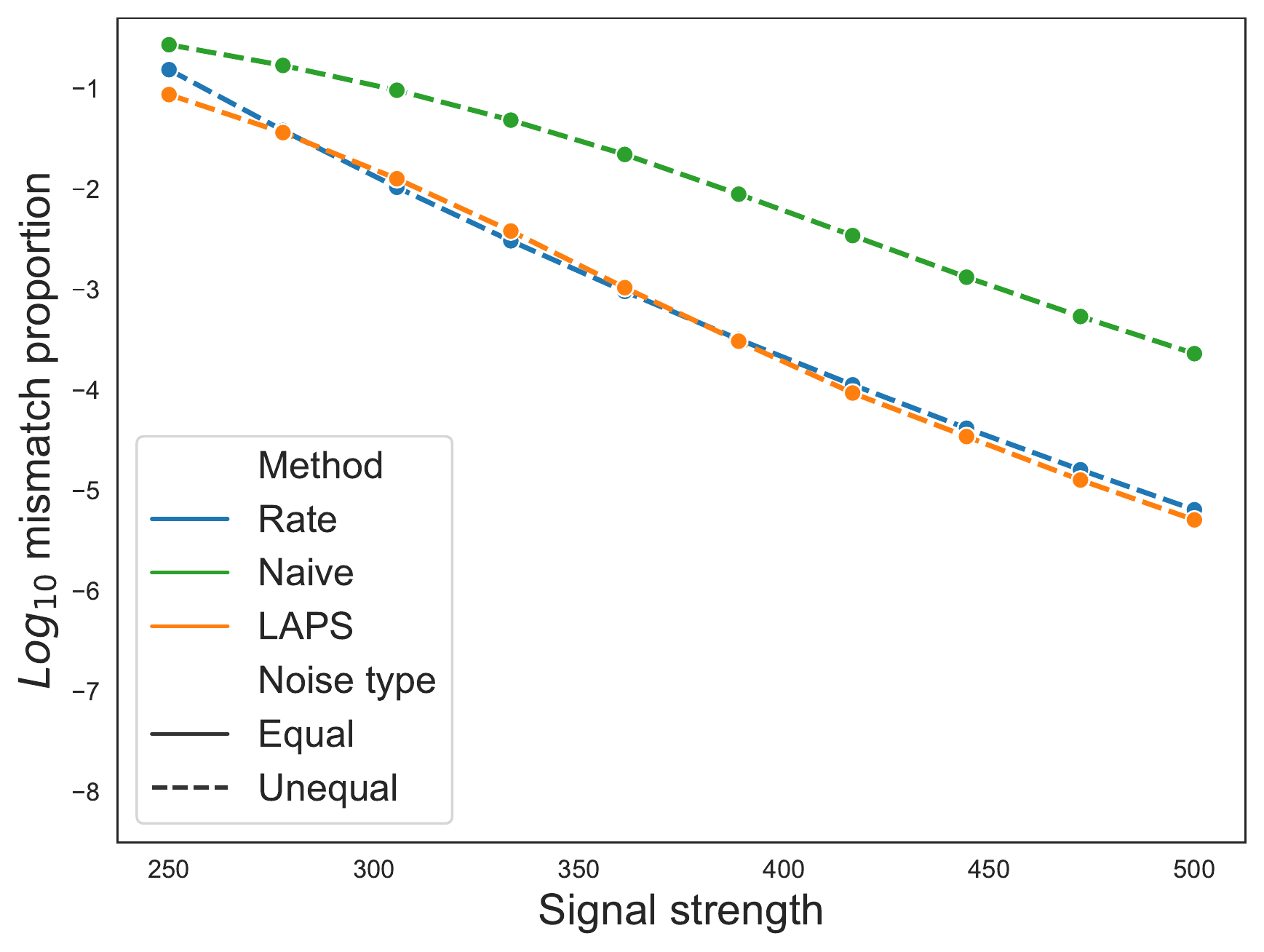}
	\end{subfigure}
       \caption{\small Average mismatch proportions {over 1000 repetitions} (in $\log_{10}$ scale) of LAPS, naive linear assignment without projection, and the theoretical prediction (i.e., the minimax rate) as the signal strength varies.}
	\label{fig:vary_d}
\end{figure}

To start with, we present a simulation study that examines the effect of signal strength, as measured by the magnitude of $(d_1, \hdots, d_r)$. Here, we set $n = 1000, p=50, r=10$.
We generate Haar distributed random orthogonal matrices $U\in O_{n,r}$ and $V\in O_{p, r}$. 
{The ground truth permutation $\Pi^\star$ is randomly sampled from $S_n$.} 
Those parameters are generated once and then fixed throughout the simulation.
The diagonal entries of $D$ are generated by first sampling a random vector $w \in \bbR^{r}$ with i.i.d.~$\textnormal{unif}(0, 1)$ entries and let $\diag(D) = \texttt{signal} \times w$, where $\texttt{signal}$ is a scalar and we vary it in $\{250(1 + i/9): 0\leq i\leq 9\}$.
We either let $\sigma_\x = \sigma_\y = 1$ or $\sigma_\x =1, \sigma_\y = 1.5$. 
The former is called ``equal noise'' case and the latter is called ``unequal noise'' case. 
For each configuration of $\texttt{signal}$ and $(\sigma_\x, \sigma_\y)$, we generate the datasets $X$ and $Y$ according to model \eqref{eq:model} for $1000$ times, and we record the {mismatch proportions \eqref{eq:loss} averaged over 1000 simulations} for both LAPS and the naive algorithm that solves linear assignment on the raw data without projection. 

Figure \ref{fig:vary_d} plots the $\log_{10}$-transformed average mismatch proportion versus the value of $\texttt{signal}$. 
To verify our theory, we also plot the minimax rate given in \eqref{eq:ub_mismatch_proportion_under_futher_assumption_2} with $o(1)$ terms omitted. 
Figure \ref{fig:vary_d} shows that LAPS uniformly outperforms its naive counterpart regardless of signal strength and noise configuration. 
When the noise levels become unequal, both methods perform worse. 
The minimax rate \eqref{eq:ub_mismatch_proportion_under_futher_assumption_2} aligns reasonably well with the empirical error made by LAPS.

\subsection{Effects of the projection step}

\begin{figure}[t] 
\centering
	\begin{subfigure}{.6\textwidth}
		\centering
		\includegraphics[width=1\linewidth]{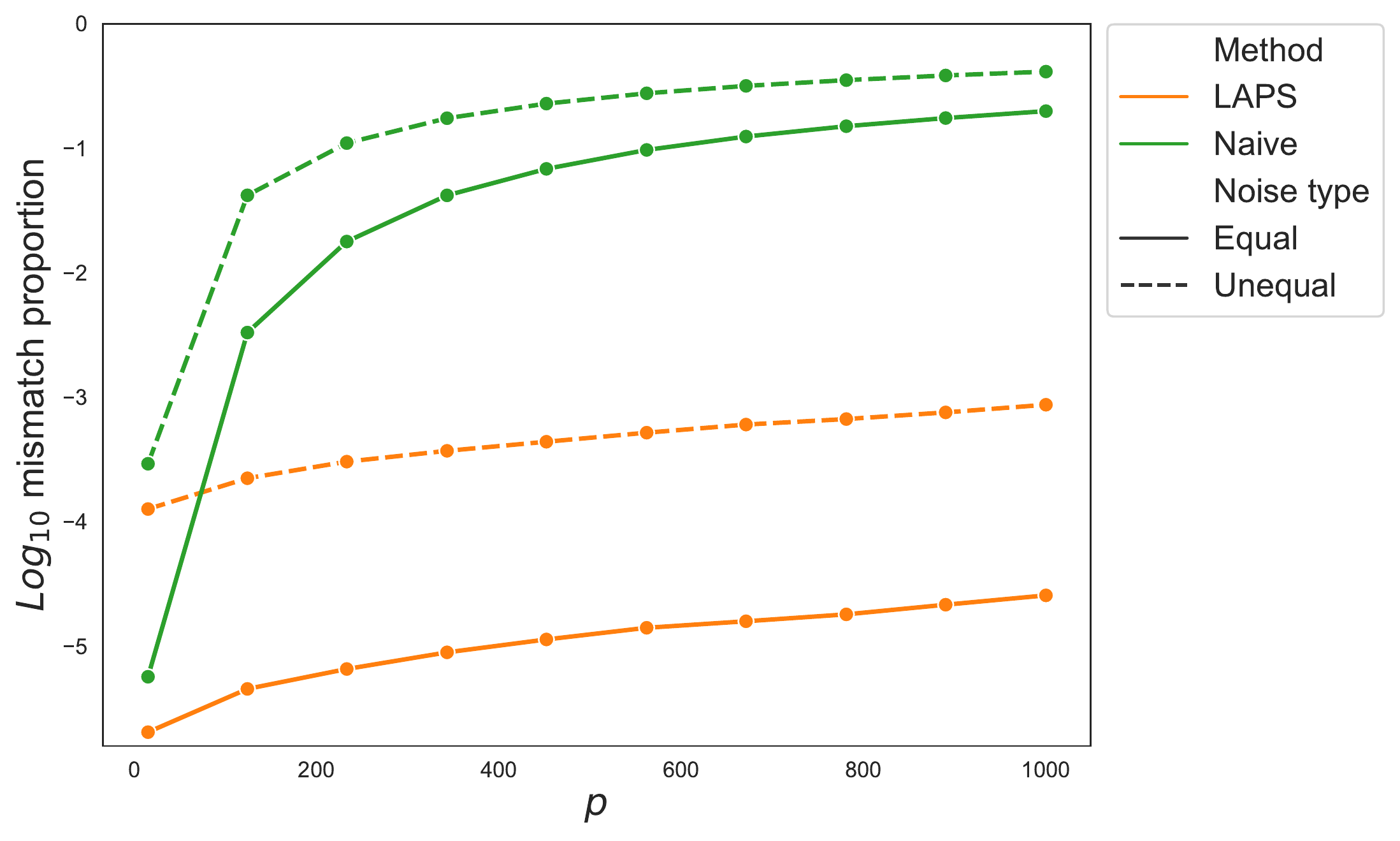}
	\end{subfigure}
       \caption{\small Average mismatch proportions {over 1000 repetitions} (in $\log_{10}$ scale) of LAPS and naive linear assignment without projection as $p$ varies.} 
	\label{fig:vary_r_and_p}
\end{figure}

We now proceed to examining the effectiveness of SVD-based projection.
We consider a similar data-generating process as in the previous simulation, but with $\texttt{signal}$ fixed at $400$. 
We fix $n=1000, r = 10$ and vary $p$ from $15$ to $1000$. Note that since $UD$ remains constant, the minimax rate remains unchanged, and the only factor that can potentially affect performance is that uncovering the low-dimensional signals becomes harder as $p$ grows.

The results are given in Figure \ref{fig:vary_r_and_p}. We again see that LAPS outperforms its naive counterpart in all scenarios considered. 
The naive method performs worse as $p$ grows, which is intuitive as the signal-to-noise ratio decays as $p$ grows.
The performance of LAPS also degrades as $p$ grows, as estimating $V$ becomes harder. However, the drop in accuracy is not as much compared to the naive method, illustrating the effectiveness of the projection step.

\section{Real data examples from single-cell biology}
\label{sec:data}

\subsection{Matching single-cell RNA-seq data} 
\label{subsec:rna}
We apply LAPS to integrate two single-cell RNA-seq datasets collected from human pancreatic islets across different technologies. 
The first dataset first appeared in \cite{grun2016novo} and was obtained using the CEL-seq2 technology \cite{hashimshony2016cel}. 
The second dataset appeared in \cite{segerstolpe2016single} and was measured using the Smart-seq2 technology \cite{picelli2013smart}. 
The raw CEL-seq2 data contain measurements on $34363$ RNAs in $2285$ cells, and the raw Smart-seq2 data contain measurements on $34363$ RNAs in $2394$ cells. 
The RNAs measured in the two datasets only partially overlap though the total number of features are identical. Human annotations of cell types are available for both datasets. 

We apply standard pre-processing pipelines provided by Python package \texttt{scanpy} \cite{wolf2018scanpy} to select top $5000$ active RNAs for both datasets. 
To ensure the datasets fit into our current model, we manually balance the two datasets as follows: for each cell type, we randomly down-sample cells of that type in one dataset, so that the numbers of cells of that type are the same in both datasets. 
After balancing, we get two data matrices $X_\act, Y_\act\in \bbR^{1935\times 5000}$ for CEL-seq2 data and Smart-seq2 data, respectively. The cell type composition is shown in the left panel of Figure \ref{fig:rna_celltype_freq_and_acc}.
Among all the active RNAs ($5000$ in each dataset), $2508$ appeared in both datasets, thus giving two feature-wise aligned data matrices $X_\sha, Y_\sha \in \bbR^{1935\times 2508}$. 
We then apply LAPS to the pair $(X_\sha, Y_\sha)$ with $r=30$ and $\hat{V}$ estimated from $X_\sha$.

\begin{figure}[t]
\centering
	\begin{subfigure}[t]{.48\textwidth}
		\vspace{0pt}
		\centering
		\includegraphics[width=1\linewidth]{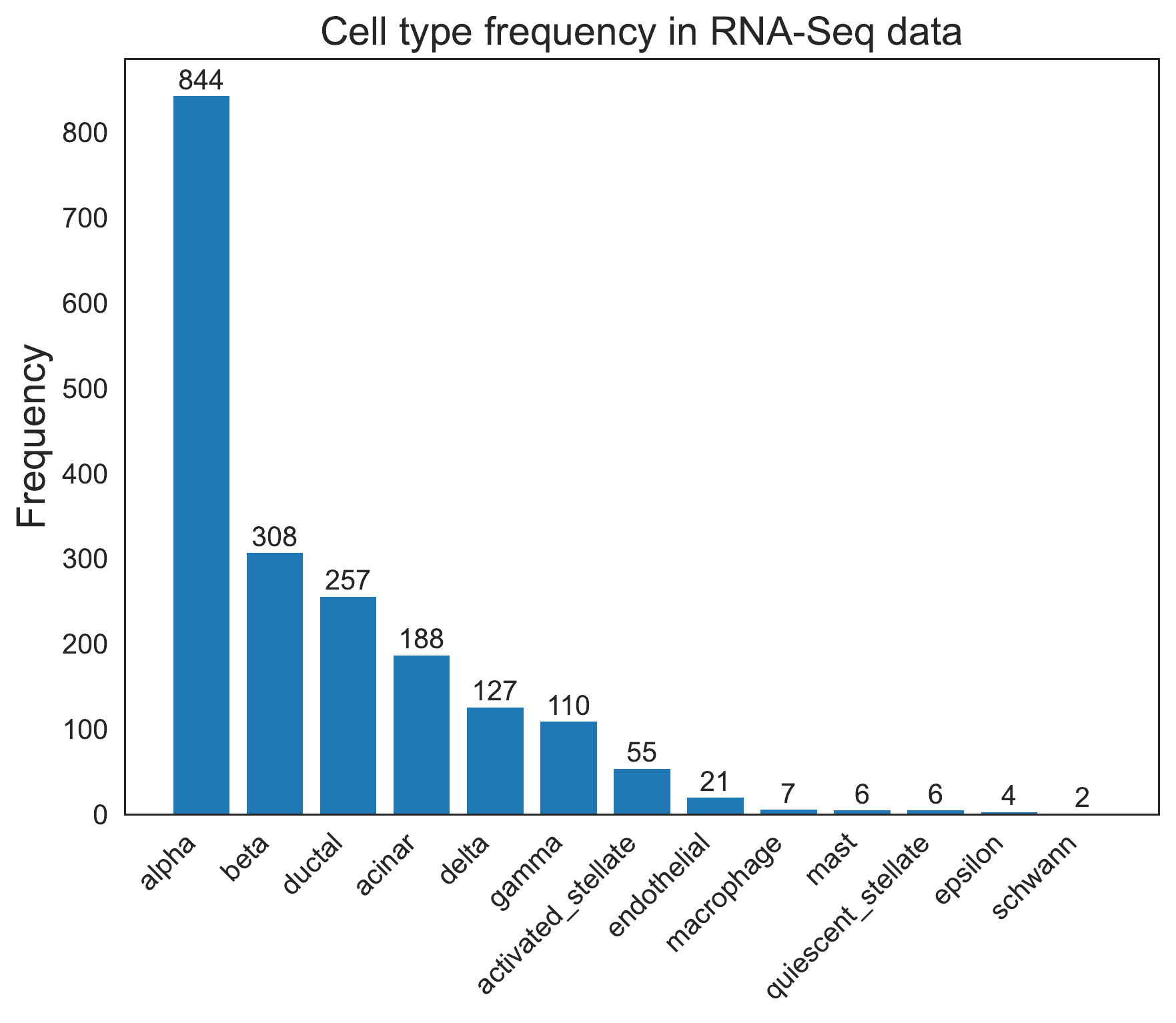}
	\end{subfigure}
	\begin{subfigure}[t]{.48\textwidth}
		\vspace{0pt}
		\centering 
		\includegraphics[width=1.04\linewidth]{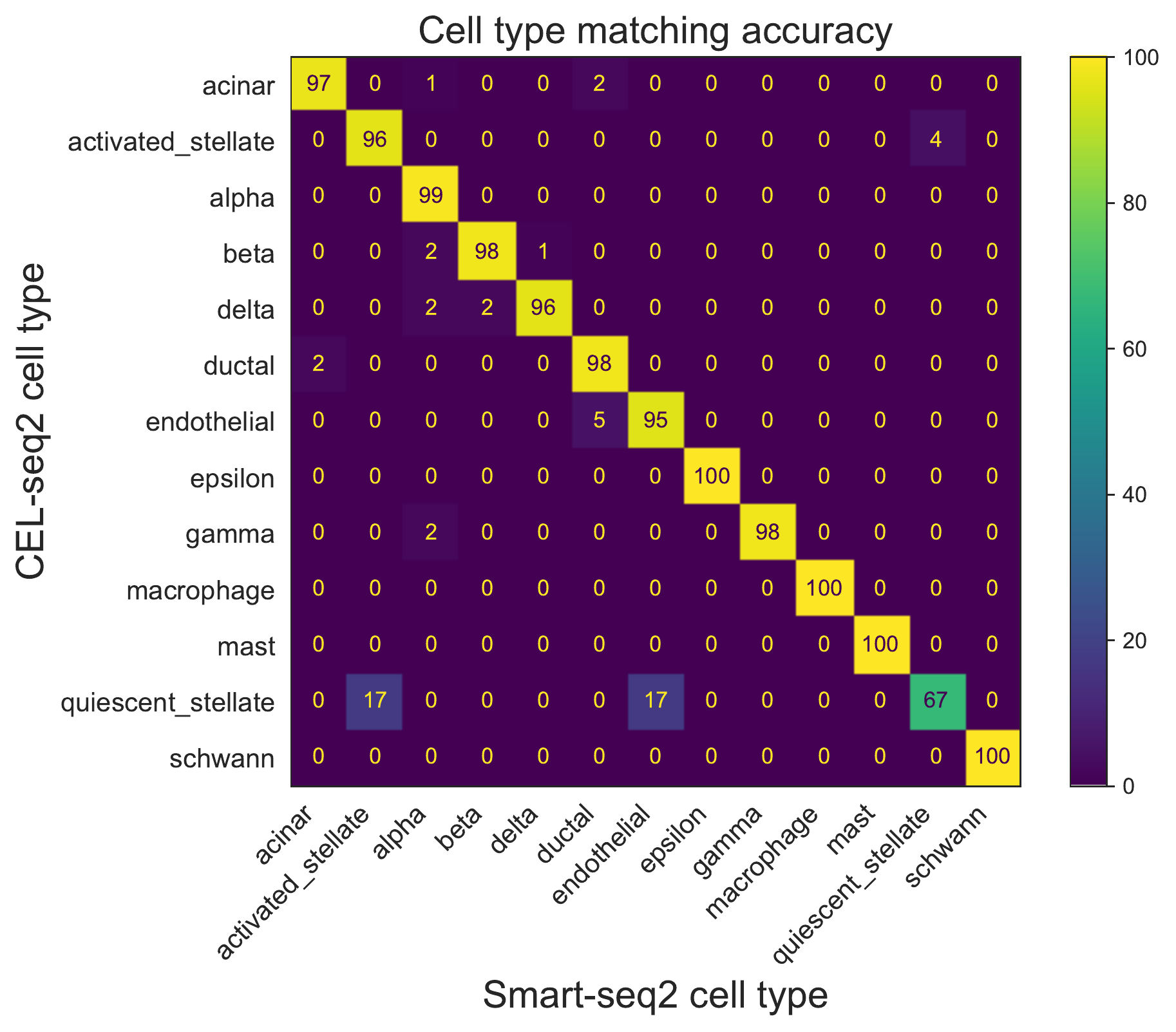}
	\end{subfigure}
       \caption{\small The cell type composition of single-cell RNA-seq data (left panel) and cell-type level matching accuracy.}
	\label{fig:rna_celltype_freq_and_acc}
\end{figure}

Since there is no ground truth matching available, we evaluate the performance of LAPS by computing the cell type level matching accuracy, i.e., we claim $\hat \pi_i$ is correct if $X_{i,\bigcdot}$ and $Y_{\hat\pi_i, \bigcdot}$ are of the same cell type. LAPS achieves $97.93\%$ overall accuracy and the right-panel of Figure \ref{fig:rna_celltype_freq_and_acc} displays the confusion matrix, from which we see that LAPS achieves high accuracy even for infrequent cell types.

In order to project the two datasets into a common subspace, we fit canonical correlation analysis (CCA) on $(X_\act, \hat \Pi Y_\act)$, obtain top $30$ CCA scores, and plot the two-dimensional UMAP embeddings in {the bottom two panels of} Figure \ref{fig:rna_umap}. 
Note that fitting CCA on {all active features (5000 in either dataset) as opposed to the shared features (2508 shared between two datasets) retains more biological information useful for downstream analyses. 
{{For} comparison, the top two panels display the UMAP embeddings obtained from $(X_\sha\hat V, Y_\sha \hat V)$.}
From the left two panels of Figure \ref{fig:rna_umap}, {we see the embeddings {of} two datasets are better mixed after LAPS-based integration}, illustrating successful correction of technological differences between CEL-seq2 and Smart-seq2. 
The right two panels of Figure \ref{fig:rna_umap} {shows different cell types are better separated after LAPS-based integration}, which indicates that biological signals are preserved.

\begin{figure}[t]
\centering
	\begin{subfigure}[t]{.49\textwidth}
		\vspace{0pt}
		\centering
		\includegraphics[width=1\linewidth]{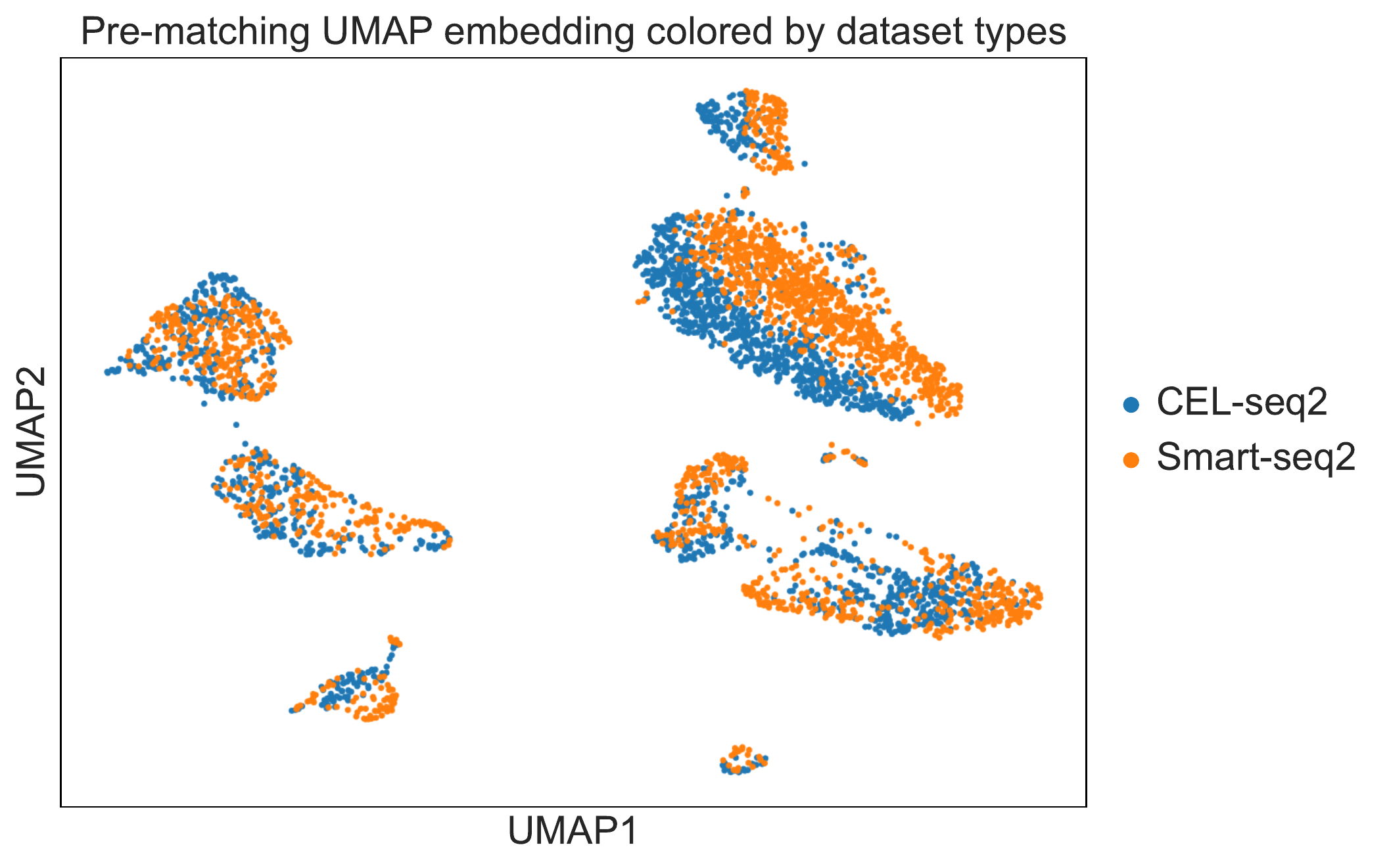}
	\end{subfigure}
	\begin{subfigure}[t]{.49\textwidth}
		\vspace{0pt}
		\centering 
		\includegraphics[width=1.09\linewidth]{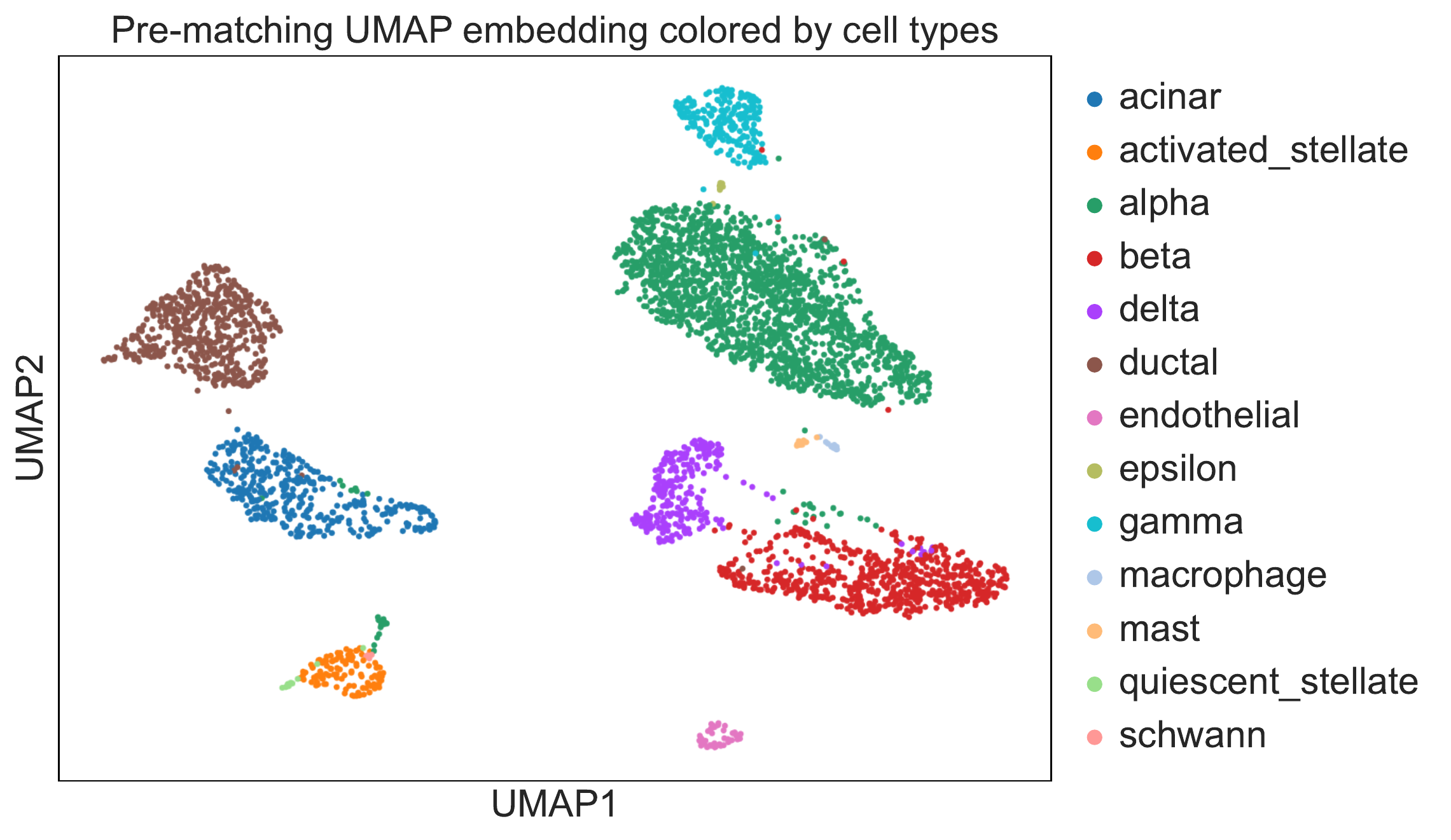}
	\end{subfigure}
	\begin{subfigure}[t]{.49\textwidth}
		\vspace{0pt}
		\centering
		\includegraphics[width=1\linewidth]{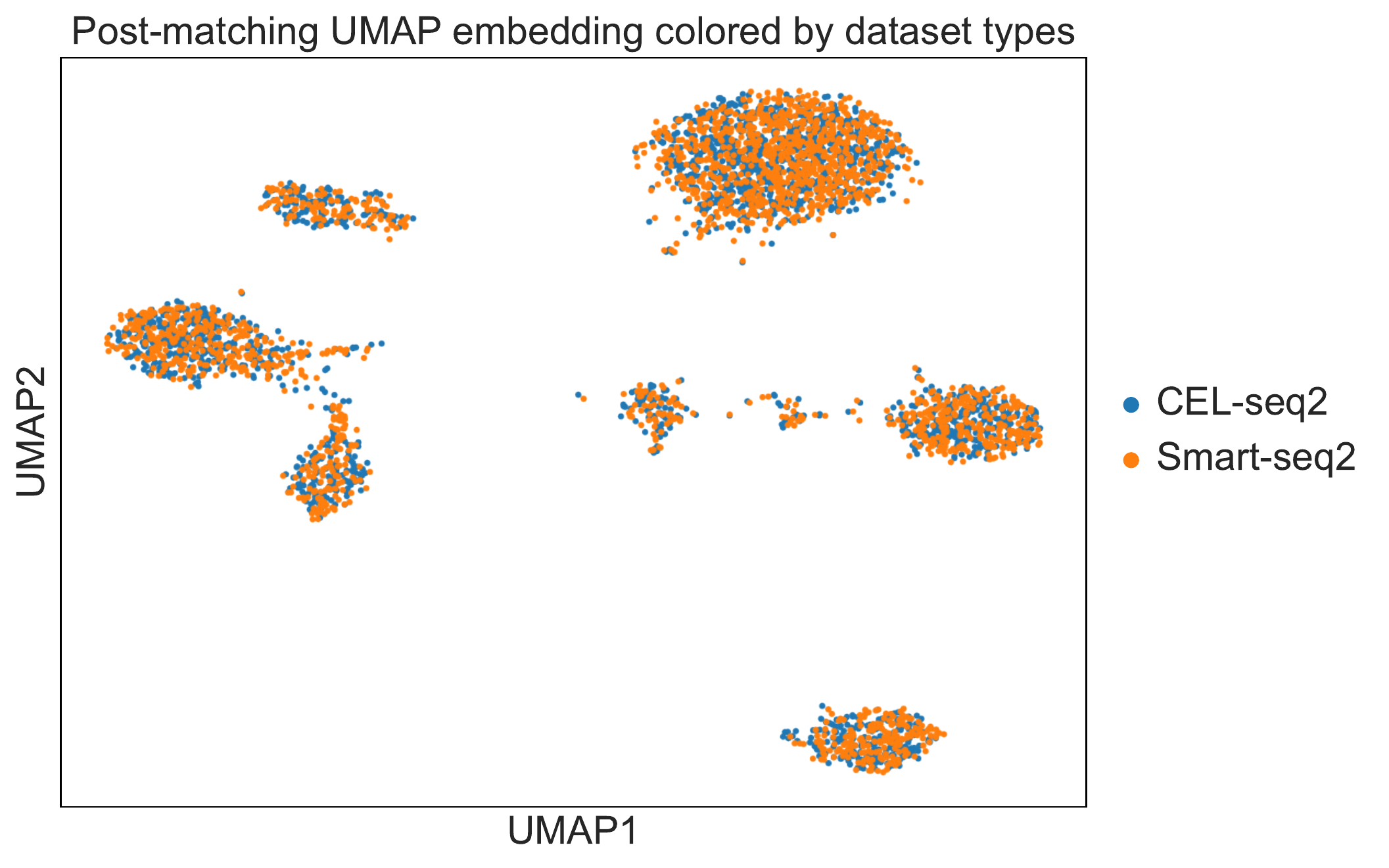}
	\end{subfigure}
	\begin{subfigure}[t]{.49\textwidth}
		\vspace{0pt}
		\centering 
		\includegraphics[width=1.09\linewidth]{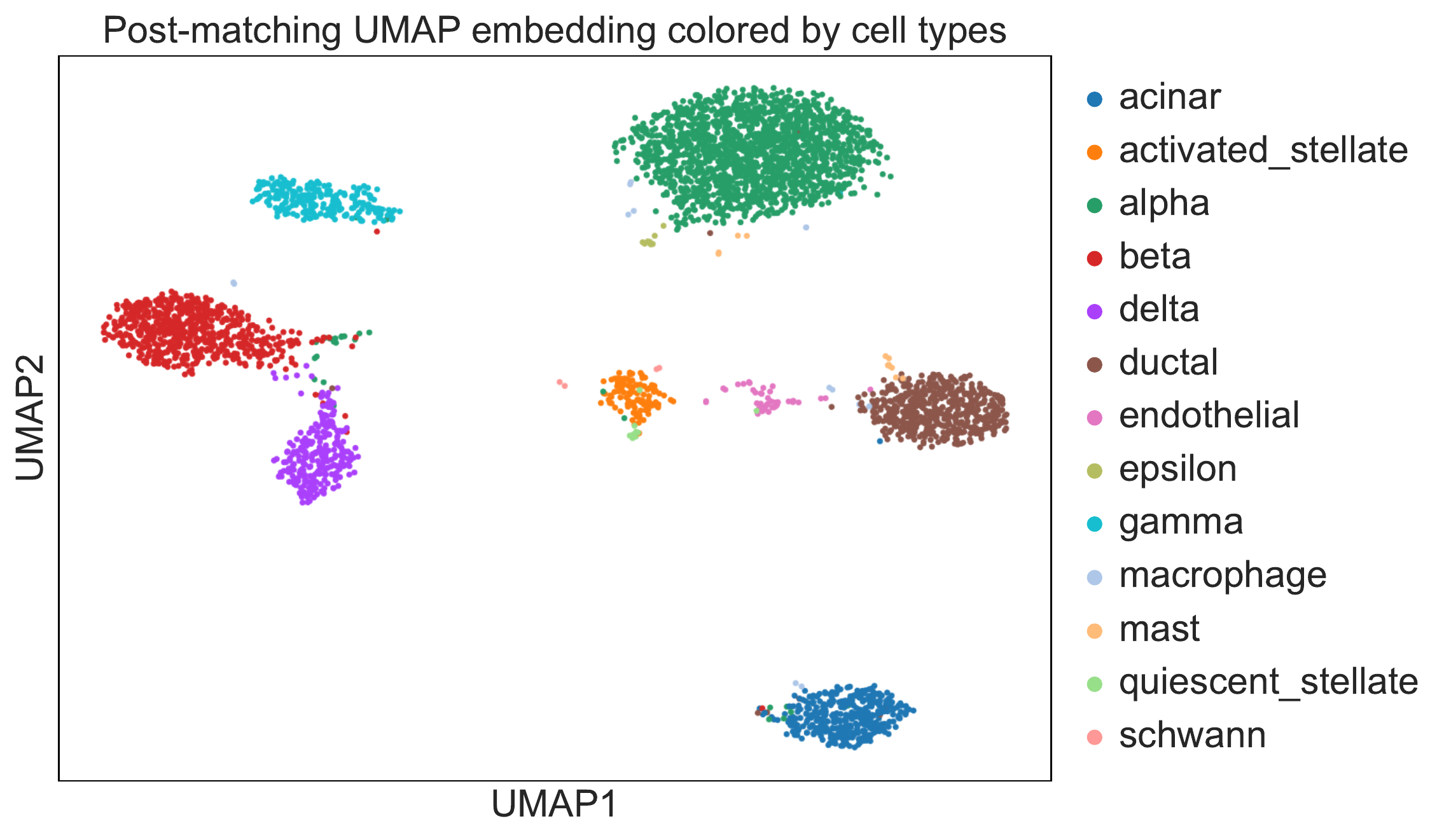}
	\end{subfigure}
       \caption{\small Two-dimensional UMAP embeddings of CEL-seq2 and Smart-seq2 data, colored by dataset type (left two panels) and cell type (right two panels). {The top two panels represent UMAP embeddings obtained from shared features (which do not rely on matching), whereas the bottom two panels represent UMAP embeddings obtained from CCA embeddings of all active features (which rely on matching by LAPS).}}
	\label{fig:rna_umap}
\end{figure}

\subsection{Matching spatial and non-spatial proteomics datasets}
\label{sec:proteomics}

We now apply LAPS to match two single-cell proteomics datasets. 
The first dataset appeared in \cite{gayoso2021joint} and was obtained from murine spleen using CITE-seq technology \cite{stoeckius2017simultaneous}. 
The second dataset appeared in \cite{zhu2021robust} and was obtained from BALBc murine spleen with CODEX multiplexed imaging technology (a.k.a.~PhenoCycler) \cite{goltsev2018deep}.
The raw CITE-seq data measures expression levels of $208$ proteins in $15202$ cells, and the raw CODEX data contains measurements of expression levels of $31$ proteins in $48332$ cells. 
In addition to protein measurements, CITE-seq is capable of measuring RNA expression levels in the same collection of cells, and CODEX provides spatial coordinates of the same collection of cells on the slice of tissue cut by the experimenter. 
If one can accurately match CITE-seq data and CODEX data using their protein measurements, then one can transfer the spatial information in CODEX data to the CITE-seq data and examine the spatial distributions of RNA expression levels. 
While spatial transcriptomics technology is only recently emerging and remains very costly \cite{marx2021method}, matching and transferring information between CITE-seq and CODEX datasets provides an economical alternative to investigating the spatial patterns of RNA expression levels and to making use of the rich archive of non-spatial single-cell transcriptomics data collected over the years.

Since signal-to-noise ratios in matching proteomics data are significantly lower than those in matching transcriptomics data, 
we first perform several additional pre-processing steps. 
For CITE-seq data, we first calculate the median size of six cell type clusters, and down-sample the data such that the size of all cell type clusters are no greater than the median size. 
After down-sampling, we get a data matrix of size $5091\times 208$.
We then aggregate the data into ``meta-cells'' by clustering the data into $2545\approx 5091/2$ clusters (using \texttt{scanpy} clustering pipeline) and taking the cluster centroids as new data. The cell type of each meta-cell is determined by majority voting.
Thus, we obtain a data matrix of dimension $2545\times 208$. 
We apply similar pre-processing steps on CODEX data, except that we aggregate it into meta-cells as averages of around four cells each. 
This gives a data matrix of dimension $6014\times 31$. 
Similar to what we have done for CEL-seq2 and Smart-seq2 data in Section \ref{subsec:rna}, we then down-sample the two meta-cell data matrices to balance their cell type composition, and the result is shown in the left panel of Figure \ref{fig:protein_celltype_freq_and_acc}. 
The balanced data matrices are denoted as $X_\act \in \bbR^{2545\times 208}, Y_\act \in \bbR^{2545\times 31}$ for CITE-seq and CODEX, respectively. We take the proteins that appear in both datasets and form two feature-wise aligned matrices $X_\sha, Y_\sha \in \bbR^{2545\times 28}$. 

\begin{figure}[t]
\centering
	\begin{subfigure}[t]{.48\textwidth}
		\vspace{0pt}
		\centering
		\includegraphics[width=1\linewidth]{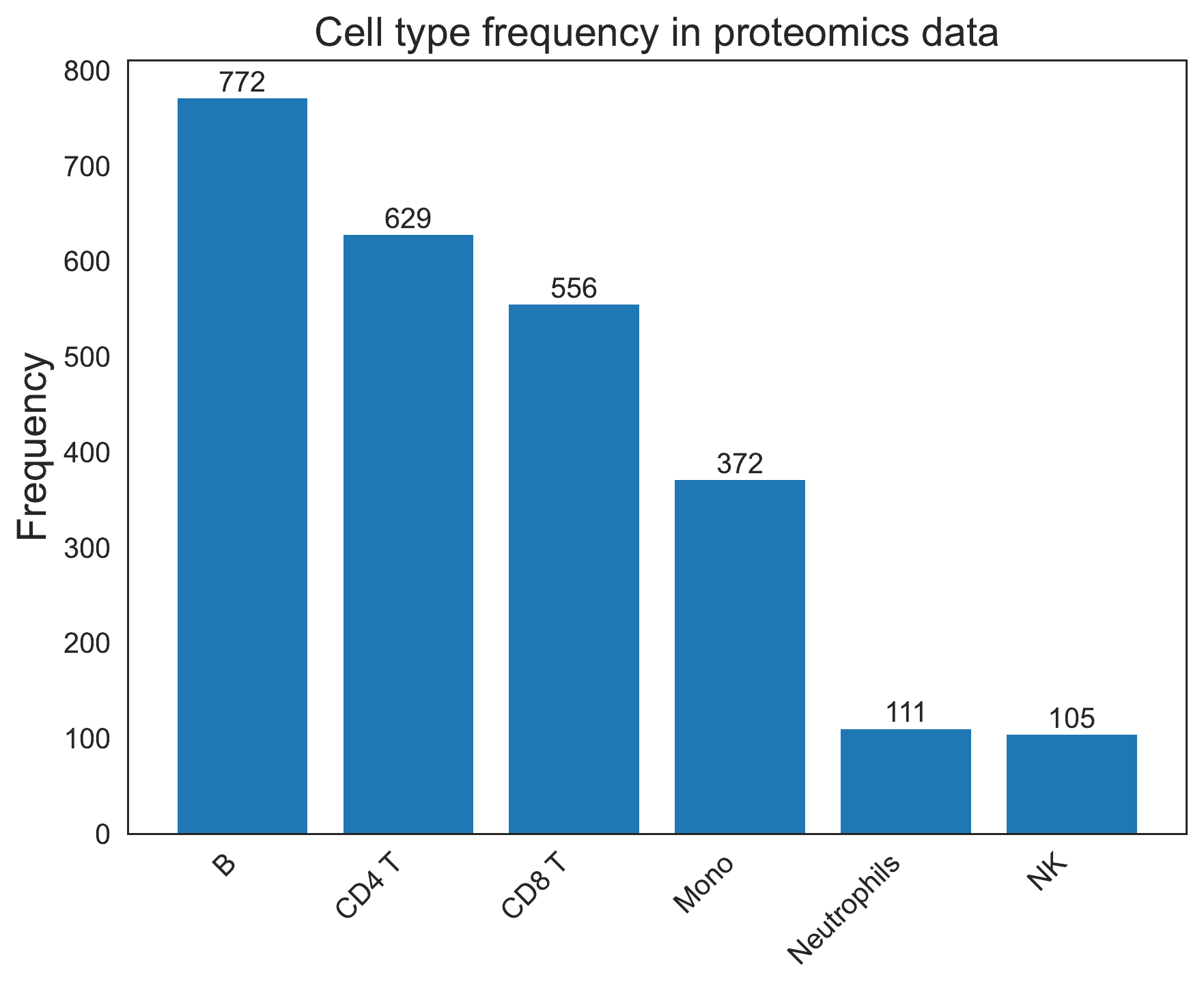}
	\end{subfigure}
	\begin{subfigure}[t]{.48\textwidth}
		\vspace{0pt}
		\centering 
		\includegraphics[width=0.96\linewidth]{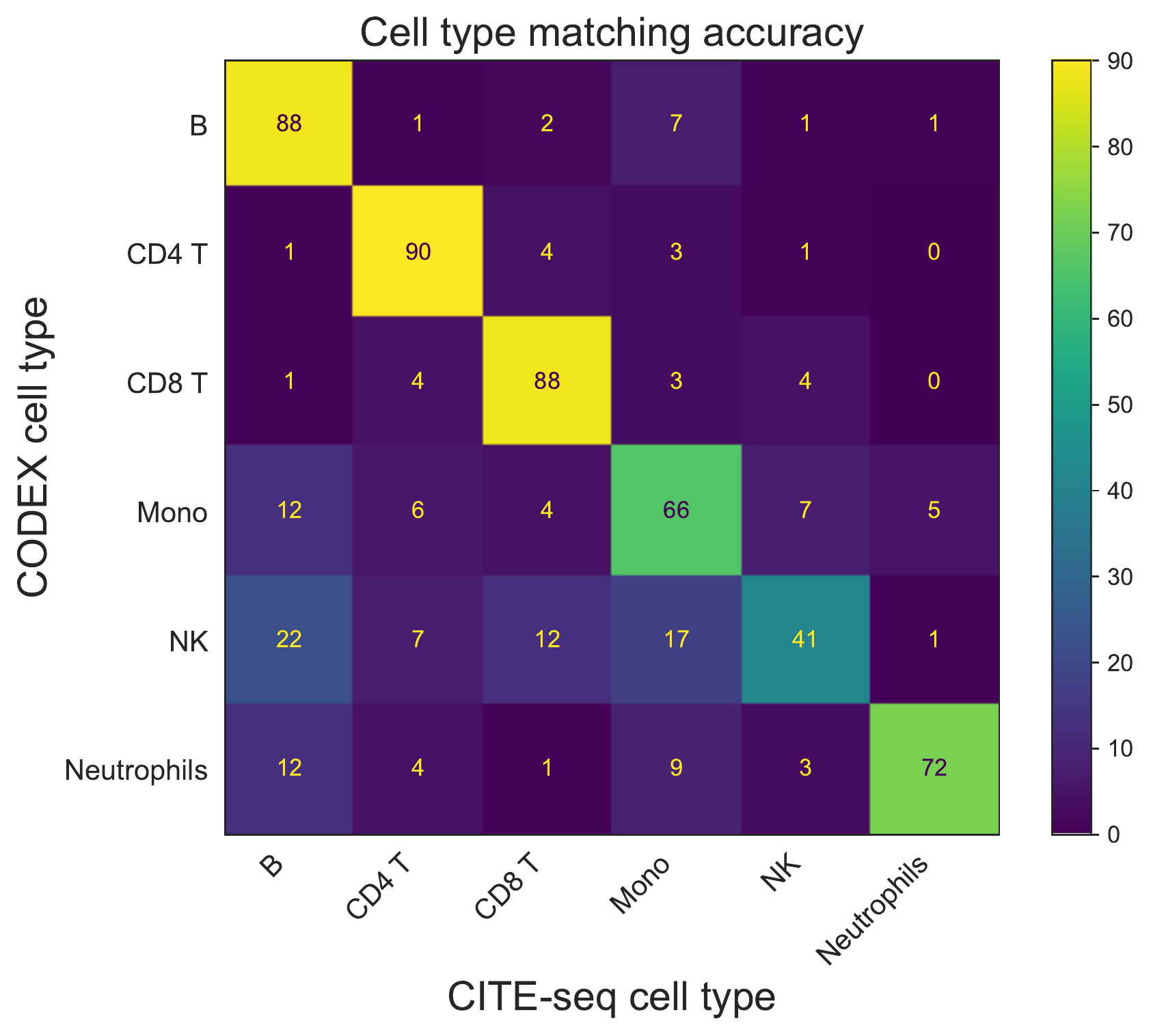}
	\end{subfigure}
       \caption{\small The cell type composition of single-cell RNA-seq data (left panel) and cell-type level matching accuracy.}
	\label{fig:protein_celltype_freq_and_acc}
\end{figure}

The rest of the analysis is similar to that appeared in Section \ref{subsec:rna}. 
LAPS is applied to this pair of matrices with $r=15$ and $\hat V$ estimated on $X_\sha$. 
The overall cell type level matching accuracy is $84.27\%$, and the confusion matrix is shown in the right panel of Figure \ref{fig:protein_celltype_freq_and_acc}. The accuracy is lower than that in RNA-seq matching, especially for minority cell types. 
This is expected, as the signal-to-noise ratio is lower. We refer the readers to the companion paper \cite{zhu2021robust} for a full pipeline with additional post-processing steps that establishes state-of-the-art performance with over $90\%$ accuracy in cell type correspondence.

\begin{figure}[t]
\centering
	\begin{subfigure}[t]{.49\textwidth}
		\vspace{0pt}
		\centering
		\includegraphics[width=1\linewidth]{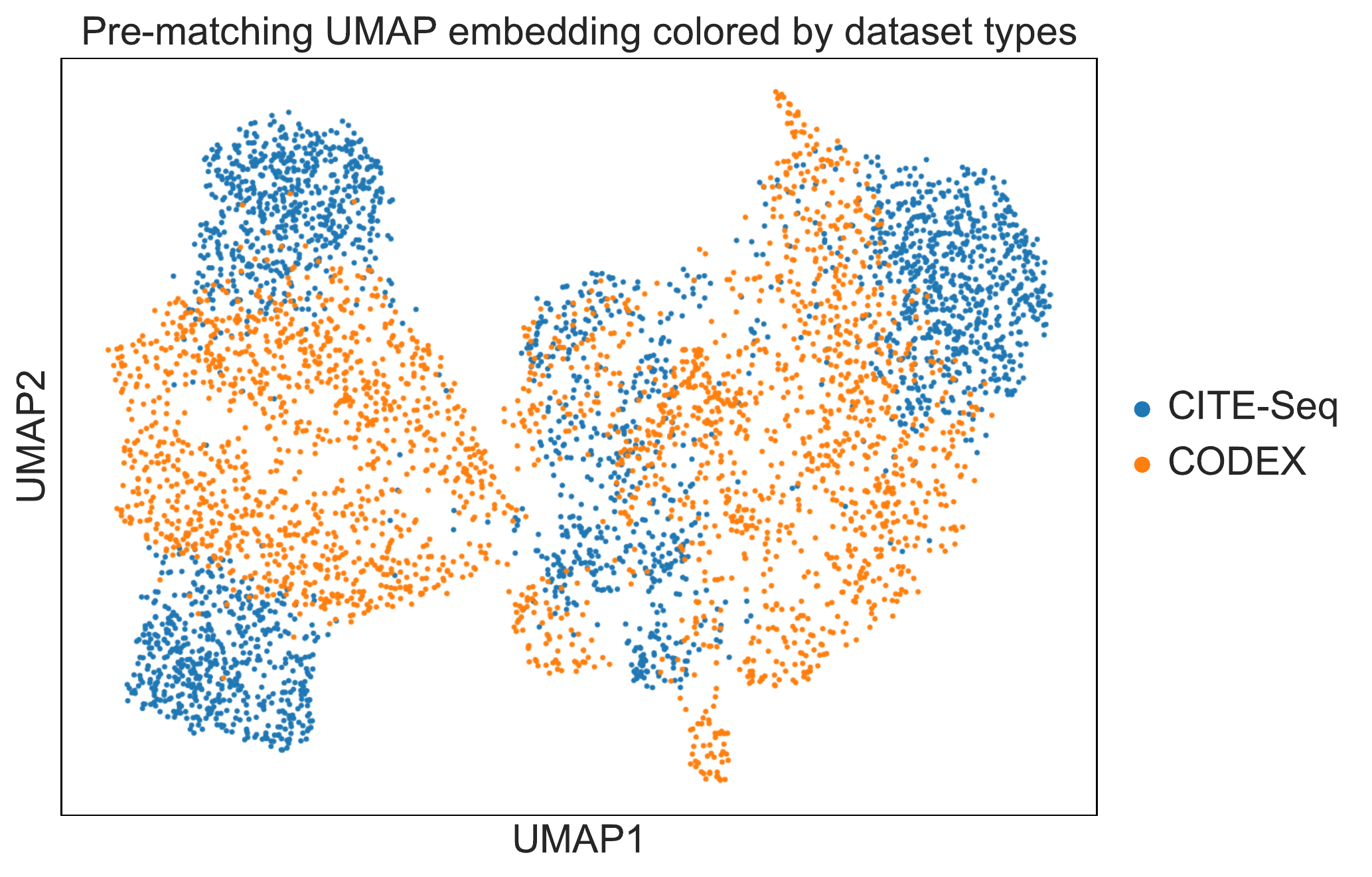}
	\end{subfigure}
	\begin{subfigure}[t]{.49\textwidth}
		\vspace{0pt}
		\centering 
		\includegraphics[width=1.045\linewidth]{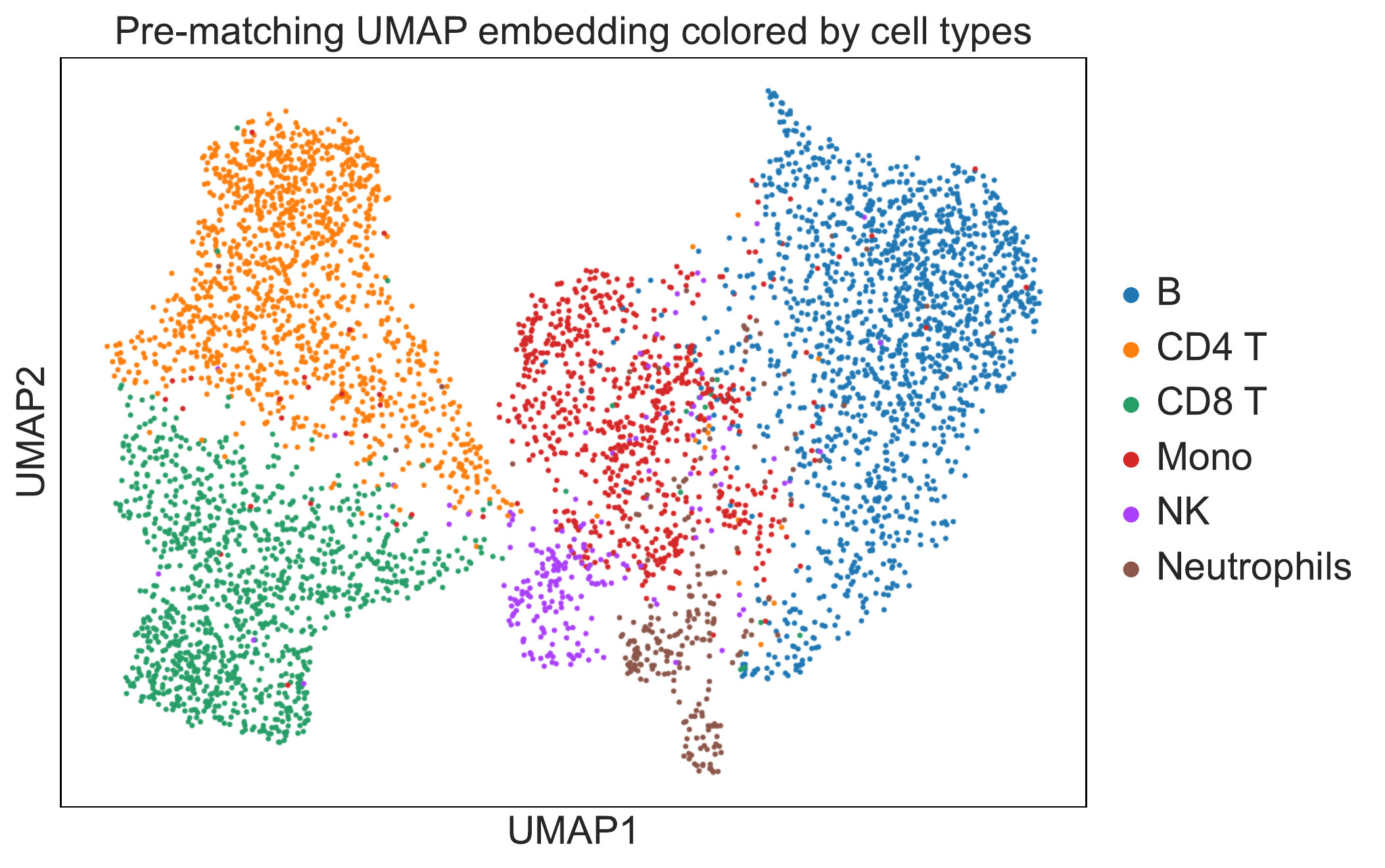}
	\end{subfigure}
	\begin{subfigure}[t]{.49\textwidth}
		\vspace{0pt}
		\centering
		\includegraphics[width=1\linewidth]{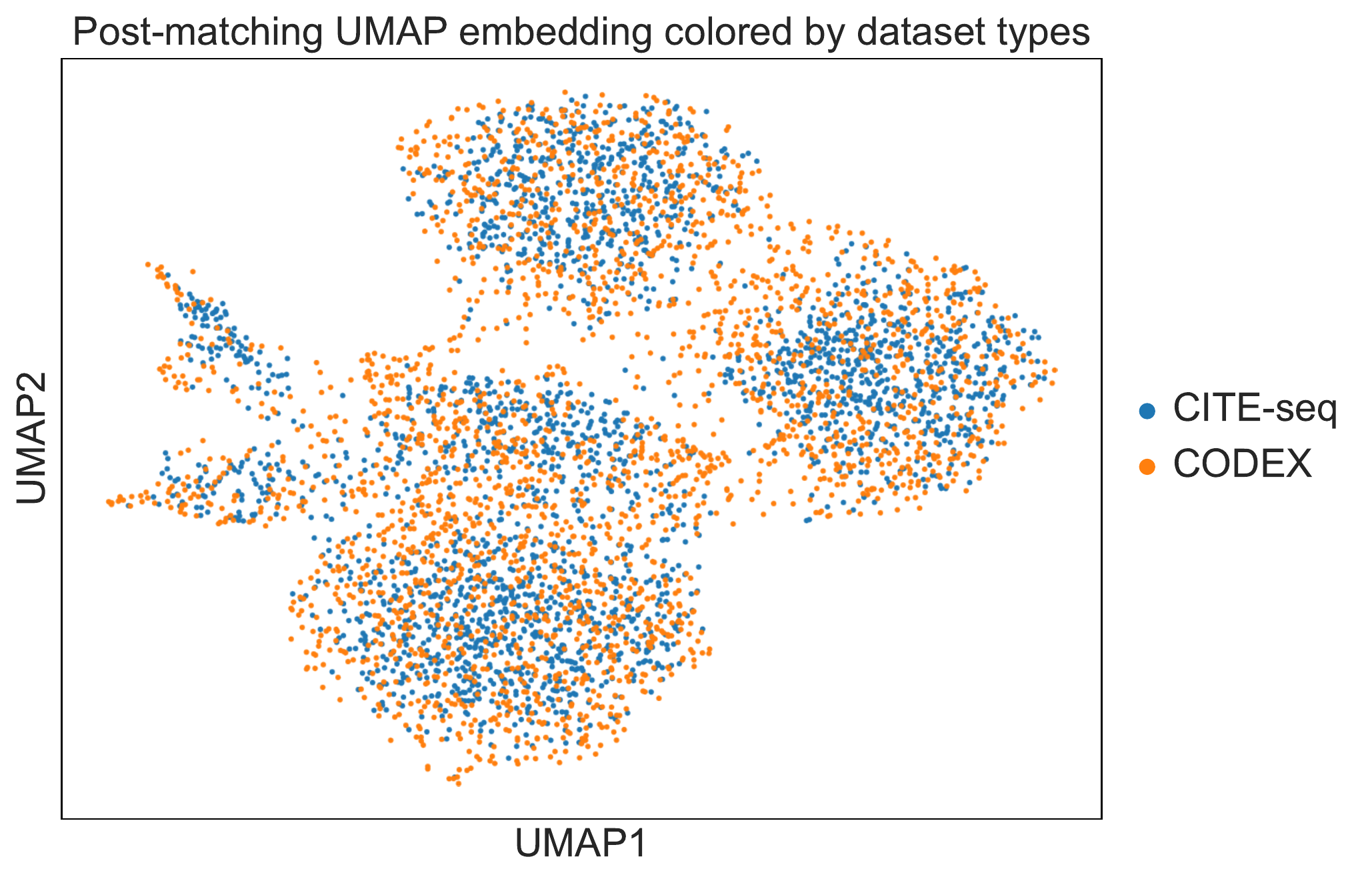}
	\end{subfigure}
	\begin{subfigure}[t]{.49\textwidth}
		\vspace{0pt}
		\centering 
		\includegraphics[width=1.045\linewidth]{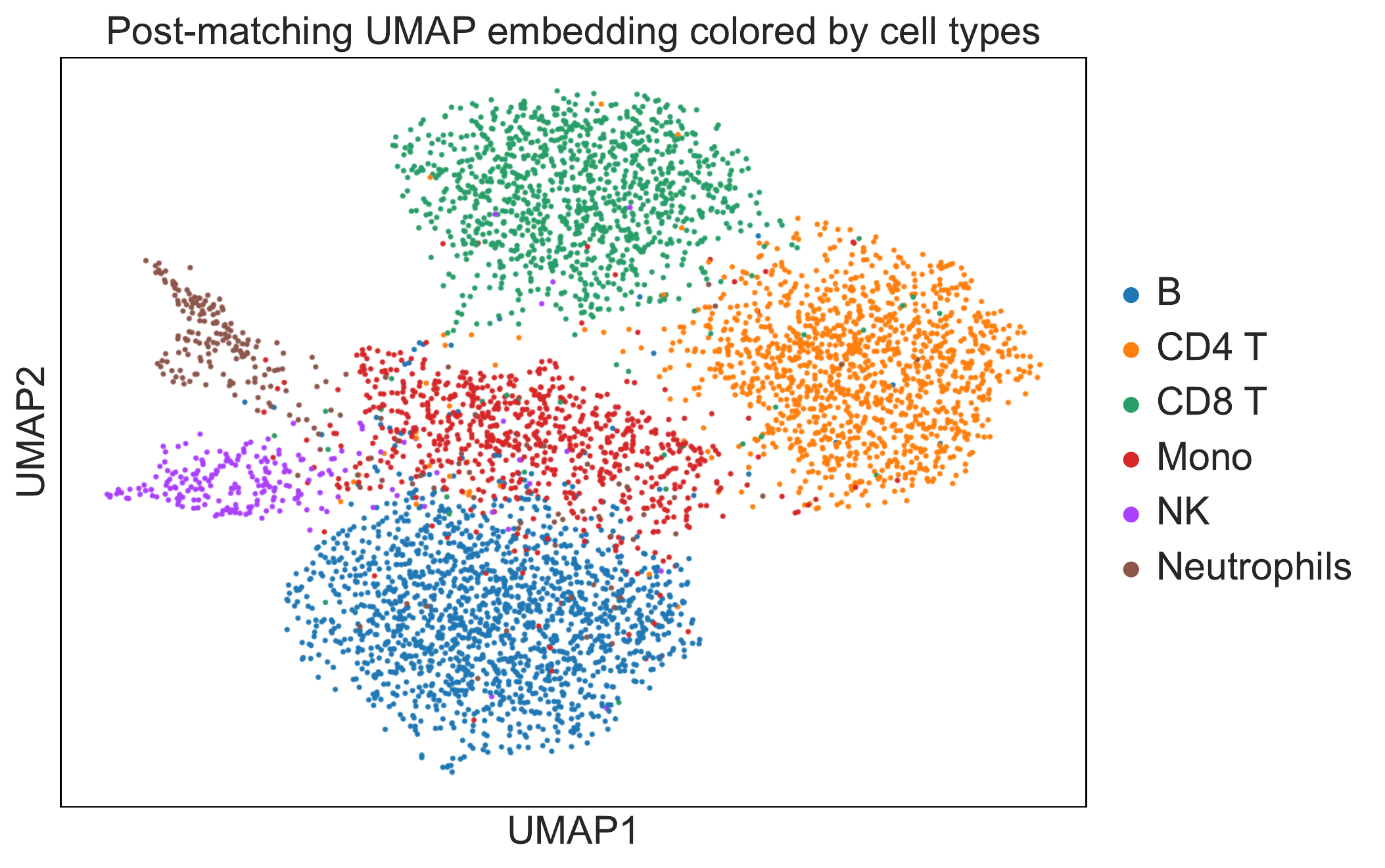}
	\end{subfigure}
       \caption{\small Two-dimensional UMAP embedding of CITE-seq and CODEX data, colored by dataset type (left two panels) and cell type (right two panels). {Similar to Figure \ref{fig:rna_umap}, the top two panels represent UMAP embeddings obtained from shared features (which do not rely on matching), whereas the bottom two panels represent UMAP embeddings obtained from CCA embeddings of all active features (which rely on matching by LAPS).}}
	\label{fig:protein_umap}
\end{figure}

{Figure \ref{fig:protein_umap} shows the UMAP embeddings of $(X_\sha\hat V, Y_\sha\hat V)$ (top two panels)} and those after performing CCA on $(X_\act, \hat \Pi Y_\act)$ (bottom two panels). Similar to what was shown in Figure \ref{fig:rna_umap}, we see a better mixing between two technologies (the left two panels) and a better separation among different cell types (the right two panels).
We refer interested readers to \cite{zhu2021robust} for more downstream analyses performed on these spleen datasets and for additional data examples, including 
transfer of {spatial information and transcriptomic information} on matched cells.
}

\section*{Acknowledgments}
S.C.~and Z.M.~are supported in part by NSF DMS-2210104.
G.P.N.~is supported in part by Hope Realized Medical Foundation 209477, Bill \& Melinda Gates Foundation INV-002704, and Rachford and Carlota A. Harris Endowed Professorship.

\small{
\setlength{\bibsep}{0.2pt plus 0.3ex}
\bibliographystyle{abbrvnat}
\bibliography{references}}

\newpage
\appendixtitleon
\appendixtitletocon
\begin{appendices}

\section{Proofs of lower bounds}

\subsection{Proof of Lemma \ref{lemma:cyc_decomp}} \label{prf:lemma:cyc_decomp}
By definition, we have
\begin{align*}
	\bbE[d(\hat\pi, \pi^\star)] & = \sum_{i_1\in[n]}\bbP(\hat \pi_{i_1} \neq \pi^\star_{i_1}) \\
	& = \sum_{i_1\in[n]} \sum_{i_2\neq i_1} \bbP(\hat \pi_{i_1} = \pi^\star_{i_2})\\
	& = \sum_{i_1\in[n]} \sum_{i_2 \neq i_1} 
	\left[
		\bbP(\hat\pi_{i_1} = \pi^\star_{i_2}, \hat \pi_{i_2} = \pi^\star_{i_1})
		+ 
		\bbP(\hat \pi_{i_1} = \pi^\star_{i_2}, \hat\pi_{i_2} \neq \pi^\star_{i_1})
	\right] \\
	& = \sum_{i_1\neq i_2}	
		\bbP(\hat\pi_{i_1} = \pi^\star_{i_2}, \hat \pi_{i_2} = \pi^\star_{i_1}) 
		+
		\sum_{i_1\in[n]} \sum_{i_2\neq i_1}
		\bbP(\hat \pi_{i_1} = \pi^\star_{i_2}, \hat\pi_{i_2} \neq \pi^\star_{i_1}).
\end{align*}
Note that we can decompose the second term in the right-hand side above as
\begin{align*}
	& \sum_{i_1\in[n]} \sum_{i_2\neq i_1}
		\bbP(\hat \pi_{i_1} = \pi^\star_{i_2}, \hat\pi_{i_2} \neq \pi^\star_{i_1})\\
	& = \sum_{i_1\in[n]} \sum_{i_2\neq i_1} \sum_{i_3\notin i_{1:2}}
		\bbP(\hat \pi_{i_1} = \pi^\star_{i_2}, \hat\pi_{i_2} = \pi^\star_{i_3}) \\
	& = \sum_{i\in[n]} \sum_{i_2 \neq i_1} \sum_{i_3 \notin i_{1:2}} 
		\left[
		\bbP(\hat \pi_{i_1} = \pi^\star_{i_2}, \hat\pi_{i_2} = \pi^\star_{i_3}, \hat\pi_{i_3} = \pi^\star_{i_1}) +  
		\bbP(\hat \pi_{i_1} = \pi^\star_{i_2}, \hat\pi_{i_2} = \pi^\star_{i_3}, \hat\pi_{i_3} \neq \pi^\star_{i_1}) 
		\right]\\
	& = \sum_{i_1\neq i_2 \neq i_3} 
		\bbP(\hat \pi_{i_1} = \pi^\star_{i_2}, \hat\pi_{i_2} = \pi^\star_{i_3}, \hat\pi_{i_3} = \pi^\star_{i_1}) +  
		\sum_{i_1\in[n]}\sum_{i_2\neq i_1} \sum_{i_3\notin i_{1:2}} \bbP(\hat \pi_{i_1} = \pi^\star_{i_2}, \hat\pi_{i_2} = \pi^\star_{i_3}, \hat\pi_{i_3} \neq \pi^\star_{i_1}).
\end{align*}
Thus, we have shown that
\begin{align*}
	\bbE[d(\hat\pi, \pi^\star)] & = \sum_{i_1\neq i_2}	
		\bbP(\hat\pi_{i_1} = \pi^\star_{i_2}, \hat \pi_{i_2} = \pi^\star_{i_1}) 
		+ \sum_{i_1\neq i_2 \neq i_3} 
		\bbP(\hat \pi_{i_1} = \pi^\star_{i_2}, \hat\pi_{i_2} = \pi^\star_{i_3}, \hat\pi_{i_3} = \pi^\star_{i_1}) \\  
		& \qquad + \sum_{i_1\in[n]}\sum_{i_2\neq i_1} \sum_{i_3\notin i_{1:2}} \bbP(\hat \pi_{i_1} = \pi^\star_{i_2}, \hat\pi_{i_2} = \pi^\star_{i_3}, \hat\pi_{i_3} \neq \pi^\star_{i_1}).
\end{align*}	
Recursively applying the above arguments, we get
\begin{align*}
	\bbE[d(\hat\pi, \pi^\star)]
	& = \sum_{i_1\neq i_2} 
		\bbP(\hat\pi_{i_1} = \pi^\star_{i_2}, \hat \pi_{i_2} = \pi^\star_{i_1})+ 
		\sum_{i_1\neq i_2 \neq i_3} 
		\bbP(\hat \pi_{i_1} = \pi^\star_{i_2}, \hat\pi_{i_2} = \pi^\star_{i_3}, \hat\pi_{i_3} = \pi^\star_{i_1})\\
		& \qquad + \cdots + 
		\sum_{i_1\neq i_2 \neq \cdots \neq i_n} \bbP(\hat\pi_{i_1} = \pi^\star_{i_2}, \hat\pi_{i_2} = \pi^\star_{i_3}, \hdots, \hat\pi_{i_{n-1}} = \pi^\star_{i_{n}}, \hat\pi_{i_n} = \pi^\star_{i_1}) \\
	& = \sum_{k=2}^n \sum_{i_1\neq i_2 \neq \cdots \neq i_k} 
	 \bbP(\hat\pi_{i_1} = \pi^\star_{i_2}, \hat\pi_{i_2} = \pi^\star_{i_3}, \hdots, \hat\pi_{i_{k-1}} = \pi^\star_{i_{k}}, \hat\pi_{i_k} = \pi^\star_{i_1}),
\end{align*}
which is the desired result.

\subsection{Proof of Theorem \ref{thm:lb}} \label{prf:thm:lb}
Recall that there is an one-to-one correspondence between a permutation matrix $\Pi$ and its vector representation $\pi$.
For notational simplicity, we omit the dependence of $\hat\Pi$ and $\hat \pi$ on $X $ and $Y $ when there is no ambiguity. For a discrete set $A$, we let $S(A)$ denotes the symmetric group on $A$. As a special case, we let $S_n = S([n])$ denote the set of all permutations on $[n]$.

Invoking Lemma \ref{lemma:cyc_decomp}, we have
\begin{align*}
	\inf_{\hat \Pi} \sup_{\Pi^\star} \bbE[\ell(\hat\Pi, \Pi^\star)] 
	& = \inf_{\hat \pi} \sup_{\pi^\star} \frac{1}{n} \sum_{k=2}^n \sum_{i_1\neq i_2 \neq \cdots \neq i_k} \bbP(\hat\pi_{i_1} = \pi^\star_{i_2}, \hat\pi_{i_2} = \pi^\star_{i_3}, \hdots, \hat\pi_{i_k} = \pi^\star_{i_1}) \\
	& \geq \inf_{\hat \pi} \bbE_{\pi^\star \sim \textnormal{Unif}(S_n)} \frac{1}{n} \sum_{k=2}^n \sum_{i_1\neq i_2 \neq \cdots \neq i_k} \bbP(\hat\pi_{i_1} = \pi^\star_{i_2}, \hat\pi_{i_2} = \pi^\star_{i_3}, \hdots, \hat\pi_{i_k} = \pi^\star_{i_1}),
\end{align*}
where $\textnormal{Unif}(S_n)$ is the uniform distribution on the set of all permutations $S_n$, and the last inequality is by the fact that the minimax error is lower bounded by the Bayes error. 
Fix any $k\geq 2$ and any indices $i_1\neq \cdots \neq i_k$, we can enumerate every $\pi^\star \in S_n$ as follows: 
\begin{enumerate}
	\item 
	Fix any $k$ indices $1\leq j_1 < j_2 < \cdots < j_k\leq n$; 
	\item 
	Set $\pi^\star_{-i_{1:k}}$ (i.e., the vector formed by collecting $\pi^\star_i$ for $i\notin i_{1:k}$) to be a certain element in $S([n]\setminus j_{1:k})$; 
	\item 
	Set $\pi^\star_{i_{1:k}}$ to be a certain element in $S(j_{1:k})$. 
\end{enumerate}
Note that this procedure indeed enumerates all elements in $S_n$, because it produces $\binom{n}{k} \cdot (n-k)! \cdot k! = n!$ many distinct permutations. Such an enumeration procedure gives the following lower bound
\begin{align*}
	& \inf_{\hat \Pi} \sup_{\Pi^\star} \bbE[\ell(\hat\Pi, \Pi^\star)] \\
	& \geq \inf_{\hat \pi}
	\frac{1}{n\cdot n!} \sum_{k=2}^n \sum_{i_1\neq i_2 \neq \cdots \neq i_k} \sum_{j_1 < j_2 < \cdots < j_k} \sum_{\pi^\star_{-i_{1:k}} \in S([n]\setminus j_{1:k})} \sum_{\gamma\in S_k} 
	\bbP(\hat\pi_{i_1} = \pi^\star_{i_2}, \hat\pi_{i_2} = \pi^\star_{i_3}, \hdots, \hat\pi_{i_k} = \pi^\star_{i_1}\mid \pi^\star_{i_{1:k}} = j_{\gamma}).
\end{align*}
Now we consider another procedure that enumerates the set of all permutations. In particular, we enumerate $S_k$ from $S_{k-1}$ as follows: 
\begin{enumerate}
	\item 
	Generate $\gamma \in S_{k-1}$; 
	\item 
	Let $\gamma^{(k)} = (\gamma, k)$, the concatenation of $\gamma$ and $k$; 
	\item Apply cyclic shifts to get $I^{\leftarrow}_k \gamma^{(k)}, (I^{\leftarrow}_k)^2 \gamma^{(k)}, \hdots, (I^{\leftarrow}_k)^{k-1} \gamma^{(k)}$. 
\end{enumerate}
In the end, we have
$$
	S_{k} = \bigg\{(I_{k}^\leftarrow)^{\ell} \gamma^{(k)}: \ell \in\{0, 1, \hdots, k-1\}, \gamma \in S_{k-1}\bigg\},
$$
which again follows from the fact that the above procedure produces $(k-1)! \cdot k = k!$ many distinct permutations.

Thus, we have
\begin{align*}
	& \inf_{\hat \Pi} \sup_{\Pi^\star} \bbE[\ell(\hat\Pi, \Pi^\star)] \\
	& \geq \inf_{\hat \pi}
	\frac{1}{n\cdot n!} \sum_{k=2}^n \sum_{i_1\neq i_2 \neq \cdots \neq i_k} \sum_{j_1 < j_2 < \cdots < j_k} \sum_{\pi^\star_{-i_{1:k}} \in S([n]\setminus j_{1:k})} \\
	& \qquad \sum_{\gamma\in S_{k-1}} 
	\bigg[
		\bbP\bigg(\hat\pi_{i_{1:k}} = I^\leftarrow_k j_{\gamma^{(k)}} ~\bigg|~ \pi^\star_{i_{1:k}} = j_{\gamma^{(k)}}\bigg) 
		+ 
		\bbP\bigg(\hat\pi_{i_{1:k}} = (I^\leftarrow_k)^2 j_{\gamma^{(k)}} ~\bigg|~ \pi^\star_{i_{1:k}} = (I^{\leftarrow}_k) j_{\gamma^{(k)}}\bigg)  \\
		& \qquad\qquad\qquad\qquad + \cdots + 
		\bbP\bigg(\hat\pi_{i_{1:k}} = (I^\leftarrow_k)^k j_{\gamma^{(k)}} ~\bigg|~ \pi^\star_{i_{1:k}} = ((I^{\leftarrow}_k))^{k-1} j_{\gamma^{(k)}}\bigg)
	\bigg] \\
	& \overset{(*)}{\geq}
		\frac{1}{n\cdot n!} \sum_{k=2}^n \sum_{i_1\neq i_2 \neq \cdots \neq i_k} \sum_{j_1 < j_2 < \cdots < j_k} \sum_{\pi^\star_{-i_{1:k}} \in S([n]\setminus j_{1:k})} \sum_{\gamma\in S_{k-1}} \\
	& \qquad
	\bigg\{
	\inf_{\hat \pi_{i_{1:k}} }
		\frac{1}{2}
		\bigg[\bbP\bigg(\hat\pi_{i_{1:k}} = I^\leftarrow_k j_{\gamma^{(k)}} ~\bigg|~ \pi^\star_{i_{1:k}} = j_{\gamma^{(k)}}\bigg) 
		+ 
		\bbP\bigg(\hat\pi_{i_{1:k}} = (I^\leftarrow_k)^2 j_{\gamma^{(k)}} ~\bigg|~ \pi^\star_{i_{1:k}} = I^{\leftarrow}_k j_{\gamma^{(k)}}\bigg)\bigg]  \\
	& \qquad + \inf_{\hat\pi_{i_{1:k}}}
	\frac{1}{2}
		\bigg[\bbP\bigg(\hat\pi_{i_{1:k}} = (I^\leftarrow_k)^2 j_{\gamma^{(k)}} ~\bigg|~ \pi^\star_{i_{1:k}} = I^\leftarrow_k j_{\gamma^{(k)}}\bigg) 
		+ 
		\bbP\bigg(\hat\pi_{i_{1:k}} = (I^\leftarrow_k)^3 j_{\gamma^{(k)}} ~\bigg|~ \pi^\star_{i_{1:k}} = (I^{\leftarrow}_k)^2 j_{\gamma^{(k)}}\bigg)\bigg]  \\
	& \qquad + \cdots \\
	& \qquad + \inf_{\hat\pi_{i_{1:k}}}
	\frac{1}{2}
		\bigg[\bbP\bigg(\hat\pi_{i_{1:k}} = (I^\leftarrow_k)^k j_{\gamma^{(k)}} ~\bigg|~ \pi^\star_{i_{1:k}} = (I^\leftarrow_k)^{k-1} j_{\gamma^{(k)}}\bigg) 
		+ 
		\bbP\bigg(\hat\pi_{i_{1:k}} = I^\leftarrow_k j_{\gamma^{(k)}} ~\bigg|~ \pi^\star_{i_{1:k}} = j_{\gamma^{(k)}}\bigg)\bigg] 
	\bigg\} \\
	& \geq \frac{2}{n\cdot n!} \sum_{i_1\neq i_2} \sum_{j_1 < j_2} \sum_{\pi^\star_{-i_{1:2}} \in S([n]\setminus j_{1:k})}  
	\inf_{\hat\pi_{i_{1:2}}}\frac{1}{2} 
	\bigg[\bbP\bigg(\hat\pi_{i_{1}} = j_2, \hat\pi_{i_2} =j_1 ~\bigg|~ \pi^\star_{i_{1}} = j_1, \pi^\star_{i_2} = j_2\bigg) \\
	& \qquad\qquad\qquad\qquad\qquad\qquad\qquad\qquad\qquad	+ 
	\bbP\bigg(\hat\pi_{i_{1}} = j_1, \hat\pi_{i_2} =j_2 ~\bigg|~ \pi^\star_{i_{1}} = j_2, \pi^\star_{i_2} = j_1\bigg)  \bigg]
\end{align*}
where $(*)$ is by $(I_k^\leftarrow)^k = I_k$ and the fact that the infimum of the sum is lower bounded by the sum of the infimum, and $(**)$ follows from only retaining the $k=2$ term.
Note that each summand in the right-hand side above is the average type-\RN{1} and type-\RN{2} error of a hypothesis testing problem that tries to differentiate $H_0$ from $H_{1}$, where 
$$
	H_0: \pi^\star_{i_1} = j_1, \pi^\star_{i_2} =j_2,
	\qquad
	H_1: \pi^\star_{i_1} = j_2, \pi^\star_{i_2} =j_1,
$$
The likelihood function under $H_\ell$ is given by
\begin{align*}
	& L_\ell(X , Y )\\
	& = (2\pi \sigma_\x^2)^{-np /2} \exp\bigg\{-\frac{1}{2\sigma_\x^2}\|X  - U  D  V ^\top \|_F^2\bigg\} 
		\cdot (2\pi \sigma_\y^2)^{-np /2} \exp\bigg\{-\frac{1}{2\sigma_\y^2}\|\Pi^\star Y  - U  D  V ^\top \|_F^2\bigg\} \\
	& = (2\pi \sigma_\x^2)^{-np /2} (2\pi \sigma_\y^2)^{-np /2} \exp\bigg\{-\frac{1}{2\sigma_\x^2}\|X  - U  D  V ^\top \|_F^2\bigg\}  \\
	& \qquad \times \exp\bigg\{-\frac{1}{2\sigma_\y^2}\|Y_{-\pi^\star_{i_{1:2}}, \bigcdot} - (U  D  V ^\top)_{-i_{1:2}, \bigcdot}\|_F^2\bigg\}
		\cdot \exp\bigg\{- \frac{1}{2\sigma_\y^2} \|Y_{j_{1:2}, \bigcdot} - (U  D  V ^\top)_{i_{1:2}, \bigcdot} \|_F^2\bigg\} \\
	& = (2\pi \sigma_\x^2)^{-np /2} (2\pi \sigma_\y^2)^{-np /2} \exp\bigg\{-\frac{1}{2\sigma_\x^2}\|X  - U  D  V ^\top \|_F^2\bigg\} \\
	& \qquad \times \exp\bigg\{-\frac{1}{2\sigma_\y^2}\|Y_{-\pi^\star_{i_{1:2}}, \bigcdot} - (U  D  V ^\top)_{-i_{1:2}, \bigcdot}\|_F^2\bigg\} \\
	& \qquad \times \exp\bigg\{- \frac{1}{2\sigma_\y^2} \bigg( \|Y_{j_{1:2}, \bigcdot}\|_F^2 +  \|(U  D  V ^\top)_{i_{1:2}, \bigcdot}\|_F^2\bigg) + \frac{1}{\sigma_\y^2} \bigg\la (U  D  V ^\top)_{i_{1:2}, \bigcdot}, (I_2^\leftarrow)^\ell Y_{j_{\gamma^{(k)}}, \bigcdot} \bigg\ra\bigg\}. 
\end{align*}
By Neyman--Pearson lemma, the optimal test for $H_0$ v.s. $H_1$ that minimizes the average type-\RN{1} and type-\RN{2} error is given by rejecting $H_0$ when $L_0(X , Y ) \leq L_{1}(X , Y )$, which is equivalent to
\begin{align*}
	 \left\la (U  D  V ^\top)_{i_{1:2}, \bigcdot}, [(I_2^\leftarrow)^\ell - (I_2^\leftarrow)^{\ell+1}] Y_{j_{1:2}, \bigcdot}  \right\ra \leq 0
	& \iff 
	\left\la (U  D  V ^\top)_{i_{1:2}, \bigcdot}, (I_2 - I_2^\leftarrow) Y_{ (I_2^\leftarrow)^\ell j_{1:2}, \bigcdot}  \right\ra \leq 0.
\end{align*}
Note that under $H_\ell$, we can write $Y_{ (I_2^\leftarrow)^\ell j_{1:2}, \bigcdot} = (U  D   V ^\top)_{i_{1:2}, \bigcdot} + \sigma_\y \calE$, where $\calE\in\bbR^{k\times p }$ is a matrix with i.i.d.~$N(0, 1)$ entries. 
Thus, the type-\RN{1} error of the optimal test is given by
\begin{align*}
	& \bbP\left( \left\la (U  D   V ^\top)_{i_{1:2}, \bigcdot}, (I_2 - I_2^\leftarrow) (U  D   V ^\top)_{i_{1:2}, \bigcdot} + \sigma_\y \calE \right\ra \leq 0\right) \\
	& = \bbP\left(
			\left\la (U  D   V ^\top)_{i_{1:2}, \bigcdot}, (I_2 - I_2^\leftarrow) (U  D   V ^\top)_{i_{1:2}, \bigcdot} \right\ra
			+   N\left(0, \sigma_\y^2\| (I_2 - I_2^\leftarrow) (U  D   V ^\top)_{i_{1:2}, \bigcdot}\|_F^2 \right)
			\leq 0
		\right)\\
	& = \bbP\left(
			\frac{1}{2} \| (I_2 - I_2^\leftarrow)(U  D  V ^\top)_{i_{1:2}, \bigcdot}\|_F^2 + 
			N\left(0, \sigma_\y^2\| (I_2 - I_2^\leftarrow) (U  D   V ^\top)_{i_{1:2}, \bigcdot}\|_F^2 \right)
			\leq 0
		\right) \\
	& = \Phi\left(
			-\frac{ \| (I_2 - I_2^\leftarrow)(U  D  V ^\top)_{i_{1:2}, \bigcdot}\|_F}{2\sigma_\y}
		\right) \\
	& = \Phi\left(
			-\frac{ \| (I_2 - I_2^\leftarrow)U_{i_{1:2}, \bigcdot} D \|_F}{2\sigma_\y}
		\right), 
\end{align*}
where the second equality is by Lemma \ref{lemma:inner_prod_and_norm} and the last equality is by $V \in O(p , r )$.
A symmetric argument shows that the right-hand above is also the type-\RN{2} error of the optimal test. Thus, we have
\begin{align*}
	& \inf_{\hat\pi_{i_{1:2}}}\frac{1}{2} 
	\bigg[\bbP\bigg(\hat\pi_{i_{1}} = j_2, \hat\pi_{i_2} =j_1 ~\bigg|~ \pi^\star_{i_{1}} = j_1, \pi^\star_{i_2} = j_2\bigg)
	+ 
	\bbP\bigg(\hat\pi_{i_{1}} = j_1, \hat\pi_{i_2} =j_2 ~\bigg|~ \pi^\star_{i_{1}} = j_2, \pi^\star_{i_2} = j_1\bigg)  \bigg] \\
	& \geq \Phi\left(
			-\frac{ \| (I_2 - I_2^\leftarrow)U_{i_{1:2}, \bigcdot} D \|_F}{2\sigma_\y}
		\right),
\end{align*}
and hence
\begin{equation}
\label{eq:lb_1}
	\inf_{\hat \Pi} \sup_{\Pi^\star} \bbE[\ell(\hat\Pi, \Pi^\star)] 
	\geq \frac{1}{n} \sum_{i_1\neq i_2}
	\Phi\bigg(- \frac{\|(I_2 - I_2^\leftarrow)U_{i_{1:2}, \bigcdot} D \|_F }{2 \sigma_\y}\bigg)
	= \frac{1}{n} \sum_{i\neq i'} \Phi\left(- \frac{\|(U_{i, \bigcdot} - U_{i', \bigcdot})D\|}{\sqrt{2}\cdot\sigma_\y}\right).
\end{equation}
The above lower bound was derived under the data generating process
$$
	X  = U  D  V ^\top + \sigma_\x N_\x, \qquad \Pi^\star Y  = U  D  V ^\top + \sigma_\y N_\y,
$$
where $N_\x, N_\y \in \bbR^{n\times p }$ have i.i.d.~$N(0, 1)$ entries. Note that we can alternatively write the data generating process as
$$
	\tilde\Pi^\star X  = \tilde U  D  V ^\top + \sigma_\x \tilde N_\x, \qquad Y  = \tilde U  D  V ^\top + \sigma_\y \tilde N_\y,
$$
where $\tilde \Pi^\star = (\Pi^\star)^\top, \tilde U  = (\Pi^\star)^\top U $, and $\tilde N_\x, \tilde N_\y \in \bbR^{n\times p }$ again have i.i.d.~$N(0, 1)$ entries. 
Repeating the arguments that led to \eqref{eq:lb_1}, we get
\begin{align*}
	\inf_{\hat \Pi} \sup_{\Pi^\star} \bbE[\ell(\hat\pi, \pi^\star)] 
	& \geq \frac{1}{n} \sum_{i\neq i'} \Phi\left(- \frac{\|(\tilde U_{i, \bigcdot} - \tilde U_{i', \bigcdot})D\|}{\sqrt{2}\cdot\sigma_\x}\right).\\
	\label{eq:lb_2}
	& = \frac{1}{n} \sum_{i\neq i'} \Phi\left(- \frac{\|(U_{i, \bigcdot} - U_{i', \bigcdot})D\|}{\sqrt{2}\cdot\sigma_\x}\right). \numberthis
\end{align*}
Summarizing \eqref{eq:lb_1} and \eqref{eq:lb_2}, we get
\begin{align*}
	\inf_{\hat\Pi(X , Y )} \sup_{\Pi^\star \in S_n} \bbE[\ell(\hat \Pi, \Pi^\star)] 
	\geq \frac{1}{n} \sum_{i\neq i'} \Phi\left(- \frac{\|(U_{i, \bigcdot} - U_{i', \bigcdot})D\|}{\sqrt{2}\cdot\sigma_{\max}}\right).
\end{align*}	
Invoking Lemma \ref{lemma:gaussian_tail}, the right-hand side above can be further lower bounded by
\begin{align*}
	& \frac{1}{\sqrt{2\pi} \cdot n} \sum_{i\neq i'} \left( \frac{\sqrt{2}\cdot \sigma_{\max}}{\|(U_{i, \bigcdot} - U_{i', \bigcdot})D\|} - \frac{2\sigma_{\max}^2}{\|(U_{i, \bigcdot} - U_{i', \bigcdot})D\|^2}\right) \exp\left\{-\frac{\|(U_{i, \bigcdot} - U_{i', \bigcdot})D\|^2}{4\sigma_{\max}^2}\right\} \\
	& \geq \frac{1}{n} \sum_{i\neq i'} \exp\left\{-\frac{(1+o(1))\|(U_{i, \bigcdot} - U_{i', \bigcdot})D\|^2}{4\sigma_{\max}^2}\right\},
\end{align*}
where the last inequality is by $\|(U_{i, \bigcdot} - U_{i',\bigcdot})D\|^2/\sigma_{\max}^2 = \beta^2 \gg 1$.

\section{Proofs of upper bounds}

\subsection{Proof of Proposition \ref{prop:est_err_of_V}} \label{prf:prop:est_err_of_V}
Without loss of generality, we assume $\sigma_{\x} \leq \sigma_\y$, and we write $\sigma_\x = \sigma$.
We first invoke a version of $\sin\Theta$ theorem proved in Proposition 1 of \cite{cai2018rate}, which states that 
\begin{align}
\label{eq:sintheta}
	\|\hat V \hat V^\top - V V^\top\| \leq \frac{\sigma_r(XV) \|P_{XV} X V_\perp\|}{\sigma_r^2(XV) - \sigma_{r+1}^2(XV)},
\end{align}
provided $\sigma_r(XV) > \sigma_{r+1}(X)$.
In the above display, $\sigma_r(XV)$ is the $r$-th singular value of $XV$, $P_{XV}$ is the projection matrix on the column space of $XV$, and $V_\perp \in O_{p, p-r}$ collect the bottom $(p-r)$ right singular vectors of $X$. Consider the event
$$
	E_{C_1, \delta} := \{\|NV\|\leq C (\sqrt{n} + \sqrt{r} + \sqrt{\log (1/\delta)})\},
$$
where the values of $C_1 > 0, \delta\in(0, 1)$ are to be determined later. 
By Lemma \ref{lemma:gaussian_wigner_op_norm}, we know that
$$
	\bbP(E_{C_1, \delta}^c) \lesssim \delta.
$$
Under $E_{C_1, \delta}$, we have
$$
	-C_1\sigma(\sqrt{n} + \sqrt{r} + \sqrt{\log(1/\delta)}) + d_r \leq \sigma_r(XV) \leq d_r + C_1\sigma(\sqrt{n} + \sqrt{r} + \sqrt{\log(1/\delta)}).
$$
Meanwhile, consider another event
$$
	F_{C_1, \delta} := \{\|N\|\leq C_1 (\sqrt{n} + \sqrt{p} + \sqrt{\log(1/\delta)})\}.
$$
Invoking Lemma \ref{lemma:gaussian_wigner_op_norm} again, we have
$$
	\bbP(F_{C_1, \delta}^c) \lesssim \delta.
$$
Choosing $\delta = e^{-n}$, under $E_{C_1, \delta}\cap F_{C_1, \delta}$, we have
\begin{align*}
	\sigma_r(XV) - \sigma_{r+1}(X) 
	& \geq d_r - 2C_1\sigma (\sqrt{n} + \sqrt{p} + \sqrt{\log(1/\delta)}) \\
	& \geq d_r - C_2 \sigma(\sqrt{n} + \sqrt{p}) \\
	& \geq \left(1 - \frac{C_2}{C_\gap}\right) d_r,
\end{align*}
where $C_2$ is another absolute constant and the last inequality is by our assumption that $d_r/\sigma \geq C_\gap (\sqrt{n} + \sqrt{p})$. For a sufficiently large $C_\gap$, the right-hand side above is strictly positive, and thus \eqref{eq:sintheta} holds under $E_{C_1, \delta}\cap F_{C_1, \delta}$. To further upper bound the right-hand side of \eqref{eq:sintheta}, we begin by noting that under $E_{C_1, \delta}$ with $\delta = e^{-n}$, 
\begin{align*}
	\sigma_r^2(XV) - \sigma_{r+1}^2(XV) 
	& \geq \left(d_r - C_2\sigma(\sqrt{n} + \sqrt{r})\right)^2 - \left(C_2\sigma(\sqrt{n} + \sqrt{r})\right)^2\\
	& \geq \left(1 - \frac{C_2}{C_\gap}\right)^2 d_r^2 - \left(\frac{C_2}{C_\gap}\right)^2 d_r^2 \\
	& \gtrsim d_r^2
\end{align*}
for sufficiently large $C_\gap$. Meanwhile, we have
$$
	\sigma_r(XV) \leq d_r + C_1\sigma(\sqrt{n} + \sqrt{r} + \sqrt{\log(1/\delta)}) \lesssim d_r.
$$
Thus, under $E_{C_1, \delta}\cap F_{C_1, \delta}$, which happens with probability at least $1- \calO(e^{-n})$, we have
$$
	\|\hat V \hat V^\top - V V^\top\| \lesssim \frac{\|P_{XV} XV_\perp\|}{d_r}. 
$$
We now invoke Lemma 4 in \cite{cai2018rate}, which states that for any $x>0$,
\begin{align*}
	\bbP\left(\|P_{XV/\sigma} XV_\perp/\sigma\|\geq x\right) \lesssim \exp\left\{C_3 p - C_4 \min(x^2, x \sqrt{n + d_r^2/\sigma^2})\right\} + \exp\left\{-C_4 (n + d_r^2/\sigma^2)\right\},
\end{align*}
where $C_3, C_4 > 0$ are absolute constants. We choose $x = C_5 \sqrt{p \log n}$ for some absolute constant $C_5$. Then
\begin{align*}
	\min\left(x^2 , x \sqrt{n + d_r^2/\sigma^2}\right)
	& = p \cdot \min \left(C_5^2 \log n, C_5 \sqrt{(n + d_r^2/\sigma^2) (\log n)/p }  \right),
\end{align*}
and thus
\begin{align*}
	C_3 p - C_4 \min\left(x^2 , x \sqrt{n + d_r^2/\sigma^2}\right)
	& \leq -p \left( C_4 \min \left(C_5^2 \log n, C_5 \sqrt{(n + d_r^2/\sigma^2) (\log n)/p }  \right) - C_3\right).
\end{align*}
When $p \leq n / \log n$, we get
\begin{align*}
	C_3 p - C_4 \min\left(x^2 , x \sqrt{n + d_r^2/\sigma^2}\right)
	& \leq -p \left( C_4 \min \left(C_5^2 \log n, C_5 \sqrt{(n\log n)/p }  \right) - C_3\right) \\
	& \leq -p \left( C_4 \min \left(C_5^2 \log n, C_5 \log n \right) - C_3\right) \\
	& \leq - C_6 p\log n,
\end{align*}
where $C_6$ is another absolute constant. On the other hand, if $p\geq n/\log n$, we get
\begin{align*}
	C_3 p - C_4 \min\left(x^2 , x \sqrt{n + d_r^2/\sigma^2}\right)
	& \leq -p \left( C_4 \min \left(C_5^2 \log n, C_5 \sqrt{\log n}  \right) - C_3\right) \\
	& \leq -C_7 p \\
	& \leq -\frac{ C_7 n }{\log n},
\end{align*}
where the first inequality is by $d_r/\sigma \geq C_\gap(\sqrt{n} + \sqrt{p})$, and $C_7> 0$ is an absolute constant. In summary, it is possible to choose $C_5$ such that
$$
	\|P_{XV/\sigma} X V_\perp/\sigma\| \lesssim \sqrt{p\log n}
$$
with probability at least $1 - n^{-C_8}$, where $C_8$ can be arbitrarily large. This is equivalent to
$$
	\|P_{XV} XV_\perp\| \lesssim \sigma \sqrt{p\log n}
$$
with probability at least $1-n^{-C_8}$. Invoking a union bound, we conclude that 
$$
	\|\hat V \hat V^\top - VV^\top\| \lesssim \frac{\sqrt{p\log n}}{d_r/\sigma}
$$
with probability at least $1-n^{-c}$, and hence
$$
	\|\hat V \hat V^\top - V V^\top\|_F \lesssim \frac{\sqrt{rp\log n}}{d_r/\sigma}
$$
with probability at least $1-n^{-c}$.

\subsection{Proof of Proposition \ref{prop:loco_err_of_V}} \label{prf:prop:loco_err_of_V}
Without loss of generality, we assume $\sigma_\x\leq \sigma_\y$, and we write $\sigma_\x = \sigma$.
Note that $\hat V$ collects the top $r$ eigenvectors of $X^\top X = \sum_{i\in[n]} X_{i, \bigcdot}^\top, X_{i, \bigcdot}$ and $\hat V^{(-\calC_k)}$ collects the top $r$ eigenvectors of 
$$
(X^{(-\calC_k)})^\top X^{(-\calC_k)} = \sum_{i\notin\calC_k} X_{i, \bigcdot}^\top X_{i, \bigcdot} + \sum_{i\in \calC_k} \bbE[X_{i, \bigcdot}]^\top \bbE[X_{i, \bigcdot}].
$$
The Frobenius norm of the between $X^\top X$ and $(X^{(-\calC_k)})^\top X^{(-\calC_k)}$ is
\begin{align*}
 	\left\|X^\top X - (X^{(-\calC_k)})^\top X^{(-\calC_k)} \right\|_F
 	& = \left\|\sum_{i\in\calC_k} \left[\left(V D U_{i, \bigcdot}^\top + (N_{\x})_{i,\bigcdot}^\top\right) \left(V D U_{i, \bigcdot}^\top + (N_{\x})_{i,\bigcdot}^\top\right)^\top  - \left(V D U_{i, \bigcdot}^\top\right) \left(V D U_{i, \bigcdot}^\top\right)^\top\right]\right\|_F \\
 	& = \left\|\sum_{i\in \calC_k} \left[VDU_{i,\bigcdot}^\top (N_{\x})_{i, \bigcdot} + (N_{\x})_{i,\bigcdot}^\top  U_{i, \bigcdot} D V^\top + (N_\x)_{i,\bigcdot}^\top (N_\x)_{i,\bigcdot}\right]\right\|_F\\
 	& \leq \sum_{i\in\calC_k}  \left(2\|D U_{i,\bigcdot}^\top (N_\x)_{i,\bigcdot}\|_F +  \|(N_\x)_{i, \bigcdot}^\top (N_\x)_{i, \bigcdot}\|_F\right)\\
 	& = \sum_{i\in\calC_k} \left( 2 \|(N_\x)_{i,\bigcdot}\|\|U_{i,\bigcdot} D\| + \|(N_\x)_{i,\bigcdot}\|^2 \right)\\
 	& \lesssim k \cdot \left[\max_{i\in[n]} \|(N_{\x})_{i,\bigcdot}\| \cdot d_1 \max_{i\in[n]} \|U_{i,\bigcdot}\| + \max_{i\in[n]} \|(N_{\x})_{i,\bigcdot}\|^2 \|\right] \\
 	& \leq k \cdot \left[ d_1 \sqrt{\frac{\mu r}{n}} \cdot \max_{i\in[n]} \|(N_{\x})_{i,\bigcdot}\| + \max_{i\in[n]} \|(N_{\x})_{i,\bigcdot}\|^2 \|\right]
 \end{align*}	
Invoking Davis-Kahan $\sin\Theta$ theorem \citep{davis1970rotation}, we get
$$
	\|\hat V^{(-\calC_k)} (\hat V^{(-\calC_k)})^\top - \hat V \hat V^\top \|_F
	\lesssim \frac{k \cdot \left[ d_1 \sqrt{\frac{\mu r}{n}} \cdot \max_{i\in[n]} \|(N_{\x})_{i,\bigcdot}\| + \max_{i\in[n]} \|(N_{\x})_{i,\bigcdot}\|^2 \|\right]}{\sigma_r^2(X) - \sigma_{r+1}^2(X)},
$$
where $\sigma_r(X)$ is the $r$-th singular value of $X$.
Note that the right-hand side of the above inequality is independent of $\calC_k$. Hence, we have
$$
	\max_{k\leq k^\star} \max_{\calC_k}\|\hat V^{(-\calC_k)} (\hat V^{(-\calC_k)})^\top - \hat V \hat V^\top \|_F
	\lesssim \frac{k^\star \cdot \left[ d_1 \sqrt{\frac{\mu r}{n}} \cdot \max_{i\in[n]} \|(N_{\x})_{i,\bigcdot}\| + \max_{i\in[n]} \|(N_{\x})_{i,\bigcdot}\|^2 \|\right]}{\sigma_r^2(X) - \sigma_{r+1}^2(X)},
	$$
We have shown in the proof of Proposition \ref{prop:est_err_of_V} that if $\frac{d_r}{\sigma} \geq C_\gap (\sqrt{n} + \sqrt{p})$, then $\sigma_r^2(X) - \sigma_{r+1}^2(X)\gtrsim d_r^2$ with probability at least $1-e^{-n}$. 
For a fixed $i\in[n]$, by Lemma \ref{lemma:chisq_tail}, we have
$$
	\|(N_\x)_{i, \bigcdot}\|^2 \lesssim \sigma^2 \left(p + \sqrt{p\log n} + \log n\right) \asymp  \sigma^2 \left(\sqrt{p} +\sqrt{\log n}\right)^2
$$
with probability at least $1-n^{-c'}$, where $c'$ can be chosen arbitrarily large. Invoking a union bound over $[n]$, we get
$$
	\max_{i\in[n]} \|(N_\x)_{i, \bigcdot}\|^2 \lesssim \sigma^2 \left(\sqrt{p} +\sqrt{\log n}\right)^2
$$
with probability $1 - n^{-(c'-1)}$. 
Thus, with probability $1 - n^{-(c'-1)} - e^{-n}\geq n^{-c}$, we have
$$
	\max_{k\leq k^\star}\max_{\calC_k}\|\hat V^{(-\calC_k)} (\hat V^{(-\calC_k)})^\top - \hat V \hat V^\top \|_F \lesssim 
	k^\star \cdot \left[ \frac{\sqrt{\mu r} (\sqrt{p} + \sqrt{\log n}) d_1/d_r}{\sqrt{n} d_r/\sigma}  + \frac{(\sqrt{p} + \sqrt{\log n})^2}{d_r^2/\sigma^2}\right]
$$
which is the desired result.

\subsection{Proof of Theorem \ref{thm:poly_rate}} \label{prf:thm:poly_rate}
Recall that there is an one-to-one correspondence between a permutation matrix $\Pi$ and its vector representation $\pi$.
LAPS solves
$$
	\hat\Pi \in \argmin_{\Pi\in S_n} \| X \hat V - \Pi Y \hat V\|_F^2
$$
and that $N_\x, N_\y \in \bbR^{n\times p}$ are matrices with i.i.d. standard Gaussian entries.
For a fixed $\Pi$, we have
\begin{align*}
	\| X \hat V - \Pi Y \hat V\|_F^2
	& = \| U D V^\top \hat V + \sigma_\x N_\x \hat V - \Pi (\Pi^\star)^\top (U D V \hat V) - \Pi (\Pi^\star)^\top \sigma_\y N_\y \hat V\|_F^2 \\
	& = \| (I_n - \Pi (\Pi^\star)^\top) U D V^\top \hat V + \sigma_\x N_\x \hat V - \Pi (\Pi^\star)^\top \sigma_\y N_\y \hat V\|_F^2.
\end{align*}
By construction, we have $\|X \hat V - \hat\Pi Y \hat V\|_F^2 \leq \|X \hat V - \Pi^\star Y \hat V\|_F^2$, which gives
$$
	\| (I_n - \hat \Pi (\Pi^\star)^\top) U D V^\top \hat V + \sigma_\x N_\x \hat V - \hat \Pi (\Pi^\star)^\top \sigma_\y N_\y \hat V\|_F^2 \leq \| \sigma_\x N_\x \hat V - \sigma_\y N_\y \hat V\|_F^2,
$$
and it implies that
\begin{align*}
	& \|(I_n - \hat \Pi (\Pi^\star)^\top) U D V^\top \hat V\|_F^2 \\
	& \leq 
	- \| \sigma_\x N_\x \hat V - \hat \Pi (\Pi^\star)^\top \sigma_\y N_\y \hat V\|_F^2 
	+ 2 \bigg\la (I_n - \hat \Pi (\Pi^\star)^\top) U D V^\top \hat V,~ \hat \Pi (\Pi^\star)^\top \sigma_\y N_\y \hat V - \sigma_\x N_\x \hat V \bigg\ra \\
	& \qquad + \|(\sigma_\x N_\x - \sigma_\y N_\y) \hat V\|_F^2\\
	& \leq 2 \|(I_n - \hat \Pi (\Pi^\star)^\top) U D V^\top \hat V\|_F \|\hat \Pi (\hat \Pi^\star)^\top  \sigma_\y N_\y \hat V - \sigma_\x N_\x \hat V \|_F + \|(\sigma_\x N_\x - \sigma_\y N_\y) \hat V\|_F^2\\
	& \leq 2 \|(I_n - \hat \Pi (\Pi^\star)^\top) U D V^\top \hat V\|_F \cdot \bigg(\|\sigma_\x N_\x \hat V\|_F + \|\sigma_\y N_\y \hat V\|_F\bigg) + \bigg(\|\sigma_\x N_\x \hat V\|_F + \|\sigma_\y N_\y \hat V\|_F\bigg)^2.
\end{align*}	
Solving the above quadratic inequality, we get
$$
	\|(I_n - \hat \Pi (\Pi^\star)^\top) U D V^\top \hat V\|_F^2 \lesssim \|\sigma_\x N_\x \hat V\|_F^2 + \|\sigma_\y N_\y \hat V\|_F^2.
$$
Note that
\begin{align*}
	\|N_\x \hat V\|_F^2 
	& = \tr(N_\x \hat V \hat V^\top N_\x^\top ) \\
	& = \tr(N_\x (\hat V \hat V^\top -  V V^\top)N_\x^\top ) + \|N_\x V\|_F^2 \\
	& \leq \|N_\x^\top N_\x\|_F \|\hat V \hat V^\top - V V^\top\|_F  + \|N_\x V\|_F^2 \\
	& \leq \|N_\x\|_F \|N_\x\|_2\|\hat V\hat V^\top - VV^\top\|_F + \|N_\x V\|_F^2
\end{align*}	
and a similar bound holds for $\|N_\x \hat V\|_F^2$. Thus,
$$
	\|(I_n - \hat \Pi (\Pi^\star)^\top) U D V^\top \hat V\|_F^2 \lesssim 
	\|\hat V \hat V^\top - V V^\top\|_F \bigg( \sigma_\x^2 \|N_\x\|_F\|N_\x\|_2 + \sigma_\y^2\|N_\y\|_F\|N_\y\|_2 \bigg) + \sigma_\x^2\|N_\x V\|_F^2 + \sigma_\y^2 \|N_\y V\|_F^2.
$$
We now lower bound the left-hand side above. Note that
\begin{align*}
	& \|(I_n - \hat \Pi (\Pi^\star)^\top) U D V^\top \hat V\|_F^2 \\
	& = \tr\bigg(
		(I_n - \hat\Pi (\Pi^\star)^\top) U D V^\top \hat V \hat V^\top V D U^\top (I_n - \hat\Pi (\Pi^\star)^\top)^\top
	\bigg)\\
	& = \tr\bigg(
		(I_n - \hat \Pi (\Pi^\star)^\top) U D V^\top  (\hat V \hat V^\top - V V^\top) V D U^\top (I_n - \hat \Pi (\Pi^\star)^\top)^\top
	\bigg)\\
	& \qquad 
	+ \tr\bigg(
		(I_n - \hat \Pi (\Pi^\star)^\top) U D^2 U^\top (I_n - \hat \Pi (\Pi^\star)^\top)^\top
	\bigg)\\
	& \geq - \|(I_n - \hat \Pi (\Pi^\star)^\top) U D V^\top\|_F^2 \cdot  \|\hat V \hat V^\top - V V^\top\|_F + \|(I_n - \hat \Pi (\Pi^\star)^\top) U D\|_F^2 \\
	& = \bigg(1 - \|\hat V \hat V^\top - V V^\top\|_F\bigg) \|(I_n - \hat \Pi (\Pi^\star)^\top)U D\|_F^2 .
\end{align*}
Without loss of generality, we assume $\Pi^\star$ is the identity permutation. Let $\calC_k(\hat \pi)$ be the set of all length-$k$ cycles of $\hat\pi$, so that for a specific cycle $\{i_1, \hdots, i_k\}\in\calC_k (\hat\pi)$, we have $\hat\pi_{i_1} =\pi^\star_{i_2} = i_2, \hat\pi_{i_2}= \pi^\star_{i_3} = i_3, \hdots, \hat\pi_{i_k} = \pi^\star_{i_1}=i_1$. We can then write
\begin{align*}
	\|(I_n - \hat \Pi (\Pi^\star)^\top)U D\|_F^2
	& = \sum_{k=2}^n \sum_{i_{1:k}\in\calC_k(\hat\pi)} \| (UD)_{i_{1:k},\bigcdot} - I_k^\leftarrow (U D)_{i_{1:k},\bigcdot}\|_F^2 \\
	& \geq \sum_{k=2}^n \sum_{i_{1:k}\in \calC_k(\hat\pi)} k \cdot \min_{i\neq i'}  \|(U D)_{i, \bigcdot} - (U D)_{i', \bigcdot} \|^2 \\
	& = \min_{i\neq i'}  \|(U D)_{i, \bigcdot} - (U D)_{i', \bigcdot} \|^2  \cdot \sum_{k=2}^n k \cdot |\calC_k(\hat \pi )| \\
	& = \min_{i\neq i'}  \|(U D)_{i, \bigcdot} - (U D)_{i', \bigcdot} \|^2 \cdot d(\hat\pi, \pi^\star).
\end{align*}	
Hence, we have
\begin{align*}
	& d(\hat\pi, \pi^\star) \\
	& \lesssim \frac{\sigma_{\max}^2}{ \min_{i\neq i'}  \|(U D)_{i, \bigcdot} - (U D)_{i', \bigcdot} \|^2} \cdot \frac{1}{1 - \|\hat V \hat V^\top - V V^\top\|_F} \\
	& \qquad \times \bigg[\|\hat V \hat V^\top - V V^\top\|_F \bigg(  \|N_\x\|_F\|N_\x\|_2 +\|N_\y\|_F\|N_\y\|_2 \bigg) + \|N_\x V\|_F^2 + \|N_\y V\|_F^2\bigg] \\
	& = \frac{1}{\beta^2} \cdot\frac{1}{1 - \|\hat V \hat V^\top - V V^\top\|_F}\cdot \bigg[\|\hat V \hat V^\top - V V^\top\|_F \bigg(  \|N_\x\|_F\|N_\x\|_2 +\|N_\y\|_F\|N_\y\|_2 \bigg) + \|N_\x V\|_F^2 + \|N_\y V\|_F^2\bigg].
\end{align*}
By Proposition \ref{prop:est_err_of_V}, we have
$$
	\|\hat V \hat V^\top - VV^\top\|_F \lesssim \frac{\sqrt{rp\log n}}{d_r/\sigma_{\min}} = \calE_\unif
$$
with probability at least $1-n^{-c}$ for some absolute constant $c$. Meanwhile, note that $\|N_\x\|_F^2$ and $\|N_\y\|^2_F$ are both $\chi^2$ random variables with $np$ degrees of freedom, so by Lemma \ref{lemma:chisq_tail}, we have
$$
\|N_\x\|_F^2 \lor \|N_\y\|_F^2 \lesssim np + \sqrt{np\log n} + \log n \lesssim np
$$
with probability at least $1-n^{-c}$. Similarly
$$
	\|N_\x V\|_F^2 \lor \|N_\y V\|_F^2 \lesssim nr
$$
with probability at least $1-n^{-c}$.
Moreover, by Lemma \ref{lemma:gaussian_wigner_op_norm}, 
$$
	\|N_\x\|_2 \lor \|N_\y\|_2 \lesssim \sqrt{n} + \sqrt{p} + \sqrt{\log n} \lesssim \sqrt{n} + \sqrt{p}
$$
with probability at least $1-n^{-c}$. Invoking a union bound and using $\calE_\unif \ll 1$, we get
\begin{align*}
	\ell(\hat\Pi, \Pi^\star) = \frac{1}{n}d(\hat \pi, \pi^\star) \lesssim \frac{1}{\beta^2} \cdot \frac{\calE_\unif \sqrt{np}(\sqrt{n}+\sqrt{p}) + nr}{n} \lesssim \frac{p}{\beta^2} \cdot \calE_\unif \left(\frac{1}{\sqrt{n}}+ \frac{1}{\sqrt{p}}\right) + \frac{r}{\beta^2}
\end{align*}
with probability $1-n^{-c'}$ for another constant $c'$, which is the desired result.

\subsection{Proof of Theorem \ref{thm:ub_simplified}} \label{prf:thm:ub_simplified}
We first state a generalization of Theorem \ref{thm:ub_simplified}.

\begin{theorem}
	\label{thm:ub}
	Recall the definitions of $\calE_\unif$ in \eqref{eq:est_err_of_V} and $\calE_\loco$ in \eqref{eq:loco_err_of_V}. Define
	$$
		\omega_n := \frac{d_1^2/\sigma_{\max}^2}{p} \cdot \calE_\unif^2 = \frac{r d_1^2 \sigma_{\min}^2\log n}{d_r^2 \sigma_{\max}^2}.
	$$
	Assume $\beta^2 \gg r, \frac{d_r}{\sigma_{\min}} \geq C_\gap (\sqrt{n} + \sqrt{p})$.
	In addition, assume either one of the following conditions hold:
	\begin{enumerate}
		\item either $\calE_\unif \ll 1$ and
		\begin{equation}
		\label{eq:ub_full_unif_bound_condition}
		\frac{\beta^2}{p} \gg \omega_n \lor \calE_\unif^2 \lor \frac{\calE_\unif}{\sigma_{\max}/\sigma_{\min}},
		\end{equation}
		\item or $\calE_\unif \lor \calE_\loco \ll 1$ and
		\begin{align*}
			\frac{\beta^2}{p}
			& \gg 
			\left( \omega_n \cdot \frac{\calE_\loco}{\calE_\unif}\right)
			\lor \left( \omega_n \cdot \frac{\calE_\loco^2}{\calE_\unif^2}\right)
			\lor \left( \omega_n \cdot \frac{1}{n}\right)
			\lor \left( \omega_n^{3/4} \cdot \calE_\loco \right)
			\lor \left( \omega_n^{2/3} \cdot \calE_\loco^{8/9} \cdot \frac{\sigma_{\min}^{8/9}}{\sigma_{\max}^{8/9}}\right)
			\lor \left( \omega_n^{3/5} \cdot \calE_\loco^{4/5} \cdot \frac{\sigma_{\min}^{4/5}}{\sigma_{\max}^{4/5}}\right)\\
			& \qquad
			\lor \left( \omega_n^{1/2} \cdot \calE_\unif\calE_\loco\right)
			\lor \left( \omega_n^{2/5} \cdot \calE_\unif^{4/5}\calE_\loco^{4/5} \cdot \frac{\sigma_{\min}^{8/5}}{\sigma_{\max}^{8/5}}\right)
			\lor \left( \omega_n^{1/3} \cdot \calE_\loco^{2/3}\cdot \frac{\sigma_{\min}^{4/3}}{\sigma_{\max}^{4/3}}\right)\\
			& \qquad
			\lor \left(\frac{\calE_\unif^2}{n}\right)
			\lor \left(\calE_\unif^{4/3} \calE_\loco^{2/3} \cdot \frac{\sigma_{\min}^{2/3}}{\sigma_{\max}^{2/3}}\right) 
			\lor \left(\calE_\unif^{6/5} \calE_\loco^{4/5} \cdot \frac{\sigma_{\min}^{4/5}}{\sigma_{\max}^{4/5}}\right)
			\lor \left(\calE_\unif \calE_\loco\right) \\
			& \qquad
			\lor \left(\frac{\calE_\unif}{\sqrt{n}} \cdot \frac{\sigma_{\min}}{\sigma_{\max}}\right)
			\lor \left(\calE_\unif^{6/7} \calE_\loco^{4/7} \cdot \frac{\sigma_{\min}^{6/7}}{\sigma_{\max}^{6/7}}\right)
			\lor \left(\calE_\unif^{1/2} \calE_\loco^{1/2} \cdot \frac{\sigma_{\min}}{\sigma_{\max}}\right) 
			\lor \left(\calE_\loco^2\right)
			\label{eq:ub_full_unif_and_loco_bound_condition}
			\lor \left(\calE_\loco \cdot \frac{\sigma_{\min}}{\sigma_{\max}}\right). \numberthis
		\end{align*}
	\end{enumerate}
	If for any $o(1)$ sequence $\delta_n$, \eqref{eq:ub_vanishing_err} holds,
	then there exists two sequences $\overline{\delta}_n, \overline{\delta}_n' = o(1)$ such that with probability $1-o(1)$, \eqref{eq:ub_mismatch_proportion} holds.
\end{theorem}

Theorem \ref{thm:ub_simplified} directly follows from Theorem \ref{thm:ub}.
We will only use condition \eqref{eq:ub_full_unif_and_loco_bound_condition}.
When $\sigma_\x \asymp \sigma_\y, d_1 \asymp d_r$, and $\calE_\unif \ll 1$, we have $\omega_n \asymp r \log n, \sigma_{\min}/\sigma_{\max}\asymp 1$, and  
$$
	\frac{\calE_\loco}{\calE_\unif} = \left(\frac{d_1}{d_r}\sqrt{\frac{\mu}{n}} + \frac{\sqrt{p} + \sqrt{\log n}}{\sqrt{r}d_r/\sigma_{\min}}\right) \left(\frac{1}{\sqrt{\log n}} + \frac{1}{\sqrt{p}}\right) 
	\lesssim \sqrt{\frac{\mu}{n}} \left(\frac{1}{\sqrt{\log n}} + \frac{1}{\sqrt{p}}\right) + o(1) \lesssim 1.
$$
Thus, \eqref{eq:ub_full_unif_and_loco_bound_condition} reduces to 
\begin{align*}
	\frac{\beta^2}{p}
	& \gg 
	\left( \omega_n \cdot \frac{\calE_\loco}{\calE_\unif}\right)
	\lor \left( \omega_n \cdot \frac{\calE_\loco^2}{\calE_\unif^2}\right)
	\lor \left( \omega_n \cdot \frac{1}{n}\right)
	\lor \left( \omega_n^{3/4} \cdot \calE_\loco \right)
	\lor \left( \omega_n^{2/3} \cdot \calE_\loco^{8/9} \right)
	\lor \left( \omega_n^{3/5} \cdot \calE_\loco^{4/5} \right)\\
	& \qquad
	\lor \left( \omega_n^{1/2} \cdot \calE_\unif\calE_\loco\right)
	\lor \left( \omega_n^{2/5} \cdot \calE_\unif^{4/5}\calE_\loco^{4/5} \right)
	\lor \left( \omega_n^{1/3} \cdot \calE_\loco^{2/3}\right)\\
	& \qquad
	\lor \left(\frac{\calE_\unif^2}{n}\right)
	\lor \left(\calE_\unif^{4/3} \calE_\loco^{2/3} \right) 
	\lor \left(\calE_\unif^{6/5} \calE_\loco^{4/5} \right)
	\lor \left(\calE_\unif \calE_\loco\right) \\
	& \qquad
	\lor \left(\frac{\calE_\unif}{\sqrt{n}} \right)
	\lor \left(\calE_\unif^{6/7} \calE_\loco^{4/7}\right)
	\lor \left(\calE_\unif^{1/2} \calE_\loco^{1/2} \right) 
	\lor \left(\calE_\loco^2\right)
	\lor \left(\calE_\loco \right)\\
	& \overset{(*)}{=}
	\left( \omega_n \cdot \frac{\calE_\loco}{\calE_\unif}\right)
	\lor \left( \omega_n \cdot \frac{\calE_\loco^2}{\calE_\unif^2}\right)
	\lor \left( \omega_n \cdot \frac{1}{n}\right)
	\lor \left( \omega_n^{3/4} \cdot \calE_\loco \right)
	\lor \left( \omega_n^{3/5} \cdot \calE_\loco^{4/5} \right)\\
	& \qquad
	\lor \left( \omega_n^{1/2} \cdot \calE_\unif\calE_\loco\right)
	\lor \left( \omega_n^{2/5} \cdot \calE_\unif^{4/5}\calE_\loco^{4/5} \right)
	\lor \left( \omega_n^{1/3} \cdot \calE_\loco^{2/3}\right)\\
	& \qquad
	\lor \left(\frac{\calE_\unif^2}{n}\right)
	\lor \left(\calE_\unif^{4/3} \calE_\loco^{2/3} \right) 
	\lor \left(\calE_\unif^{6/5} \calE_\loco^{4/5} \right)
	\lor \left(\calE_\unif \calE_\loco\right) \\
	& \qquad
	\lor \left(\frac{\calE_\unif}{\sqrt{n}} \right)
	\lor \left(\calE_\unif^{6/7} \calE_\loco^{4/7}\right)
	\lor \left(\calE_\unif^{1/2} \calE_\loco^{1/2} \right) 
	\lor \left(\calE_\loco^2\right)
	\lor \left(\calE_\loco \right)\\
	& \overset{(**)}{=}
	\left( \omega_n \cdot \frac{\calE_\loco}{\calE_\unif}\right)
	\lor \left( \omega_n \cdot \frac{\calE_\loco^2}{\calE_\unif^2}\right)
	\lor \left( \omega_n \cdot \frac{1}{n}\right)
	\lor \left( \omega_n^{3/5} \cdot \calE_\loco^{4/5} \right)\\
	& \qquad
	\lor \left( \omega_n^{1/3} \cdot \calE_\loco^{2/3}\right)\\
	& \qquad
	\lor \left(\frac{\calE_\unif}{\sqrt{n}} \right)
	\lor \left(\calE_\unif^{1/2} \calE_\loco^{1/2} \right) 
	\\
	& \overset{(***)}{=}
	\left( \omega_n \cdot \frac{\calE_\loco}{\calE_\unif}\right)
	\lor \left( \omega_n^{3/5} \cdot \calE_\loco^{4/5} \right)\\
	& \qquad
	\lor \left( \omega_n^{1/3} \cdot \calE_\loco^{2/3}\right)\\
	& \qquad
	\lor \left(\calE_\unif^{1/2} \calE_\loco^{1/2} \right)
\end{align*}
where equality $(*)$ is by 
\begin{align*}
\left( \omega_n^{3/4} \cdot \calE_\loco \right) \lor \left( \omega_n^{2/3} \cdot \calE_\loco^{8/9} \right)\lor \left( \omega_n^{3/5} \cdot \calE_\loco^{4/5} \right)
 & = \left( \omega_n^{3/4} \cdot \calE_\loco \right) \lor \left( \omega_n^{3/4} \cdot \calE_\loco \right)^{8/9} \lor \left( \omega_n^{3/4} \cdot \calE_\loco \right)^{4/5}
 \\
 & = \left( \omega_n^{3/4} \cdot \calE_\loco \right) \lor \left( \omega_n^{3/5} \cdot \calE_\loco^{4/5} \right),
\end{align*}
equality $(**)$ is by $\omega_n = r \log n \geq 1, \calE_\unif\lor \calE_\loco \ll 1$ and equality $(***)$ is by $\calE_\loco \lesssim \calE_\unif \ll 1$ and
$$
	\frac{\calE_\loco}{\calE_\unif} \geq \frac{\sqrt{\mu}}{\sqrt{n \log n}} \geq \frac{1}{n}.
$$
Thus, Theorem \ref{thm:ub_simplified} is proved.

In the rest of this section, we give a proof of Theorem \ref{thm:ub}.
Without loss of generality, we assume $\sigma_\x \leq \sigma_\y$, so that $\sigma_{\min} = \sigma_x$ and $\sigma_{\max} = \sigma_\y$.
Under this assumption, the columns of matrix $\hat V$ are the top $r$ right singular vectors of $X$. 
For any $k$ distinct indices $1 \leq i_1\neq i_2 \neq \cdots \neq i_k\leq n$, we let $\hat V^{(-i_{1:k})}\in O_{p, r}$ be the orthonormal whose columns are the top $r$ right singular vectors of $X_{-i_{1:k}, \bigcdot}$. 
For notational simplicity, we interpret $i_0=i_k, i_1 = i_{k+1}$ and so on.
We define the following ``globally good'' event
\begin{align*}
	E_\glob := \left\{ \|\hat V \hat V^\top - V V^\top\|_F \leq \xi \right\} \bigcap \left\{ \max_{k\leq k^\star}\max_{i_{1:k}}\|\hat V \hat V^\top - \hat V^{(-i_{1:k})} (\hat V^{(-i_{1:k})})^\top\|_F \leq \xi' \right\},
\end{align*}
where the exact values of $k^\star\in[n]$, $\xi=o(1)$ and $\xi'=o(1)$ will be determined later.
In the meantime, for any specific configuration of $i_{1:k}$, we define the following ``locally good'' event
\begin{align*}
	E_{i_{1:k}} := \left\{ \left\| \left( \frac{U_{i_{1:k}, \bigcdot} D V^\top }{\sigma_\x}  + (N_\x)_{i_{1:k}, \bigcdot}\right)^\top (I_k^\leftarrow - I_k) \left( \frac{U_{i_{1:k}, \bigcdot} D V^\top }{\sigma_\y}  + (N_\y)_{i_{1:k}, \bigcdot}\right) \right\|_F \leq \xi_k \right\},
\end{align*}
where the quantity $\xi_k$ depends on $k$ (but not the specific configuration of $i_{1:k}$). The exact value of $\xi_k$ will be determined later. 
We can compute the expected mismatch proportion by
\begin{align*}
	\bbE[\ell(\hat\Pi, \Pi^\star)] & = \frac{1}{n} \sum_{i\in[n]}  \bbE\left[d(\hat \pi, \pi^\star)\right] \\
	& = \frac{1}{n} \sum_{i\in[n]} \bbP\left(\hat\pi_i \neq \pi^\star_i\right) \\
	& \leq \frac{1}{n} \sum_{i\in[n]} \bbP \left( \hat \pi_i \neq \pi^\star_i \textnormal{ and } E_\glob \right) + \bbP(E_\glob^c)\\
	& \overset{(*)}{=} \left[\frac{1}{n}\sum_{k=2}^n \sum_{i_1 \neq \cdots \neq i_k} \bbP\left( \hat \pi_{i_1} = \pi^\star_{i_2}, \hat \pi_{i_2} = \pi^\star_{i_3}, \hdots, \hat\pi_{i_k} = \pi^\star_{i_1} \textnormal{ and } E_\glob \right)\right] + \bbP(E_\glob^c) \\
	\label{eq:ub_expected_mismatch}
	& \leq \left[\frac{1}{n}\sum_{k=2}^n \sum_{i_1 \neq \cdots \neq i_k} \left( \scrP_{i_{1:k}} + \bbP(E_{i_{1:k}}^c) \right)\right] + \bbP(E_\glob^c), \numberthis
\end{align*}
where $(*)$ is by Lemma \ref{lemma:cyc_decomp} and
$$
	\scrP_{i_{1:k}} := \bbP\left( \hat \pi_{i_1} = \pi^\star_{i_2}, \hat \pi_{i_2} = \pi^\star_{i_3}, \hdots, \hat\pi_{i_k} = \pi^\star_{i_1} \textnormal{ and } E_\glob\cap E_{i_{1:k}} \right).
$$
To upper bound $\scrP_{i_{1:k}}$, note that by construction,
$$
	\left\la X \hat V , \hat \Pi Y \hat V \right\ra \geq \left\la X \hat V , \Pi Y \hat V \right\ra 
$$
for any permutation matrix $\Pi$. The event $\{\hat \pi_{i_1} = \pi^\star_{i_2} , \hdots, \hat\pi_{i_k} = \pi^\star_{i_1}\}$ implies that
\begin{align*}
	& (X \hat V)_{i_1, \bigcdot}^\top  (Y \hat V)_{\pi^\star_{i_2}} 
	+ (X \hat V)_{i_2, \bigcdot}^\top  (Y \hat V)_{\pi^\star_{i_3}}
	+ \cdots 
	+ (X \hat V)_{i_k, \bigcdot}^\top  (Y \hat V)_{\pi^\star_{i_1}}
	+ \sum_{\ell \notin i_{1:k}} (X \hat V)_{\ell, \bigcdot}^\top  (Y \hat V)_{\hat\pi_{\ell}} \\
	& \geq 
	\sum_{\ell\in i_{1:k}} (X \hat V)_{\ell, \bigcdot}^\top  (Y \hat V)_{\pi_{\ell}}
	+ 
	\sum_{\ell\notin i_{1:k}} (X \hat V)_{\ell, \bigcdot}^\top  (Y \hat V)_{\pi_{\ell}},
	\qquad \forall \pi \in S_n.
\end{align*}
We choose $\pi$ such that $\pi_\ell = \pi^\star_\ell$ for any $\ell \in i_{1:k}$ and $\pi_\ell = \hat \pi_\ell$ for any $\ell \notin i_{1:k}$. Note that such a choice is feasible because under the event $\{\hat \pi_{i_1} = \pi^\star_{i_2} , \hdots, \hat\pi_{i_k} = \pi^\star_{i_1}\}$, we have  
$
	\{\hat\pi_\ell: \ell \in i_{1:k}\} = \{\pi^\star_\ell: \ell \in i_{1:k}\}.
$
Under this choice of $\pi$, the event $\{\hat \pi_{i_1} = \pi^\star_{i_2} , \hdots, \hat\pi_{i_k} = \pi^\star_{i_1}\}$ implies that
\begin{align*}
	& (X \hat V)_{i_1, \bigcdot}^\top  (Y \hat V)_{\pi^\star_{i_2}} 
	+ (X \hat V)_{i_2, \bigcdot}^\top  (Y \hat V)_{\pi^\star_{i_3}}
	+ \cdots 
	+ (X \hat V)_{i_k, \bigcdot}^\top  (Y \hat V)_{\pi^\star_{i_1}} \\
	& \geq 
	(X \hat V)_{i_1, \bigcdot}^\top  (Y \hat V)_{\pi^\star_{i_1}} 
	+ (X \hat V)_{i_2, \bigcdot}^\top  (Y \hat V)_{\pi^\star_{i_2}}
	+ \cdots 
	+ (X \hat V)_{i_k, \bigcdot}^\top  (Y \hat V)_{\pi^\star_{i_k}},
\end{align*}
or equivalently
$$
	\left\la (X \hat V)_{i_{1:k}, \bigcdot}, (I_k^\leftarrow - I_k) (Y\hat V)_{\pi^\star_{i_{1:k}},\bigcdot} \right\ra \geq 0 
	\iff
	\left\la \frac{(X \hat V)_{i_{1:k}, \bigcdot}}{\sigma_\x}, \frac{(I_k^\leftarrow - I_k) (Y\hat V)_{\pi^\star_{i_{1:k}},\bigcdot}}{\sigma_\y} \right\ra \geq 0 .
$$
As the result, we have
\begin{align*}
	\scrP_{i_{1:k}} \leq 
	\bbP\left( \left\la \frac{(X \hat V)_{i_{1:k}, \bigcdot}}{\sigma_\x}, \frac{(I_k^\leftarrow - I_k) (Y\hat V)_{\pi^\star_{i_{1:k}},\bigcdot}}{\sigma_\y} \right\ra \geq 0 \textnormal{ and } E_\glob \cap E_{i_{1:k}} \right).
\end{align*}
To proceed further, we need the following proposition.
\begin{proposition}
\label{prop:core_upper_bound}
Let $W, Z\in \bbR^{n\times r}$ be two independent random matrices whose entries are $W_{i,j}\overset{ind}{\sim} N(\mu_{i, j}, 1), Z_{i, j} \overset{ind}{\sim} N(\nu_{i, j}, 1)$, respectively. Let $\Theta = \diag(\theta_1, \hdots, \theta_r)$ be a diagonal matrix where $0 < \theta_j\leq 1$ for any $j\in[r]$ and let $\xi>0$ be an arbitrary constant. Suppose we can partition $[r]$ into two disjoint sets $J_1\cup J_2$, where
$$
	\theta_j \mu_{\bigcdot, j} = \nu_{\bigcdot, j} , \forall j\in J_1,
	\qquad
	\theta_j \nu_{\bigcdot, j} = \mu_{\bigcdot, j} , \forall j\in J_2.
$$
Assume
$$
	\tilde \beta := \left(\min_{i\neq i'} \|\mu_{i, \bigcdot} - \mu_{i', \bigcdot}\|^2\right) \land \left(\min_{i\neq i'} \|\nu_{i, \bigcdot} - \nu_{i', \bigcdot}\|^2\right) \gg r.
$$
There exists a sequence $\ep_n = o(1)$ such that for any $k$ distinct indices $1\leq i_1 \neq i_2 \neq \cdots \neq i_k\leq n$, we have
\begin{align*}
	& \bbP\left(\left\la (W)_{i_{1:k}, \bigcdot} \Theta, (I_k^\leftarrow - I_k)Z_{i_{1:k},\bigcdot}\right\ra + \xi \geq 0\right) \\
	& \leq 
	\exp\left\{ -(1-\ep_n) C_k \left(\sum_{j\in J_1} \|(I_k^\leftarrow - I_k)\nu_{i_{1:k}, j}\|^2 + \sum_{j\in J_2} \|(I_k^\leftarrow - I_k) \mu_{i_{1:k}, j}\|^2 \right) + C' \xi\right\},
\end{align*}
where $C_k$ is defined in \eqref{eq:Ck}
and $C'>0$ is another absolute constant. 
\end{proposition}
\begin{proof}
	See Appendix \ref{prf:prop:core_upper_bound}.
\end{proof}

In the following, we upper bound $\scrP_{i_{1:k}}$ using two methods. Method (a) is ``uniform'' in the sense that it directly uses the closeness between $\hat V$ and $V$. Method (b) uses ``leave-one-cycle-out'' arguments that takes advantage of the closeness between $\hat V$ and $\hat V^{(-i_{1:k})}$. 
\begin{enumerate}[label=(\alph*)]
	\item 
	We start by computing
	\begin{align*}
		& \left\la \frac{(X \hat V)_{i_{1:k}, \bigcdot}}{\sigma_\x}, \frac{(I_k^\leftarrow - I_k) (Y\hat V)_{\pi^\star_{i_{1:k}},\bigcdot}}{\sigma_\y} \right\ra\\
		& = \left\la \left( \frac{U_{i_{1:k},\bigcdot} D V^\top}{\sigma_\x} + (N_\x)_{i_{1:k},\bigcdot} \right)\hat V , (I_k^\leftarrow - I_k) \left(\frac{U_{i_{1:k},\bigcdot} D V^\top}{\sigma_\y} + (N_\x)_{i_{1:k},\bigcdot} \right)\hat V \right\ra \\
		& = \tr \left[
			\left(\frac{U_{i_{1:k},\bigcdot} D V^\top}{\sigma_\x} + (N_\x)_{i_{1:k},\bigcdot} \right)
			(\hat V \hat V^\top - VV^\top)
			\left(\frac{U_{i_{1:k},\bigcdot} D V^\top}{\sigma_\y} + (N_\x)_{i_{1:k},\bigcdot} \right)^\top (I_k^\leftarrow - I_k)^\top
		\right] \\
		& \qquad + 
		\left\la 
			\frac{U_{i_{1:k},\bigcdot} D}{\sigma_\x} + (N_\x)_{i_{1:k},\bigcdot} V, 
			(I_k^\leftarrow - I_k) 
			\left(
				\frac{U_{i_{1:k},\bigcdot} D}{\sigma_\y} + (N_\y)_{i_{1:k},\bigcdot} V
			\right)
		\right\ra\\
		& \leq \left\| \left(\frac{U_{i_{1:k},\bigcdot} D V^\top}{\sigma_\x} + (N_\x)_{i_{1:k},\bigcdot} \right)^\top (I_k^\leftarrow - I_k)  \left(\frac{U_{i_{1:k},\bigcdot} D V^\top}{\sigma_\y} + (N_\x)_{i_{1:k},\bigcdot} \right)\right\|_F \|\hat V \hat V^\top - VV^\top\|_F \\
		& \qquad + 
		\left\la 
			\frac{U_{i_{1:k},\bigcdot} D}{\sigma_\x} + (N_\x)_{i_{1:k},\bigcdot} V, 
			(I_k^\leftarrow - I_k) 
			\left(
				\frac{U_{i_{1:k},\bigcdot} D}{\sigma_\y} + (N_\y)_{i_{1:k},\bigcdot} V
			\right)
		\right\ra \\
		& \leq \xi \xi_k + 
		\left\la 
			\frac{U_{i_{1:k},\bigcdot} D}{\sigma_\x} + (N_\x)_{i_{1:k},\bigcdot} V, 
			(I_k^\leftarrow - I_k) 
			\left(
				\frac{U_{i_{1:k},\bigcdot} D}{\sigma_\y} + (N_\y)_{i_{1:k},\bigcdot} V
			\right)
		\right\ra,
	\end{align*}
	where the last inequality holds under $E_\glob \cap E_{i_{1:k}}$. Since $\beta \gg r$, we can invoke Proposition \ref{prop:core_upper_bound} with $\Theta = \frac{\sigma_{\min}}{\sigma_{\max}} I_k = \frac{\sigma_{\x}}{\sigma_{\y}} I_k$ to conclude that
	\begin{align*}
		\scrP_{i_{1:k}}
		& \leq 
		\bbP\left( \frac{\sigma_{\min}}{\sigma_{\max}}\left\la \frac{(X \hat V)_{i_{1:k}, \bigcdot}}{\sigma_\x}, \frac{(I_k^\leftarrow - I_k) (Y\hat V)_{\pi^\star_{i_{1:k}},\bigcdot}}{\sigma_\y} \right\ra + \frac{\sigma_{\min} \xi \xi_k}{\sigma_{\max}} \geq 0 \right) \\
		\label{eq:ub_uniform_bound}
		& \leq
		\exp\left\{ - \frac{(1-o(1))C_k \|(I_k^\leftarrow - I_k) U_{i_{1:k},\bigcdot} D\|_F^2}{\sigma_{\max}^2}  + \calO\left( \frac{\sigma_{\min}\xi\xi_k}{\sigma_{\max}}\right)\right\},\numberthis
	\end{align*}
	where $C_k$ is defined in \eqref{eq:Ck}.

	\item 
	For this method to work, we restrict ourselves to the case when $k\leq k^\star$.
	We begin by computing
	\begin{align*}
		& \left\la \frac{(X \hat V)_{i_{1:k}, \bigcdot}}{\sigma_\x}, \frac{(I_k^\leftarrow - I_k) (Y\hat V)_{\pi^\star_{i_{1:k}},\bigcdot}}{\sigma_\y} \right\ra\\
		& = \left\la \left( \frac{U_{i_{1:k},\bigcdot} D V^\top}{\sigma_\x} + (N_\x)_{i_{1:k},\bigcdot} \right)\hat V , (I_k^\leftarrow - I_k) \left(\frac{U_{i_{1:k},\bigcdot} D V^\top}{\sigma_\y} + (N_\x)_{i_{1:k},\bigcdot} \right)\hat V \right\ra \\
		& = \tr \left[
			\left(\frac{U_{i_{1:k},\bigcdot} D V^\top}{\sigma_\x} + (N_\x)_{i_{1:k},\bigcdot} \right)
			(\hat V \hat V^\top - \hat V^{(-i_{1:k})}(\hat V^{(-i_{1:k})})^\top)
			\left(\frac{U_{i_{1:k},\bigcdot} D V^\top}{\sigma_\y} + (N_\x)_{i_{1:k},\bigcdot} \right)^\top (I_k^\leftarrow - I_k)^\top
		\right] \\
		& \qquad + 
		\left\la 
			\frac{U_{i_{1:k},\bigcdot} D V^\top \hat V^{(-i_{1:k})}}{\sigma_\x} + (N_\x)_{i_{1:k},\bigcdot} \hat V^{(-i_{1:k})}, 
			(I_k^\leftarrow - I_k) 
			\left(
				\frac{U_{i_{1:k},\bigcdot} D V^\top \hat V^{(-i_{1:k})}}{\sigma_\y} + (N_\y)_{i_{1:k},\bigcdot} \hat V^{(-i_{1:k})}
			\right)
		\right\ra\\
		& \leq \left\| \left(\frac{U_{i_{1:k},\bigcdot} D V^\top}{\sigma_\x} + (N_\x)_{i_{1:k},\bigcdot} \right)^\top (I_k^\leftarrow - I_k)  \left(\frac{U_{i_{1:k},\bigcdot} D V^\top}{\sigma_\y} + (N_\x)_{i_{1:k},\bigcdot} \right)\right\|_F \|\hat V \hat V^\top - \hat V^{(-i_{1:k}}(V^{(-i_{1:k})})^\top\|_F \\
		& \qquad + 
		\left\la 
			\frac{U_{i_{1:k},\bigcdot} D V^\top \hat V^{(-i_{1:k})}}{\sigma_\x} + (N_\x)_{i_{1:k},\bigcdot} \hat V^{(-i_{1:k})}, 
			(I_k^\leftarrow - I_k) 
			\left(
				\frac{U_{i_{1:k},\bigcdot} D V^\top \hat V^{(-i_{1:k})}}{\sigma_\y} + (N_\y)_{i_{1:k},\bigcdot} \hat V^{(-i_{1:k})}
			\right)
		\right\ra\\
		& \leq \xi' \xi_k +
		\underbrace{
		\left\la 
			\frac{U_{i_{1:k},\bigcdot} D V^\top \hat V^{(-i_{1:k})}}{\sigma_\x} + (N_\x)_{i_{1:k},\bigcdot} \hat V^{(-i_{1:k})}, 
			(I_k^\leftarrow - I_k) 
			\left(
				\frac{U_{i_{1:k},\bigcdot} D V^\top \hat V^{(-i_{1:k})}}{\sigma_\y} + (N_\y)_{i_{1:k},\bigcdot} \hat V^{(-i_{1:k})}
			\right)
		\right\ra}_{\scrT_{i_{1:k}}},
	\end{align*}
	where the last inequality holds under $E_\glob\cap E_{i_{1:k}}$. Let
	$$
		E_{i_{1:k}}': = 
		\left\{ \|\hat V^{(-i_{1:k})} (\hat V^{(-i_{1:k})})^\top - V V^\top \|_F \leq \xi + \xi' \right\}
	$$
	be an event that is implied by $E_\glob$. We have
	\begin{align*}
		\scrP_{i_{1:k}} 
		& \leq \bbP\left(
			\frac{\sigma_{\min}\scrT_{i_{1:k}}}{\sigma_{\max}}
			+ \frac{\sigma_{\min} \xi' \xi_k}{\sigma_{\max}}  \geq 0
			\textnormal{ and } E_\glob \cap E_{i_{1:k}}
		\right)\\
		& = \bbP\left(
			\frac{\sigma_{\min}\scrT_{i_{1:k}}}{\sigma_{\max}}
			+ \frac{\sigma_{\min} \xi' \xi_k}{\sigma_{\max}}  \geq 0
			\textnormal{ and } E_\glob \cap E_{i_{1:k}} \cap E_{i_{1:k}}'
		\right)\\
		& \leq \bbP\left(
			\frac{\sigma_{\min}\scrT_{i_{1:k}}}{\sigma_{\max}}
			+ \frac{\sigma_{\min} \xi' \xi_k}{\sigma_{\max}}  \geq 0
			\textnormal{ and } E_{i_{1:k}}'
		\right) \\
		& = \bbE\left[
		\bbP\left(
			\frac{\sigma_{\min}\scrT_{i_{1:k}}}{\sigma_{\max}}
			+ \frac{\sigma_{\min} \xi' \xi_k}{\sigma_{\max}}  \geq 0 
			~ \bigg| ~ (N_{\x})_{-i_{1:k},\bigcdot}, (N_\y)_{-{1:k},\bigcdot}
		\right) 
		\cdot \Indc_{E_{i_{1:k}}'}
		\right].
	\end{align*}
	Conditional on the randomness in $(N_{\x})_{-i_{1:k},\bigcdot}$ and $(N_\y)_{-i_{1:k},\bigcdot}$, the matrix $\hat V^{(-i_{1:k})}$ become a fixed matrix, and thus we can apply Proposition \ref{prop:core_upper_bound} to conclude that 
	\begin{align*}
		& \bbP\left(
			\frac{\sigma_{\min}\scrT_{i_{1:k}}}{\sigma_{\max}}
			+ \frac{\sigma_{\min} \xi' \xi_k}{\sigma_{\max}}  \geq 0 
			~ \bigg| ~ (N_{\x})_{-i_{1:k},\bigcdot}, (N_\y)_{-{1:k},\bigcdot}
		\right) \\
		& \leq
		\exp\left\{ - \frac{(1-o(1))C_k \|(I_k^\leftarrow - I_k) U_{i_{1:k},\bigcdot} D V^\top \hat V^{(-i_{1:k})}\|_F^2}{\sigma_{\max}^2}  + \calO\left( \frac{\sigma_{\min}\xi'\xi_k}{\sigma_{\max}}\right)\right\},
	\end{align*}
	provided 
	$$
		\tilde \beta^2 = \min_{i\neq i'} \frac{\|(\hat V^{(-i_{1:k})})^\top V D (U_{i, \bigcdot} - U_{i', \bigcdot})^\top\|^2}{\sigma_{\max}^2} \gg r.
	$$
	Recall that we will choose $\xi, \xi' = o(1)$, and thus under $E'_{i_{1:k}}$, we have
	$
		\|V^\top \hat V^{(-i_{1:k})}\| = 1+o(1).
	$
	This means that under $E'_{i_{1:k}}$, we have $\tilde \beta^2 = (1+o(1))\beta^2 \gg r$ and
	\begin{align*}
		& \bbP\left(
			\frac{\sigma_{\min}\scrT_{i_{1:k}}}{\sigma_{\max}}
			+ \frac{\sigma_{\min} \xi' \xi_k}{\sigma_{\max}}  \geq 0 
			~ \bigg| ~ (N_{\x})_{-i_{1:k},\bigcdot}, (N_\y)_{-{1:k},\bigcdot}
		\right) \\
		& \leq
		\exp\left\{ - \frac{(1-o(1))C_k \|(I_k^\leftarrow - I_k) U_{i_{1:k},\bigcdot} D\|_F^2}{\sigma_{\max}^2}  + \calO\left( \frac{\sigma_{\min}\xi'\xi_k}{\sigma_{\max}}\right)\right\}.
	\end{align*}
	Consequently, we get
	\begin{align}
	\label{eq:ub_loco}
		\scrP_{i_{1:k}}
		\leq
		\exp\left\{ - \frac{(1-o(1))C_k \|(I_k^\leftarrow - I_k) U_{i_{1:k},\bigcdot} D\|_F^2}{\sigma_{\max}^2}  + \calO\left( \frac{\sigma_{\min}\xi'\xi_k}{\sigma_{\max}}\right)\right\}.
	\end{align}
\end{enumerate}

Now that we have upper bounded $\scrP_{i_{1:k}}$, we proceed to upper bound $\bbP(E_{i_{1:k}}^c)$.
The following proposition is useful for this purpose.

\begin{proposition}
\label{prop:frob_norm_concentration}
Fix $2\leq k\leq n$, $1\leq i_1 \neq i_2 \neq \cdots \neq i_k \leq n$, and $\delta \in(0,1)$. Then with probability at least $1-5\delta$, we have
\begin{align*}
	& \frac{\|(X)_{i_{1:k},\bigcdot}^\top (I_k^\leftarrow - I_k)(Y)_{\pi^\star_{i_{1:k}},\bigcdot}\|_F}{\sigma_\x \sigma_\y}  \\
	& \lesssim \frac{\|D U_{i_{1:k},\bigcdot}^\top (I_k^\leftarrow - I_k) U_{i_{1:k},\bigcdot} D \|_F}{\sigma_\x\sigma_\y}  \\
	& \qquad +  \bigg(\frac{1}{\sigma_\x} + \frac{1}{\sigma_\y}\bigg) 
	\bigg(
		p^{1/2} \|(I_k^\leftarrow - I_k)U_{i_{1:k},\bigcdot} D\|_F
		+ p^{1/4} \|(I_k^\leftarrow - I_k)U_{i_{1:k},\bigcdot}D^2U_{i_{1:k},\bigcdot}^\top(I_k^\leftarrow-I_k)^\top\|_F^{1/2} (\log(1/\delta))^{1/4}\\
		& \qquad \qquad 
		+ \|(I_k^\leftarrow-I_k)U_{i_{1:k},\bigcdot}D^2U_{i_{1:k},\bigcdot}^\top(I_k^\leftarrow-I_k)^\top\|_2^{1/2} (\log(1/\delta))^{1/2}
	\bigg)\\
	& \qquad + 
	\bigg( 
		p k^{1/2} + p^{3/4}k^{1/4}(\log(1/\delta))^{1/4} + p^{5/8}k^{1/8}(\log(1/\delta))^{3/8} \\
		& \qquad \qquad + p^{1/2}k^{1/2}(\log(1/\delta))^{1/4} 
		+ p^{1/2}k^{1/4}(\log(1/\delta))^{1/2} 
		+ p^{3/8}k^{3/8}(\log(1/\delta))^{3/8} + p^{3/8}k^{1/8}(\log(1/\delta))^{5/8} \\
		\label{eq:frob_norm_concentration}
		& \qquad \qquad + p^{1/4}k^{1/4}(\log(1/\delta))^{1/2} + p^{1/4} (\log(1/\delta))^{3/4} 
		+ k^{1/2}(\log(1/\delta))^{1/2} + \log(1/\delta)
	\bigg).\numberthis
\end{align*}
\end{proposition}
\begin{proof}
	See Appendix \ref{prf:prop:frob_norm_concentration}.
\end{proof}

We let $\xi_k(\delta)$ be the right-hand side of \eqref{eq:frob_norm_concentration}. If we take
$$
	\delta = \delta_{i_{1:k}} = \exp\left\{ - \frac{(1-o(1))C_k \|(I_k^\leftarrow - I_k) U_{i_{1:k},\bigcdot} D\|_F^2}{\sigma_{\max}^2}\right\},
$$
then by \eqref{eq:ub_uniform_bound} and \eqref{eq:ub_loco}, we have
\begin{align*}
	\scrP_{i_{1:k}} + \bbP(E_{i_{1:k}}^c) & \lesssim \exp\left\{ - \frac{(1-o(1))C_k \|(I_k^\leftarrow - I_k) U_{i_{1:k},\bigcdot} D\|_F^2}{\sigma_{\max}^2}  + \calO\left( \frac{\sigma_{\min}\xi\xi_k(\delta_{i_{1:k}})}{\sigma_{\max}}\right)\right\} \qquad \forall k \geq k^\star,\\
	\scrP_{i_{1:k}} + \bbP(E_{i_{1:k}}^c) & \lesssim \exp\left\{ - \frac{(1-o(1))C_k \|(I_k^\leftarrow - I_k) U_{i_{1:k},\bigcdot} D\|_F^2}{\sigma_{\max}^2}  + \calO\left( \frac{\sigma_{\min}\xi'\xi_k(\delta_{i_{1:k}})}{\sigma_{\max}}\right)\right\} \qquad \forall k\geq k^\star.
\end{align*}
The following lemma specifies a sufficient condition under which the second term on the exponent becomes negligible. 
\begin{lemma}
\label{lemma:negligible_exponent_uniform_bound}
If $\beta^2 \geq 1$ and
\begin{align*}
	\xi 
	& \ll
	\left\{(k^\star)^{1/2} \left[ \frac{\beta}{d_1/\sigma_{\max}} \land \left(\frac{\beta^2}{p}\right)^{1/2} \land \frac{\sigma_{\max}\beta^2}{\sigma_{\min}p} \land \frac{\sigma_{\max}}{\sigma_{\min}}\left(\frac{\beta^2}{p}\right)^{5/8} \right]\right\} \\
	& \qquad \land \left\{(k^\star)^{1/4} \left[ \left(\frac{\beta^2}{p}\right)^{1/4}  \land \frac{\sigma_{\max}}{\sigma_{\min}} \left(\frac{\beta^2}{p}\right)^{1/2} \land \frac{\sigma_{\max}}{\sigma_{\min}}\left(\frac{\beta^2}{p}\right)^{3/8} \right]  \right\}
	\land 1.
\end{align*}
then we have
$$
	\max_{k\geq k^\star}\max_{i_{1:k}} \left\{\frac{\sigma_{\min}\xi\xi_k(\delta_{i_{1:k}})}{\sigma_{\max}} \bigg/ \frac{\|(I_k^\leftarrow - I_k) U_{i_{1:k},\bigcdot} D\|_F^2}{\sigma_{\max}^2}\right\} = o(1).
$$
\end{lemma}
\begin{proof}
	See Appendix \ref{prf:lemma:negligible_exponent_uniform_bound}.
\end{proof}

In view of Proposition \ref{prop:est_err_of_V}, we choose
$$
	\xi = C'_1\cdot\calE_\unif = C'_1\cdot\frac{\sqrt{rp\log n}}{d_r/\sigma_{\min}}
$$
for some absolute constant $C'_1>0$,
so that
$$
	\bbP\left( \|\hat V \hat V^\top - VV^\top\|_F \leq \xi \right) \geq 1 - n^{-(1+c_1)}
$$
for another absolute constant $c_1 > 0$.
Invoking Lemma \ref{lemma:negligible_exponent_uniform_bound}, we know that as long as
\begin{equation}
	\label{eq:ub_prf_eigengap_assump}
	\frac{d_r}{\sigma_{\min}} \geq C_\gap (\sqrt{n} + \sqrt{p})
\end{equation}	
and
\begin{align*}
	\calE_\unif
	& \ll
	\left\{(k^\star)^{1/2} \left[ \frac{\beta}{d_1/\sigma_{\max}} \land \left(\frac{\beta^2}{p}\right)^{1/2} \land \frac{\sigma_{\max}\beta^2}{\sigma_{\min}p} \land \frac{\sigma_{\max}}{\sigma_{\min}}\left(\frac{\beta^2}{p}\right)^{5/8} \right]\right\} \\
	\label{eq:ub_prf_uniform_bound_assump}
	& \qquad \land \left\{(k^\star)^{1/4} \left[ \left(\frac{\beta^2}{p}\right)^{1/4}  \land \frac{\sigma_{\max}}{\sigma_{\min}} \left(\frac{\beta^2}{p}\right)^{1/2} \land \frac{\sigma_{\max}}{\sigma_{\min}}\left(\frac{\beta^2}{p}\right)^{3/8} \right]  \right\}
	\land 1 \numberthis
\end{align*}
for any $i_{1:k}$ satisfying $k\geq k^\star$,
we have
\begin{align}
	\label{eq:ub_prf_core_ineq}
	\scrP_{i_{1:k}} + \bbP(E_{i_{1:k}}^c) \lesssim \exp\left\{ - \frac{(1-o(1))C_k \|(I_k^\leftarrow - I_k) U_{i_{1:k},\bigcdot} D\|_F^2}{\sigma_{\max}^2} \right\}.
\end{align}
In view of Proposition \ref{prop:loco_err_of_V}, we choose
$$
	\xi' = C_2' k^\star \cdot \calE_\loco = C_2' k^\star \cdot \left(\frac{d_1^2 \mu r}{d_r^2 n} + \frac{p + \sqrt{p\log n} + \log n}{d_r^2/\sigma^2_{\min}}\right) 
$$
for some absolute constant $C_2' > 0$, so that
$$
	\bbP\left(\max_{k\leq k^\star}\max_{i_{1:k}}\|\hat V \hat V^\top - \hat V^{(-i_{1:k})} (\hat V^{(-i_{1:k})})^\top\|_F \leq \xi' \right) \geq 1- n^{-(1+c_2)}
$$
for another absolute constant $c_2>0$. Invoking Lemma \ref{lemma:negligible_exponent_uniform_bound} again, we conclude that if \eqref{eq:ub_prf_eigengap_assump} holds and 
\begin{align*}
	k^\star \cdot \calE_\loco
	& \ll
	\left\{(k^\star)^{1/2} \left[ \frac{\beta}{d_1/\sigma_{\max}} \land \left(\frac{\beta^2}{p}\right)^{1/2} \land \frac{\sigma_{\max}\beta^2}{\sigma_{\min}p} \land \frac{\sigma_{\max}}{\sigma_{\min}}\left(\frac{\beta^2}{p}\right)^{5/8} \right]\right\} \\
	& \qquad \land \left\{(k^\star)^{1/4} \left[ \left(\frac{\beta^2}{p}\right)^{1/4}  \land \frac{\sigma_{\max}}{\sigma_{\min}} \left(\frac{\beta^2}{p}\right)^{1/2} \land \frac{\sigma_{\max}}{\sigma_{\min}}\left(\frac{\beta^2}{p}\right)^{3/8} \right]  \right\}
	\label{eq:ub_prf_loco_bound_assump}
	\land 1 \numberthis
\end{align*}
then for any $i_{1:k}$ satisfying $k\leq k^\star$, the inequality \eqref{eq:ub_prf_core_ineq} also holds. Moreover, we have
$
	\bbP(E_\glob^c) \leq n^{-(1+c_3)}
$
where $c_3>0$ is some absolute constant.

Note that if we choose $k^\star = 1$, then there is no need to use leave-one-cycle-out arguments. The following lemma gives a sufficient condition when choosing $k^\star = 1$ suffices for the proof. 
\begin{lemma}
	\label{lemma:ub_uniform_bound_only_condition}
	Assume 
	$\calE_\unif \ll 1$, then $\eqref{eq:ub_prf_uniform_bound_assump}$ with $k^\star = 1$ is equivalent to
	$$
	\frac{\beta^2}{p} \gg \omega_n \lor \calE_\unif^2 \lor \frac{\calE_\unif}{\sigma_{\max}/\sigma_{\min}},
	$$
	where we recall that
	$$
		\omega_n = \frac{d_1^2/\sigma_{\max}^2}{p} \cdot \calE_\unif^2 = \frac{r d_1^2 \sigma_{\min}^2\log n}{d_r^2 \sigma_{\max}^2}.
	$$
\end{lemma}
\begin{proof}
Plugging $k^\star =1$ into \eqref{eq:ub_prf_uniform_bound_assump}, we get 
\begin{align*}
	\calE_\unif \ll 
	& \left[ \frac{\beta}{d_1/\sigma_{\max}} \land \left(\frac{\beta^2}{p}\right)^{1/2} \land \frac{\sigma_{\max}\beta^2}{\sigma_{\min}p} \land \frac{\sigma_{\max}}{\sigma_{\min}}\left(\frac{\beta^2}{p}\right)^{5/8} \right] \\
	& \qquad \land \left[ \left(\frac{\beta^2}{p}\right)^{1/4}  \land \frac{\sigma_{\max}}{\sigma_{\min}} \left(\frac{\beta^2}{p}\right)^{1/2} \land \frac{\sigma_{\max}}{\sigma_{\min}}\left(\frac{\beta^2}{p}\right)^{3/8} \right]
	\land 1 \\\
	& = \frac{\beta}{d_1/\sigma_{\max}} \land \left[ \left(\frac{\beta^2}{p}\right)^{1/2} \land \left(\frac{\beta^2}{p}\right)^{1/4} \land \left(\frac{\beta^2}{p}\right)^0\right]  \\
	& \qquad \land  \left\{ \frac{\sigma_{\max}}{\sigma_{\min}} \cdot \left[ \left(\frac{\beta^2}{p}\right)^{1} \land \left(\frac{\beta^2}{p}\right)^{5/8} \land \left(\frac{\beta^2}{p}\right)^{1/2} \land \left(\frac{\beta^2}{p}\right)^{3/8}\right] \right\} \\
	& = \frac{\beta}{d_1/\sigma_{\max}}\land \left(\frac{\beta^2}{p}\right)^{1/2} \land 1
	\land \frac{\sigma_{\max} \beta^2}{\sigma_{\min}p } \land \frac{\sigma_{\max}}{\sigma_{\min}} \left(\frac{\beta^2}{p}\right)^{3/8}
\end{align*}
and $\calE_\unif \ll 1$. As along as $\calE_\unif \ll 1$ holds, the above display becomes
\begin{align*}
	\frac{\beta^2}{p} & \gg \left(\frac{\calE_{\unif} d_1 /\sigma_{\max}}{\sqrt{p}}\right)^2 \lor \calE_\unif^2 \lor \frac{\calE_\unif}{\sigma_{\max}/\sigma_{\min}} \lor \left(\frac{\calE_\unif}{\sigma_{\max}/\sigma_{\min}}\right)^{8/3} \\
	& = \left(\frac{\calE_{\unif} d_1 /\sigma_{\max}}{\sqrt{p}}\right)^2 \lor \calE_\unif^2 \lor \frac{\calE_\unif}{\sigma_{\max}/\sigma_{\min}}.
\end{align*}
The proof is finished by noting that
$
	\left(\frac{\calE_\unif d_1/\sigma_{\max}}{\sqrt{p}}\right)^2 = \frac{r d_1^2 \sigma_{\min}^2\log n}{d_r^2 \sigma_{\max}^2 }.
$
\end{proof}
The following lemma gives a sufficient condition when it is possible to choose $k^\star\in[n]$ such that both \eqref{eq:ub_prf_uniform_bound_assump} and \eqref{eq:ub_prf_loco_bound_assump} hold. 
\begin{lemma}
	\label{lemma:ub_uniform_and_loco_condition}
	Assume $\calE_\loco \lor \calE_\unif \ll 1$ and
	\begin{align*}
	\frac{\beta^2}{p}
	& \gg 
	\left( \omega_n \cdot \frac{\calE_\loco}{\calE_\unif}\right)
	\lor \left( \omega_n \cdot \frac{\calE_\loco^2}{\calE_\unif^2}\right)
	\lor \left( \omega_n \cdot \frac{1}{n}\right)
	\lor \left( \omega_n^{3/4} \cdot \calE_\loco \right)
	\lor \left( \omega_n^{2/3} \cdot \calE_\loco^{8/9} \cdot \frac{\sigma_{\min}^{8/9}}{\sigma_{\max}^{8/9}}\right)
	\lor \left( \omega_n^{3/5} \cdot \calE_\loco^{4/5} \cdot \frac{\sigma_{\min}^{4/5}}{\sigma_{\max}^{4/5}}\right)\\
	& \qquad
	\lor \left( \omega_n^{1/2} \cdot \calE_\unif\calE_\loco\right)
	\lor \left( \omega_n^{2/5} \cdot \calE_\unif^{4/5}\calE_\loco^{4/5} \cdot \frac{\sigma_{\min}^{8/5}}{\sigma_{\max}^{8/5}}\right)
	\lor \left( \omega_n^{1/3} \cdot \calE_\loco^{2/3}\cdot \frac{\sigma_{\min}^{4/3}}{\sigma_{\max}^{4/3}}\right)\\
	& \qquad
	\lor \left(\frac{\calE_\unif^2}{n}\right)
	\lor \left(\calE_\unif^{4/3} \calE_\loco^{2/3} \cdot \frac{\sigma_{\min}^{2/3}}{\sigma_{\max}^{2/3}}\right) 
	\lor \left(\calE_\unif^{6/5} \calE_\loco^{4/5} \cdot \frac{\sigma_{\min}^{4/5}}{\sigma_{\max}^{4/5}}\right)
	\lor \left(\calE_\unif \calE_\loco\right) \\
	& \qquad
	\lor \left(\frac{\calE_\unif}{\sqrt{n}} \cdot \frac{\sigma_{\min}}{\sigma_{\max}}\right)
	\lor \left(\calE_\unif^{6/7} \calE_\loco^{4/7} \cdot \frac{\sigma_{\min}^{6/7}}{\sigma_{\max}^{6/7}}\right)
	\lor \left(\calE_\unif^{1/2} \calE_\loco^{1/2} \cdot \frac{\sigma_{\min}}{\sigma_{\max}}\right) 
	\lor \left(\calE_\loco^2\right)
	\lor \left(\calE_\loco \cdot \frac{\sigma_{\min}}{\sigma_{\max}}\right),
\end{align*}
	where
	$$
		\omega_n = \frac{d_1^2/\sigma_{\max}^2}{p} \cdot \calE_\unif^2 = \frac{r d_1^2 \sigma_{\min}^2\log n}{d_r^2 \sigma_{\max}^2}.
	$$
	Then there exists a choice of $k^\star \in [n]$ such that both \eqref{eq:ub_prf_uniform_bound_assump} and \eqref{eq:ub_prf_loco_bound_assump} hold.
\end{lemma}
\begin{proof}
	See Appendix \ref{prf:lemma:ub_uniform_and_loco_condition}.
\end{proof}

By Lemmas \ref{lemma:ub_uniform_bound_only_condition} and \ref{lemma:ub_uniform_and_loco_condition}, we know that under the assumptions imposed by this theorem, \eqref{eq:ub_prf_core_ineq} holds for any $i_{1:k}$ and $\bbP(E_{\glob}) \leq n^{-(1+c_3)}$. Recalling \eqref{eq:ub_expected_mismatch} and $\beta^2 \gg 1$, we get
\begin{align}
	\label{eq:ub_exponential_plus_polynomial}
	\bbE[\ell(\hat\Pi, \Pi^\star)]
	& \leq \left[\frac{1}{n}\sum_{k=2}^n \sum_{i_1 \neq \cdots \neq i_k} \exp\left\{ - \frac{(1-\overline{\delta}_n)C_k \|(I_k^\leftarrow - I_k) U_{i_{1:k},\bigcdot} D\|_F^2}{\sigma_{\max}^2} \right\}\right] + n^{-(1+c_3)},
\end{align}
where $\overline{\delta}_n = o(1)$. 

Let us choose
\begin{align*}
	\overline{\delta}_n' = \bigg[\log\bigg(\frac{1}{\frac{1}{n} \sum_{k=2}^n \sum_{i_1\neq \cdots\neq i_k} \exp\big\{{-(1-\overline{\delta}_n)C_k\|(I_k^\leftarrow -I_k)U_{i_{1:k},\bigcdot} D\|_F^2}/{(\sigma_{\max}^2)}\big\} }\bigg)\bigg]^{-1/2} = o(1),
\end{align*}
where the last inequality is by \eqref{eq:ub_vanishing_err}.
Invoking Markov's inequality, we have
\begin{align*}
	& \bbP\left( \ell(\hat\Pi, \Pi^\star) \geq \bigg(\frac{1}{n} \sum_{k=2}^n \sum_{i_1\neq \cdots\neq i_k} \exp\big\{{-(1-\overline{\delta}_n)C_k\|(I_k^\leftarrow -I_k)U_{i_{1:k},\bigcdot} D\|_F^2}/{\sigma_{\max}^2}\big\} \bigg)^{1-\overline{\delta}_n'} \right)\\
	& \leq \bigg(\frac{1}{n} \sum_{k=2}^n \sum_{i_1\neq \cdots\neq i_k} \exp\big\{{-(1-\overline{\delta}_n)C_k\|(I_k^\rightarrow -I_k)U_{i_{1:k},\bigcdot} D\|_F^2}/{\sigma_{\max}^2}\big\}\bigg)^{\overline{\delta}_n'} \\
	& \qquad + n^{-(1+c_3)} \cdot \bigg(\frac{1}{n} \sum_{k=2}^n \sum_{i_1\neq \cdots\neq i_k} \exp\big\{{-(1-\overline{\delta}_n)C_k\|(I_k^\leftarrow -I_k)U_{i_{1:k},\bigcdot} D\|_F^2}/{\sigma_{\max}}\big\}\bigg)^{\overline{\delta}_n'-1}.
\end{align*}
If 
$$
	\bigg(\frac{1}{n} \sum_{k=2}^n \sum_{i_1\neq \cdots\neq i_k} \exp\big\{{-(1-\overline{\delta}_n)C_k\|(I_k^\leftarrow -I_k)U_{i_{1:k},\bigcdot} D\|_F^2}/{\sigma_{\max}^2}\big\}\bigg)^{1-\overline{\delta}_n'} \geq n^{-(1+\frac{c_3}{2})},
$$
then we have
\begin{align*}
	& \bbP\left( \ell(\hat\Pi, \Pi^\star)  \geq \bigg(\frac{1}{n} \sum_{k=2}^n \sum_{i_1\neq \cdots\neq i_k} \exp\big\{{-(1-\overline{\delta}_n)C_k\|(I_k^\leftarrow -I_k)U_{i_{1:k},\bigcdot} D\|_F^2}/{\sigma_{\max}^2}\big\} \bigg)^{1-\overline{\delta}_n'} \right)\\
	& \leq \exp\bigg\{- \bigg[\frac{1}{\frac{1}{n} \sum_{k=2}^n \sum_{i_1\neq \cdots\neq i_k} \exp\big\{{-(1-\overline{\delta}_n)C_k\|(I_k^\leftarrow -I_k)U_{i_{1:k},\bigcdot} D\|_F^2}/{\sigma_{\max}^2}\big\}}\bigg]^{1/2}\bigg\} + n^{-c_3/2}\\
	& = o(1),
\end{align*}
where the last inequality is by \eqref{eq:ub_vanishing_err} and $n^{-c_3}=o(1)$.
Otherwise, we proceed by
\begin{align*}
	& \bbP\left(\ell(\hat\Pi, \Pi^\star) \geq \bigg(\frac{1}{n} \sum_{k=2}^n \sum_{i_1\neq \cdots\neq i_k} \exp\big\{{-(1-\overline{\delta}_n)C_k\|(I_k^\leftarrow -I_k)U_{i_{1:k},\bigcdot} D\|_F^2}/{\sigma_{\max}^2}\big\} \bigg)^{1-\overline{\delta}_n'} \right)\\
	& \leq \bbP\left(\ell(\hat\Pi, \Pi^\star)  > 0\right) \\
	& \leq \sum_{i\in[n]} \bbP(\hat\pi_i\neq \pi^\star_i) \\
	& = \bbE[d(\hat \pi, \pi^\star)] \\
	& \leq n \cdot \frac{1}{n} \sum_{k=2}^n \sum_{i_1\neq \cdots\neq i_k} \exp\big\{{-(1-\overline{\delta}_n)C_k\|(I_k^\leftarrow -I_k)U_{i_{1:k},\bigcdot} D\|_F^2}/{\sigma_{\max}^2}\big\} + n^{-c_3} \\
	& = n\exp\bigg\{-\log \bigg(\frac{1}{\frac{1}{n} \sum_{k=2}^n \sum_{i_1\neq \cdots\neq i_k} \exp\big\{{-(1-\overline{\delta}_n)C_k\|(I_k^\leftarrow-I_k)U_{i_{1:k},\bigcdot} D\|_F^2}/{\sigma_{\max}^2}\big\}}\bigg)\bigg\} + n^{-c_3} \\
	& \leq n\exp\bigg\{-(1-\overline{\delta}_n')\log \bigg(\frac{1}{\frac{1}{n} \sum_{k=2}^n \sum_{i_1\neq \cdots\neq i_k} \exp\big\{{-(1-\overline{\delta}_n)C_k\|(I_k^\leftarrow -I_k)U_{i_{1:k},\bigcdot} D\|_F^2}/{\sigma_{\max}^2}\big\}}\bigg)\bigg\} \\
	& \qquad + n^{-c_3} \\
	& < n \cdot n^{-(1+\frac{c_3}{2})} + n^{-c_3},
\end{align*}
where the last inequality is by
$$
	\bigg(\frac{1}{n} \sum_{k=2}^n \sum_{i_1\neq \cdots\neq i_k} \exp\big\{{-(1-\overline{\delta}_n)C_k\|(I_k^\leftarrow -I_k)U_{i_{1:k},\bigcdot} D\|_F^2}/{\sigma_{\max}^2}\big\}\bigg)^{1 - \overline{\delta}_n'} < n^{-(1+\frac{c_3}{2})}.
$$
Since $n^{-c_3} = o(1)$, we again get
$$
	\bbP\left( \ell(\hat\Pi, \Pi^\star)  \geq \bigg(\frac{1}{n} \sum_{k=2}^n \sum_{i_1\neq \cdots\neq i_k} \exp\big\{{-(1-\overline{\delta}_n)C_k\|(I_k^\leftarrow -I_k)U_{i_{1:k},\bigcdot} D\|_F^2}/{\sigma_{\max}^2}\big\} \bigg)^{1-\overline{\delta}_n'} \right) = o(1).
$$
Summarizing the two cases above gives the desired result.

\subsubsection{Proof of Proposition \ref{prop:core_upper_bound}} \label{prf:prop:core_upper_bound}
Let us fix a specific collection of $i_1\neq i_2\neq \cdots i_k$. For notational simplicity, we let $W_j := W_{i_{1:k}, j}\sim N(\mu_j, I_k)$ and $Z_j = Z_{i_{1:k}, j} \sim N(\nu_j, I_k)$. For $\ell \in [k]$, we let the $\ell$-th entry of $W_{j}$ be $W_{j, \ell}$ and so on.
We also adopt the convention that $W_{j, \ell+mk} = W_{j, \ell}$ for any $m\in \bbN$. In particular, $W_{j, 0} = W_{j, k}$.
By Markov's inequality, for any $t \geq 0$, we have
\begin{align*}
	\bbP\left(\left\la (W)_{i_{1:k}, \bigcdot} \Theta, (I_k^\leftarrow - I_k)Z_{i_{1:k},\bigcdot}\right\ra + \xi \geq 0\right) 
	& = \bbP\left(\sum_{j\in[r]} \theta_j W_j^\top(I_k^\leftarrow - I_k) Z_j + \xi \geq 0\right)\\
	\label{eq:core_upper_bound_markov}
	& \leq e^{t\xi} \prod_{j\in[r]} \bbE\left[ \exp\left\{ t \theta_j W_j^\top (I_k^\leftarrow - I_k) Z_j \right\} \right]. \numberthis
\end{align*}
We break the proof into two cases according to whether $k\leq 5$ or not.

\paragraph{Case A. $\boldsymbol{k\leq 5}$.}
We first consider $j\in J_1$, so that $\theta_j \mu_j = \nu_j$.
Note that
\begin{align*}
	& \bbE\left[ \exp\left\{ t \theta_j W_j^\top (I_k^\leftarrow - I_k) Z_j \right\} \right]\\
	& = \bbE\left[ \prod_{\ell\in[k]} \exp\left\{ t\theta_j W_{j,\ell} (Z_{j,\ell+1} - Z_{j, \ell}) \right\} \right] \\
	& = \bbE \left[ \prod_{\ell\in[k]} \exp\left\{ t\theta_j \mu_{j\ell} (Z_{j, \ell+1} - Z_{j, \ell}) + \frac{1}{2} t^2\theta_j^2 (Z_{j, \ell+1} - Z_{j, \ell})^2 \right\} \right] \\
	& = (2\pi)^{-k/2} {\int}_{\bbR^k} \exp\left\{ -\frac{1}{2} \sum_{\ell\in[k]}(z_\ell - \nu_{j,\ell})^2 + t\theta_j \sum_{\ell\in[k]} \mu_{j, \ell}(z_{\ell+1} - z_\ell) + \frac{1}{2}t^2\theta_j^2 \sum_{\ell\in[k]} (z_{\ell+1} - z_\ell)^2 \right\} d z\\
	& = (2\pi)^{-k/2} \exp\left\{ -\frac{1}{2}\|\nu_j\|^2 \right\} \\
	& \qquad \times 
		\int_{\bbR^k} \exp\left\{ \sum_{\ell\in[k]} \left[ \left(t^2\theta_j^2-\frac{1}{2}\right) z_\ell^2 + (\nu_{j,\ell} + t\theta_j \left(\mu_{j, \ell-1} - \mu_{j,\ell})\right)z_\ell - t^2 \theta_j^2 z_\ell z_{\ell+1} \right] \right\} dz
\end{align*}
where the second equality is by $\bbE[e^{t \cdot N(\mu, \sigma^2)}] = e^{\mu t + \sigma^2 t^2/2}$.
We claim that there exists $a_j, b_j \in \bbR, c_j \in \bbR^k$ such that
$$
	-\frac{1}{2} \sum_{\ell\in[k]} \left( a_j z_\ell + b_j z_{\ell+1} + c_{j, \ell}\right)^2 + \frac{1}{2} \|c_j\|^2
	=
	\sum_{\ell\in[k]} \left[ \left(t^2\theta_j^2-\frac{1}{2}\right) z_\ell^2 + (\nu_{j,\ell} + t\theta_j \left(\mu_{j, \ell-1} - \mu_{j,\ell})\right)z_\ell - t^2 \theta_j^2 z_\ell z_{\ell+1} \right].
$$
If this claim holds, then we have
\begin{align*}
	& \bbE\left[ \exp\left\{ t \theta_j W_j^\top (I_k^\leftarrow - I_k) Z_j \right\} \right]\\
	& = (2\pi)^{-k/2} \exp\left\{-\frac{1}{2}(\|\nu_j\|^2 - \|c_j\|^2)\right\}
		\int_{\bbR^k} \exp\left\{ -\frac{1}{2}\sum_{\ell\in[k]} \left( a_j z_\ell + b_j z_{\ell+1} + c_{j, \ell} \right)^2 \right\} dz.
\end{align*}
Introducing the change of variable $u = \sfC(a_j, 0, \hdots, 0, b_j) z$, the above quantity can be expressed as
\begin{align*}
	& (2\pi)^{-k/2} \exp\left\{ -\frac{1}{2}(\|\nu_j\|^2 - \|c_j\|^2)\right\} \int_{\bbR^k} e^{-\|u\|^2/2} du \times \frac{1}{|\det \sfC(a_j, 0, \hdots, 0, b_j)|} \\
	& = \exp\left\{-\frac{1}{2} (\|\nu_j\|^2 - \|c_j\|)\right\} \cdot \frac{1}{|a_j^k - (-b_j)^k|},
\end{align*}
where the equality is by the Laplace expansion of the determinant. Since
$$
	-\frac{1}{2} \sum_{\ell\in[k]} (a_j z_\ell + b_j z_{\ell+1} + c_{j, \ell})^2 + \frac{1}{2} \|c_j\|^2 = \sum_{\ell\in[k]} \left[-\frac{a_j^2 + b_j^2}{2} z_\ell^2 - a_j b_j z_\ell z_{\ell+1} - (a_j c_{j, \ell} + b_j c_{j, \ell-1})z_\ell\right],
$$
the quantities $a_j, b_j, c_j$ must satisfy
\begin{align*}
	a_j^2 + b_j^2 & = 1 - 2t^2\theta_j^2, \\
	a_j b_j & = t^2 \theta_j^2, \\
	a_j c_{j, \ell} + b_j c_{j, \ell+1} & = t\theta_j(\mu_{j, \ell} - \mu_{j, \ell-1}) - \nu_{j,\ell} = (t-1) \nu_{j, \ell} - t \nu_{j,\ell},
\end{align*}	
where the last equality is by $\theta_j \mu_j = \nu_j$ for $j\in J_1$. By the first two equations above, one valid choice of $a_j, b_j$ is given by
\begin{equation}
	\label{eq:core_upper_bound_aj_bj}
	a_j = \frac{1 + \sqrt{1- 4t^2 \theta_j^2}}{2},
	\qquad
	b_j = \frac{1 - \sqrt{1- 4t^2 \theta_j^2}}{2}.
\end{equation}	
Meanwhile, the equation that involves $c_j$ can be expressed in matrix notations as
\begin{equation*}
	\sfC(a_j, b_j, 0, \hdots, 0) c_j = \sfC(t-1, -t, 0, \hdots, 0) \nu_j \iff
	c_j = \sfC(a_j, b_j, 0, \hdots, 0)^{-1} \sfC(t-1, -t, 0, \hdots, 0) \nu_j.
\end{equation*}
We now calculate the norm of $c_j$. 
Let $\omega_k = e^{-2\pi i/k}$ be a $k$-th root of unity and consider the DFT matrix
$$
	F_k := 
	\begin{pmatrix}
		1 & 1 & 1 & \cdots & 1\\
		1 & \omega_k & \omega_k^2 & \cdots & \omega_k^{k-1}\\
		1 & \omega_k^2 & \omega_k^{4} & \cdots & \omega_k^{2(k-1)} \\
		\vdots & \vdots & \vdots & \ddots & \vdots\\
		1 & \omega_k^{k-1} & \omega_k^{2(k-1)} & \cdots & \omega_k^{(k-1)^2}
	\end{pmatrix}.
$$
Let $\sfC(v)$ denote the circulant matrix with the first column being $v$:
$$
	\sfC(v) := 
	\begin{pmatrix}
		v_1 &  v_k & \cdots  & v_3 & v_2\\
		v_2 &  v_1 & \cdots  & v_4 & v_3 \\
		\vdots & \vdots & \ddots & \vdots & \vdots\\
		v_{k-1} & v_{k-2} & \cdots & v_1 & v_k\\
		v_k & v_{k-1} & \cdots & v_2 & v_1
	\end{pmatrix}.
$$
For any circulant matrix $\sfC(v)$, we can diagonalize it by
$
	\sfC(v) = (\frac{1}{\sqrt{k}} F_k^*) \diag(F_k v) (\frac{1}{\sqrt{k}} F_k),
$
where $F_k^*$ is the conjugate transpose of $F_k$. 
Thus, we have
\begin{align*}
	c_j & = 
	\left[
		\frac{F_k^*}{\sqrt{k}} \diag\left( F_k (a_j, b_j,0,\hdots, 0)^\top\right) \frac{F_k}{\sqrt{k}}
	\right]^{-1}
	\left[
	\frac{F_k^*}{\sqrt{k}} \diag\left( F_k (t-1, -t, 0,\hdots, 0)^\top\right)  \frac{F_k}{\sqrt{k}}
	\right]
	\nu_j\\
	& = 
	\frac{F_k^*}{\sqrt{k}} \left[\diag\left(F_k (a_j,b_j,0,\hdots, 0)^\top\right)\right]^{-1}
	\diag\left(F_k (t-1,-t,0,\hdots,0)^\top\right) \frac{F_k}{\sqrt{k}} \nu_j.
\end{align*}
where the last equality follows from the fact that $F_k/\sqrt{k}$ is unitary.
We then proceed by
\begin{align*}
	\|c_j\|^2
	& = \nu_j^\top \frac{F_k^*}{\sqrt{k}}	\left[\diag\left(F_k (t-1,-t,0,\hdots,0)^\top\right)\right]^*
	\left[\diag\left(F_k(a_j,b_j,0,\hdots,0)^\top\right)^{-1}\right]^* \\
	& \qquad
	\diag\left(F_k (a_j,b_j,0,\hdots,0)^\top\right)^{-1}
	\diag\left(F_k(t-1,-t,0,\hdots,0)^\top\right)
	\frac{F_k}{\sqrt{k}} \nu_j\\
	& = \nu_j^\top \frac{F_k^*}{\sqrt{k}} \diag(v_j) \frac{F_k}{\sqrt{k}} \nu_j\\
	\label{eq:core_upper_bound_cj_norm}
	& = \nu_j^\top \sfC(F_k^* v_j/k) \nu_j, \numberthis
\end{align*}
where $v_j$ is an $\bbR^k$ vector whose $\ell$-th coordinate is given by
\begin{align*}
	v_{j,\ell} & = \frac{(t-1 -t \omega_k^{\ell-1})(t-1-t \omega_k^{1-\ell})}{(a_j + b_j \omega_k^{\ell-1})(a_j + b_j\omega_k^{1-\ell})} \\
	& = \frac{(t-1)^2 + t^2 - 2t(t-1)\cos \frac{2\pi(\ell-1)}{k}}{a_j^2 + b_j^2 + 2 a_jb_j \cos \frac{2\pi(\ell-1)}{k}} \\
	& = \frac{1 + 2t(t-1)(1- \cos \frac{2\pi(\ell-1)}{k})}{(a_j+b_j)^2 - 2a_j b_j (1-\cos\frac{2\pi(\ell-1)}{k})} \\
	\label{eq:core_upper_bound_vj}
	& = \frac{1 + 2t(t-1) (1- \cos \frac{2\pi(\ell-1)}{k})}{1 - 2t^2\theta_j^2 (1-\cos \frac{2\pi(\ell-1)}{k})},\numberthis
\end{align*}
provided the denominator is non-zero.
In the above display, the first equality is by the fact that the complex conjugate of $\omega_k^{\ell-1}$ is $\omega_k^{1-\ell}$, the second equality is by $\omega_k^{\ell-1} + \omega_k^{1-\ell} = 2\cos(2\pi(\ell-1)/k)$, and the last equality is by $(a_j + b_j)^2 = 1, a_jb_j = t^2\theta_j^2$. 
In summary, we have shown that for any $j\in J_1$,
\begin{equation}
	\label{eq:core_upper_bound_mgf_J1}
	\bbE\left[ \exp\left\{ t \theta_j W_j^\top (I_k^\leftarrow - I_k) Z_j \right\} \right] 
	= \exp\left\{\frac{1}{2} \left(-\|\nu_j\|^2 + \nu_j^\top \sfC(F_k^* v_j / k) \nu_j \right)\right\} \cdot \frac{1}{|a_j^k - (-b_j)^k|},
\end{equation}	
where $a_j, b_j$ are given by \eqref{eq:core_upper_bound_aj_bj} and $v_j$ is a vector whose $\ell$-th coordinate is given by \eqref{eq:core_upper_bound_vj}.

We then consider $j \in J_2$, so that $\theta_j \nu_j = \mu_j$. Note that
\begin{align*}
	& \bbE\left[ \exp\left\{ t \theta_j W_j^\top (I_k^\leftarrow - I_k) Z_j \right\} \right]\\
	& = \bbE\left[ \prod_{\ell\in[k]} \exp\left\{ t \theta_j(W_{j, \ell-1} - W_{j, \ell}) Z_{j, \ell} \right\} \right] \\
	& = \bbE\left[ \prod_{\ell\in[k]} \exp\left\{t\theta_j \nu_{j, \ell} (W_{j, \ell-1} - W_{j, \ell}) + \frac{1}{2}t^2 \theta_j^2 (W_{j, \ell - 1} - W_{j,\ell})^2 \right\} \right]\\
	& = (2\pi)^{-k/2} \int_{\bbR^k} \exp\left\{ -\frac{1}{2}\sum_{\ell\in[k]} (w_\ell - \mu_{j, \ell})^2 + t\theta_j \sum_{\ell\in[k]} \nu_{j, \ell}(w_{\ell-1} - w_\ell) + \frac{1}{2}t^2\theta_j^2 \sum_{\ell\in[k]} (w_{\ell-1} - w_\ell)^2 \right\} dw\\
	& = (2\pi)^{-k/2}\exp\left\{ -\frac{1}{2} \|\mu_j\|^2 \right\} \\
		& \qquad \times \int_{\bbR^k}\exp\left\{ \sum_{\ell\in[k]} \left[(t^2\theta_j^2 - \frac{1}{2}) w_\ell^2 + \left( \mu_{j, \ell} + t\theta_j (\nu_{j, \ell+1} - \nu_{j, \ell}) \right) w_\ell - t^2\theta_j^2 w_{\ell-1}w_\ell \right]\right\} dw
\end{align*}
where the second equality is by $\bbE[e^{t N(\mu, \sigma^2)}] = e^{\mu t + \sigma^2 t^2/2}$. 
We now seek for $a_j, b_j\in\bbR, c_j\in\bbR^k$ such that
$$
	-\frac{1}{2} \sum_{\ell\in[k]} (a_j w_\ell + b_j w_{\ell-1} + c_{j,\ell})^2 + \frac{1}{2} \|c_j\|^2 
	= \sum_{\ell\in[k]} \left[(t^2\theta_j^2 - \frac{1}{2}) w_\ell^2 + \left( \mu_{j, \ell} + t\theta_j (\nu_{j, \ell+1} - \nu_{j, \ell}) \right) w_\ell - t^2\theta_j^2 w_{\ell-1}w_\ell\right].
$$
If the above equality holds, then 
\begin{align*}
	& \bbE\left[ \exp\left\{ t \theta_j W_j^\top (I_k^\leftarrow - I_k) Z_j \right\} \right]\\
	& = (2\pi)^{-k/2} \exp\left\{-\frac{1}{2}(\|\mu_j\|^2 - \|c_j\|^2)\right\} \int_{\bbR^k} \exp\left\{-\frac{1}{2} \sum_{\ell\in[k]} (a_j w_\ell + b_j w_{\ell-1} + c_{j,\ell})^2 + \frac{1}{2} \|c_j\|^2\right\}  dw.
\end{align*}
Introducing the change of variable $u = \sfC(a_j, b_j, 0, \hdots, 0) w$, the right-hand side above becomes
\begin{align*}
	& (2\pi)^{k/2} \exp\left\{-\frac{1}{2}(\|\mu_j\|^2 - \|c_j\|^2)\right\} \int_{\bbR^k} e^{-\|u\|^2/2} dw \times \frac{1}{|\det \sfC(a_j, b_j, 0, \hdots, 0)|} \\
	& =  \exp\left\{-\frac{1}{2}(\|\mu_j\|^2 - \|c_j\|^2)\right\} \cdot \frac{1}{|a_j^k- (-b_j)^k|}.
\end{align*}
Since 
$$
	-\frac{1}{2} \sum_{\ell\in[k]} (a_j w_\ell + b_j w_{\ell-1} + c_{j,\ell})^2 + \frac{1}{2} \|c_j\|^2 
	= \sum_{\ell\in[k]} \left[ -\frac{a_j^2 + b_j^2}{2} w_\ell^2 - a_j b_j w_\ell w_{\ell-1} - (a_j c_{j, \ell} + b_j c_{j, \ell + 1}) w_\ell \right],
$$
we need the following three equalities to hold:
\begin{align*}
	a_j^2 + b_j^2 & = 1 - 2t^2\theta_j^2 \\
	a_j b_j & = t^2 \theta_j^2 \\
	a_j c_{j, \ell} + b_j c_{j, \ell+1} & = t\theta_j(\nu_{j,\ell} - \nu_{j,\ell+1}) - \mu_{j,\ell} = (t-1) \mu_{j, \ell} - t \mu_{j, \ell+1},
\end{align*}
where the last equality is by $\theta_j \nu_j = \mu_j$ for $j\in J_2$. Note that our previous choice of $a_j, b_j$ in \eqref{eq:core_upper_bound_aj_bj} can make the first two equalities above to hold. The third equality above now translates to
$$
	c_j = \sfC(a_j, 0, \hdots, 0, b_j)^{-1} \sfC(t-1, -t, 0, \hdots, 0, -t)\mu_j.
$$
By a nearly identical calculation as what led to \eqref{eq:core_upper_bound_cj_norm}, we have
$$
	\|c_j\|^2 = \mu_j^\top \sfC(F_k^* v_j/k) \mu_j
$$
where the $\ell$-th entry of $v_j$ is precisely given by \eqref{eq:core_upper_bound_vj}. Recalling \eqref{eq:core_upper_bound_mgf_J1}, we have shown that for any $j\in[r]$,
\begin{equation}
	\label{eq:core_upper_bound_mgf_J1J2}
	\bbE\left[ \exp\left\{ t \theta_j W_j^\top (I_k^\leftarrow - I_k) Z_j \right\} \right] 
	= \exp\left\{\frac{1}{2} \left(-\|u_j\|^2 + u_j^\top \sfC(F_k^* v_j / k) u_j \right)\right\} \cdot \frac{1}{|a_j^k - (-b_j)^k|},
\end{equation}
where
$$
	u_j = 
	\begin{cases}
		\nu_j & \textnormal{if } j\in J_1,\\
		\mu_j & \textnormal{if } j\in J_2,
	\end{cases}
$$
the quantities $a_j, b_j$ are given by \eqref{eq:core_upper_bound_aj_bj} and $v_j$ is a vector whose $\ell$-th coordinate is given by \eqref{eq:core_upper_bound_vj}.

To proceed further, we need the following lemma.
\begin{lemma}
	For the vector $v_j$ whose $\ell$-th element is given by \eqref{eq:core_upper_bound_vj}, the circulant matrix $\sfC(F_k^* v_j/k)$ is symmetric.
\end{lemma}
\begin{proof}
	It suffices to show $(F_k^* v_j)_\ell = (F_k^* v_j)_{k-\ell+2}$ for any $2\leq \ell \leq k$. The $\ell$-th coordinate of $F^*_k v_j$ is 
	$$
		(F^*_k v_j)_\ell = v_{j,1} + \omega_k^{-(\ell-1)} v_{j,2} + \omega_k^{-2(\ell-1)} v_{j,3} + \cdots + \omega_k^{-(k-1)(\ell-1)} v_{j,k},
	$$
	whereas the $(k-\ell+2)$-th coordinate of $F^*_k v_j$ is
	$$
		(F^*_k v_j)_{k-\ell+2} = v_{j,1} + \omega_k^{-(k-\ell+1)} v_{j,2} + \omega_k^{-2(k-\ell+1)} v_{j,3} + \cdots + \omega_k^{-(k-1)(k-\ell+1)} v_{j,k}.	
	$$
	To show the above two displays are equal, it suffices to show 
	$$
	v_{j,m} = v_{j, k-m+2}, \qquad \omega_k^{-(m-1)(\ell-1)} = \omega_k^{-(k-m+1)(k-\ell+1)}, \qquad \forall2\leq m\leq k.
	$$
	The first equality follows from the expression of $v_j$ and $\cos \frac{2\pi(m-1)}{k} = \cos \frac{2\pi(k-m+1)}{k}$, and the second equality follows from the definition of $\omega_k$.
\end{proof}
By the above lemma, when $k\geq 2$ is even, we can write
\begin{align*}
	u_j^\top \sfC(F_k^* v_j/k) u_j
	& = (\frac{F_k^* v_j}{k})_1  \|u_j\|^2  + \Indc\{k\geq 4\} \cdot 2 \left((\frac{F_k^* v_j}{k})_2 u_j^\top I_k^\leftarrow u_j  + \cdots + (\frac{F_k^* v_j}{k})_{k/2}  u_j^\top (I_k^\leftarrow)^{k/2-1} u_j\right) \\ 
	\label{eq:err_refined_match_k_even}
	& \qquad + (\frac{F_k^* v_j}{k} )_{k/2+1} u_j^\top (I_k^\leftarrow)^{k/2} u_j. \numberthis
\end{align*}
On the other hand, when $k\geq 3$ is odd, we have
\begin{align*}
	u_j^\top \sfC(F_k^* v_j/k) u_j
	\label{eq:err_refined_match_k_odd}
	& = (\frac{F_k^* v_j}{k})_1  \|u_j\|^2 + 2 \left((\frac{F_k^* v_j}{k})_2  u_j^\top I_k^\leftarrow u_j  + \cdots +  (\frac{F_k^* v_j}{k})_{(k+1)/2} u_j^\top (I_k^\leftarrow)^{(k-1)/2} u_j\right).\numberthis
\end{align*}
We now consider the each case of $k\in\{2, 3, 4,5\}$ separately.
\begin{itemize}
\item \emph{Case A.1: $k=2$.} In this case, we have
$$
	v_{j, 1} = 1, \qquad v_{j,2} = \frac{1 + 4 t(t - 1)}{1-4 t^2\theta_j^2}, \qquad \frac{F_2^* v_j}{2} = \frac{1}{2} \begin{pmatrix}1 + v_{j,2} \\ 1 - v_{j,2}\end{pmatrix}.
$$
From \eqref{eq:err_refined_match_k_even}, an application of Lemma \ref{lemma:inner_prod_and_norm} gives
\begin{align*}
	-\|u_j\|^2 + u_j^\top \sfC(F_2^* v_j/2)u_j & =   \left(-1 + \frac{1+v_{j,2}}{2} + \frac{1-v_{j,2}}{2}\right)\|u_j\|^2 - \frac{1-v_{j,2}}{4} \|(I_2 - I_2^\leftarrow)u_j\|^2 \\
	& = -\frac{1-v_{j,2}}{4} \|(I_2 - I_2^\leftarrow)u_j\|^2.
\end{align*}
Since $\theta_j\leq 1$, as long as $t < 1/2$, we have
$
	v_{j,2} \leq \frac{1 + 4t ( t -1)}{1 - 4t^2} = \frac{1 - 2t}{1+2t}.
$
Choosing $t = \frac{1-\ep}{2}$ for some $\ep=o(1)$ whose value will be specified later, we get $v_{j,2}\leq \ep$ and thus
$$
	-\|u_j\|^2 + u_j^\top \sfC(F_2^* v_j/2)u_j \leq -\frac{(1-\ep)}{4} \|(I_2 - I_2^\leftarrow)u_j\|^2.
$$
Moreover, under such a choice of $t$, we have
$$
	\frac{1}{a_j^2 - b_j^2} = \frac{1}{\sqrt{1 - (1-\ep)^2 \theta_j^2}} \leq  \frac{1}{\sqrt{1-(1-\ep)^2}} \lesssim \frac{1}{\sqrt{\ep}},
$$
where we have used $\theta_j\leq 1$. Recalling \eqref{eq:core_upper_bound_mgf_J1J2}, we get
\begin{align*}
	& \bbE\left[\exp\left\{ t \theta_j W_j^\top (I_k^\leftarrow- I_k) Z_j\right\}\right] \leq \exp\left\{ -\frac{1-\ep}{8}
	\|(I_2 - I_2^\leftarrow)u_j\|^2
	+ \calO\left(\log \frac{1}{\ep}\right)
	\right\}
\end{align*}	
Plugging the above inequality back to \eqref{eq:core_upper_bound_markov}, we get
\begin{align*}
	& \bbP\left( \sum_{j\in[r]}\theta_j W_j^\top (I_k^\leftarrow - I_k) Z_j + \xi \geq 0\right) \\
	& \leq  
	\exp\left\{
		-\frac{1-\ep}{8} \sum_{j\in[r]}  \|(I_2 - I_2^\leftarrow) u_j\|^2
		+ \frac{(1-\ep)\xi}{2}
		+ \calO\left(r\log\frac{1}{\ep}\right)
	\right\}.
\end{align*}
Choosing $\ep = r/\tilde\beta^2 = o(1)$, the right-hand side above becomes
\begin{align*}
	\exp\left\{
		-\frac{1-o(1)}{8} \sum_{j\in[r]} \|(I_2 - I_2^\leftarrow) u_j\|^2
		+ \frac{(1-o(1))\xi}{2}
	\right\},
\end{align*}
as desired.
\item \emph{Case A.2: $k=3$.} In this case, we have
$$
	v_{j,1} = 1, \qquad v_{j,2} = v_{j,3} = \frac{1 + 3 t(t - 1)}{1-3t^2\theta_j^2},
	\qquad
	\frac{F_3^* v_j}{3} = \frac{1}{3} \begin{pmatrix}1 + 2v_{j,2} \\ 1 - v_{j,2} \\ 1 - v_{j,3}\end{pmatrix}.
$$
From \eqref{eq:err_refined_match_k_odd}, an application of Lemma \ref{lemma:inner_prod_and_norm} gives
\begin{align*}
	- \|u_j\|^2 + u_j^\top \sfC(F_3^* v_j/3) u_j
	& = \left(-1 + \frac{1}{3} (1 + 2 v_{j,2} + 2 - 2v_{j,2})\right) \| u_j\|^2 - \frac{1-v_{j,2}}{3} \| (I_3 - I_3^\leftarrow)u_j\|^2 \\
	& = -\frac{1-v_{j,2}}{3}\| (I_3 - I_3^\leftarrow)u_j\|^2.
\end{align*}
Since $\theta_j\leq 1$, we have
$
	v_{j,2} \leq \frac{1 + 3 t(t - 1)}{1-3t^2} = \frac{1}{2}
$
where the last equality holds by choosing $t = 1/3$. 
Meanwhile, with such a choice of $t$, we have
$$
	|a_j^3 - (-b_j)^3| = a_j^3 + b_j^3 \geq a_j^3 =  \left(\frac{1+\sqrt{1 - 4\theta_j^2/9}}{2}\right)^3  \geq \left(\frac{1 + \sqrt{5}/3}{2}\right)^3,
$$
where the last inequality is again by $\theta_j\leq 1$. Recalling \eqref{eq:core_upper_bound_mgf_J1J2}, we have
\begin{align*}
	& \bbE\left[\exp\left\{ t \theta_j W_j^\top (I_k^\leftarrow - I_k) Z_j\right\}\right] 
	= \exp\left\{- \frac{1}{12} \|(I_3 - I_3^\leftarrow) u_j\|^2 + \calO(1)\right\}.
\end{align*}
Plugging the above inequality to \eqref{eq:core_upper_bound_markov}, we get
\begin{align*}
	& \bbP\left( \sum_{j\in[r]}\theta_j W_j^\top (I_k^\leftarrow - I_k) Z_j + \xi \geq 0\right) \\
	& \leq \exp\left\{- \frac{1}{12} \sum_{j\in[r]} \|(I_3 - I_3^\leftarrow)u_j\|^2 + \calO (r) + \frac{\xi}{3}\right\} \\
	& = \exp\left\{- \frac{1-o(1)}{12} \sum_{j\in[r]} \|(I_3 - I_3^\rightarrow)u_j\|^2 +  \frac{\xi}{3}\right\},
\end{align*}
where the last inequality is by $\tilde\beta^2\gg r $.
\item \emph{Case A.3: $k=4$.} In this case, we have
$$
	v_{j, 1} = 1, \qquad v_{j, 2}= v_{j, 4} = \frac{1 + 2t(t-1)}{1-2t^2 \theta_j^2}, 
	\qquad v_{j, 3} = \frac{1 + 4t(t-1)}{1-4t^2 \theta_j^2},
	\qquad
	\frac{F_4^* v_j}{4} = \frac{1}{4}
	\begin{pmatrix}
		1 + 2v_{j,2} + v_{j,3}\\
		1 - v_{j,3}\\
		1 - 2v_{j,2} + v_{j,3}\\
		1 - v_{j,3}
	\end{pmatrix}.
$$
By \eqref{eq:err_refined_match_k_even}, we have
\begin{align*}
	& -\|u_j\|^2 + u_j^\top \sfC(F_4^* v_j/4) u_j\\
	& = \frac{1}{4} \left((-4 + 1 + v_{j,2} + v_{j,3})\|u_j\|^2 + 2 (1-v_{j,3}) u_j^\top I_4^\leftarrow u_j + (1-2v_{j,2} + v_{j,3}) u_j^\top (I_4^\leftarrow)^2 u_j\right) \\
	& = \frac{1}{4} 
	\left[
		-(2v_{j,2} - v_{j,3} - 1) \left(\|u_j\|^2 - 2 u_j^\top I_4^\leftarrow u_j + u_j^\top (I_4^\leftarrow)^2 u_j\right)
		- 4(1- v_{j,2})(\|u_j\|^2 - u_j^\top I_4^\leftarrow u_j)
	\right]\\
	& = \frac{-(2v_{j,2} - v_{j,3} - 1)}{4} (u_{j,1} - u_{j,2} + u_{j,3} - u_{j,4})^2 - \frac{1-v_{j,2}}{2} \|(I_4 - I_4^\leftarrow)u_j\|^2.
\end{align*}
Choosing $t = 1/3$, we have
$
	v_{j,2} = \frac{5}{9 - 2\theta_j^2} \leq \frac{5}{7}.
$
Meanwhile, we have
$$
	2 v_{j,2} - v_{j,3} - 1 = \frac{10}{9 - 2\theta_j^2} - \frac{1}{9 - 4\theta_j^2} - 1.
$$
We claim that the above quantity is non-negative for $\theta_j \in [0, 1]$. Indeed, because the derivative of the right-hand side above with respect to $\theta_j$ satisfies
$$
	\theta_j \cdot \left(\frac{40}{(9-2\theta_j^2)^2} - \frac{8}{(9 - 4\theta_j^2)^2}\right) 
	\geq \theta_j \cdot \left(\frac{40}{81} - \frac{8}{25}\right) \geq 0,
$$
it holds that
$$
2 v_{j,2} - v_{j,3} - 1 \geq \left(\frac{10}{9 - 2\theta_j^2} - \frac{1}{9 - 4\theta_j^2} \right)\bigg|_{\theta_j = 0} = 0.
$$
Hence, we arrive at
\begin{align*}
	-\|u_j\|^2 + u_j^\top \sfC(F_4^* v_j/4) u_j
	& \leq -\frac{1}{7} \|(I_4 - I_4^\leftarrow)u_j\|^2.
\end{align*}
Note that under the choice of $t = 1/3$, both $a_j, b_j$ are at least $\Omega(1)$, and
$$
	a_j-b_j =\frac{1 + \sqrt{1 - 4\theta_j^2/9}}{2} - \frac{1 - \sqrt{1 - 4\theta_j^2/9}}{2} = \sqrt{1- 4\theta_j^2/9} \gtrsim 1.
$$
Thus, we get
$$
	\frac{1}{|a_j^4 - (-b_j)^4|} = \frac{1}{a_j^4 - b_j^4} = \frac{1}{(a_j^2+b_j^2)(a_j+b_j)(a_j - b_j)} \lesssim 1.
$$
Recalling \eqref{eq:core_upper_bound_mgf_J1J2}, we have
\begin{align*}
	& \bbE\left[\exp\left\{ t \theta_j W_j^\top (I_k^\leftarrow - I_k) Z_j\right\}\right] 
	= \exp\left\{- \frac{1}{14}
	\|(I_4 - I_4^\leftarrow) u_j\|^2
	+ \calO(1)
	\right\}.
\end{align*}
Plugging the above inequality to \eqref{eq:core_upper_bound_markov} and using $\tilde\beta^2 \gg r$, we get
\begin{align*}
	\bbP\left( \sum_{j\in[r]}\theta_j W_j^\top (I_k^\leftarrow - I_k) Z_j + \xi \geq 0\right) 
	& = \exp\left\{- \frac{1-o(1)}{14} \sum_{j\in[r]} \|(I_4 - I_4^\leftarrow)u_j\|^2 +  \frac{\xi}{3}\right\}.
\end{align*}

\item \emph{Case A.4: $k = 5$.} In this case, we have
$$
	v_{j, 1}= 1,
	\qquad
	v_{j,2} = v_{j,5} = \frac{1 + 2t(t-1)(1 - \cos\frac{2\pi}{5})}{1 - 2t^2 \theta_j^2 (1 - \cos\frac{2\pi}{5})},
	\qquad
	v_{j,3} = v_{j,4} = \frac{1 + 2t(t-1)(1-\cos\frac{4\pi}{5})}{1 - 2t^2\theta_j^2 (1-\cos\frac{4\pi}{5})}
$$
and
\begin{align*}
	(F_5^* v_j)_1 & = 1 + 2v_{j,2} + 2v_{j,3},\\
	(F_5^* v_j)_2 & = 1 + 2\cos\frac{2\pi}{5} v_{j,2} + 2 \cos \frac{4\pi}{5} v_{j,3}\\
	(F_5^* v_j)_3 & = 1 + 2 \cos\frac{4\pi}{5} v_{j,2} + 2 \cos\frac{2\pi}{5} v_{j,3}.
\end{align*}
From \eqref{eq:err_refined_match_k_even}, we get
\begin{align*}
	& -\|u_j\|^2 + u_j^\top \sfC(F_5^* v_j/5) u_j\\
	& = \frac{1}{5} \bigg[
		(-4+2v_{j,2} + 2v_{j,3}) \|u_j\|^2
		+ 2\left(1 + 2 \cos\frac{2\pi}{5} v_{j,2} + 2 \cos\frac{4\pi}{5} v_{j,3}\right) u_j^\top I_5^\leftarrow u_j\\
		& \qquad + 2 \left(1 + 2\cos \frac{4\pi}{5} v_{j,2} + 2 \cos\frac{2\pi}{5} v_{j,3}\right) u_j^\top (I_5^\leftarrow)^2 u_j
	\bigg]\\
	& = \frac{1}{5} \left(2 + 2 \cos\frac{2\pi}{5} + 2\cos\frac{4\pi}{5}\right)(v_{j,3} + v_{j,3}) \|u_j\|^2 \\
	& \qquad - \frac{1}{5} \left(1 + 2\cos \frac{2\pi}{5} v_{j,2} + 2\cos \frac{4\pi}{5}v_{j,3}\right) \|(I_5 - I_5^\leftarrow)u_j\|^2 \\
	& \qquad - \frac{1}{5} \left(1 + 2\cos\frac{4\pi}{5} v_{j,2} + 2\cos \frac{2\pi}{5} v_{j,3}\right) \|(I_5 - (I_5^\leftarrow)^2)u_j\|^2 \\
	& = - \frac{1}{5} \left(1 + \frac{\sqrt{5}-1}{2} v_{j,2} -\frac{\sqrt{5}+1}{2}v_{j,3}\right) \|(I_5 - I_5^\leftarrow)u_j\|^2 \\
	& \qquad - \frac{1}{5} \left(1 -\frac{\sqrt{5}+1}{2} v_{j,2} + \frac{\sqrt{5}-1}{2}v_{j,3}\right) \|(I_5 - (I_5^\leftarrow)^2)u_j\|^2. 
\end{align*}
where the second equality is by Lemma \ref{lemma:inner_prod_and_norm}. We choose $t = 1/4$. Then, we have
\begin{align*}
	v_{j,2} & = \frac{1 - \frac{3}{8}(1 - \frac{\sqrt{5}-1}{4})}{1 - \frac{1}{8}\theta_j^2 (1 - \frac{\sqrt{5}-1}{4})} = \frac{17 + 3\sqrt{5}}{32 - (5-\sqrt{5})\theta_j^2}\\
	v_{j,3} & = \frac{1 - \frac{3}{8}(1 + \frac{\sqrt{5}+1}{4})}{1 - \frac{1}{8}\theta_j^2 (1 + \frac{\sqrt{5}+1}{4})} = \frac{17 - 3\sqrt{5}}{32 - (5+\sqrt{5})\theta_j^2}.
\end{align*}	
Thus,
\begin{align*}
	& -\|u_j\|^2 + u_j^\top \sfC(F_5^* v_j/5) u_j\\
	& = -\frac{1}{5} \left(1 + \frac{\sqrt{5}-1}{2}\frac{17 + 3\sqrt{5}}{32 - (5-\sqrt{5})\theta_j^2} - \frac{\sqrt{5}+1}{2}\frac{17 - 3 \sqrt{5}}{32 - (5+\sqrt{5})\theta_j^2}\right) \|(I_5 - I_5^\leftarrow)u_j\|^2 \\
	& \qquad + \frac{1}{5}\underbrace{\left(-1 + \frac{\sqrt{5}+1}{2}\frac{17 + 3\sqrt{5}}{32 - (5-\sqrt{5})\theta_j^2} - \frac{\sqrt{5}-1}{2}\frac{17 - 3 \sqrt{5}}{32 - (5+\sqrt{5})\theta_j^2}\right)}_{f(\theta_j)} \|(I_5 - (I_5^\leftarrow)^2) u_j\|^2.
\end{align*}
With some algebra, one can readily check that 
$$
	\frac{d f(\theta_j)}{\theta_j^2} = \frac{5 (65\theta_j^4 - 512\theta_j^2 + 768)}{(5\theta_j^4 - 80\theta_j + 256)^2} \geq 0
$$
for $\theta_j \in [0, 1]$. Hence, we have $f(\theta_j)\geq f(0) = 0$.
Meanwhile, note that 
$$
\|(I_5 - (I_5^\leftarrow)^2)u_j\|^2 = \|(I_5 + I_5^\leftarrow)(I_5 - I_5^\leftarrow)u_j\|^2  \leq \|I_5 + I_5^\leftarrow\|^2 \|(I_5 - I_5^\leftarrow)u_j\|^2 \leq 4\|(I_5 - I_5^\leftarrow)u_j\|^2.
$$
This yields
\begin{align*}
	& -\|u_j\|^2 + u_j^\top \sfC(F_5^* v_j/5) u_j\\
	& \leq \bigg[-\frac{1}{5} \left(1 + \frac{\sqrt{5}-1}{2}\frac{17 + 3\sqrt{5}}{32 - (5-\sqrt{5})\theta_j^2} - \frac{\sqrt{5}+1}{2}\frac{17 - 3 \sqrt{5}}{32 - (5+\sqrt{5})\theta_j^2}\right)  \\
	& \qquad + \frac{4}{5}\left(-1 + \frac{\sqrt{5}+1}{2}\frac{17 + 3\sqrt{5}}{32 - (5-\sqrt{5})\theta_j^2} - \frac{\sqrt{5}-1}{2}\frac{17 - 3 \sqrt{5}}{32 - (5+\sqrt{5})\theta_j^2}\right) \bigg]\|(I_5 - I_5^\leftarrow)u_j\|^2  \\
	& = -\frac{1}{5} \bigg[
		5 + \left(\frac{\sqrt{5}-1}{2} - 2(\sqrt{5} + 1)\right) \frac{17 +3\sqrt{5}}{32 - (5-\sqrt{5})\theta_j^2} \\
		& \qquad
		+ \left(-\frac{\sqrt{5}+1}{2} + 2 (\sqrt{5}-1)\right) \frac{17 - 3\sqrt{5}}{32- (5+\sqrt{5})\theta_j^2}
	\bigg] \|(I_5 - I_5^\rightarrow)u_j\|^2 \\
	& = -\frac{1}{5} g(\theta_j) \|(I_5 - I_5^\leftarrow)u_j\|^2,
\end{align*}
where
$$
	\frac{d g(\theta_j)}{d\theta_j^2} = \frac{-5(245\theta_j^4 - 2080\theta_j + 4096)}{(5\theta_j^4 - 80\theta_j^2 + 256)^2} \leq 0
$$
for $\theta_j\in[0, 1]$. Thus, we get $g(\theta_j)\geq g(1) = 110/181$, and 
\begin{align*}
	-\|u_j\|^2 + u_j^\top \sfC(F_5^* v_j/5) u_j
	& \leq -\frac{22}{181} \|(I_5 - I_5^\leftarrow)u_j\|^2.
\end{align*}
Now, recall the definition of $a_j, b_j$, we have
\begin{align*}
	|a_j^5 - (-b_j)^5| = a_j^5 + b_j^5 \geq a_j^5 = \left(\frac{1 + \sqrt{1 - \theta_j^2/4}}{2}\right)^2 \gtrsim 1 \implies \frac{1}{|a_j^5 - (-b_j)^5|} \lesssim 1.
\end{align*}
Recalling \eqref{eq:core_upper_bound_mgf_J1J2}, we get
\begin{align*}
	& \bbE\left[\exp\left\{ t \theta_j W_j^\top (I_k^\leftarrow - I_k) Z_j\right\}\right] 
	= \exp\left\{- \frac{11}{181}
	\|(I_5 - I_5^\rightarrow) u_j\|^2
	+ \calO(1)
	\right\}.
\end{align*}
Plugging the above inequality to \eqref{eq:core_upper_bound_markov} and using $\tilde\beta^2 \gg r$, we get
\begin{align*}
	& \bbP\left( \sum_{j\in[r]}\theta_j W_j^\top (I_k^\leftarrow - I_k) Z_j + \xi \geq 0\right) \\
	& = \exp\left\{- \frac{(1-o(1)) \cdot 11}{181} \sum_{j\in[r]} \|(I_5 - I_5^\rightarrow)u_j\|^2 +  \frac{\xi}{4}\right\}.
\end{align*}
\end{itemize}

\paragraph{Case B. $\boldsymbol{k > 6}$.}
We write $W_j = \mu_j + N_{\x, j}, Z_j = \nu_j + N_{\y, j}$, where $N_{\x,j}\indep N_{\y, j}$ and both follow $N(0, I_k)$.
Recall that $u_j = \nu_j$ for $j\in J_1$ and $u_j = \mu_j$ for $j\in J_2$.
We start by computing
\begin{align*}
	& \bbP\left( \sum_{j\in[r]}\theta_j W_j^\top (I_k^\leftarrow - I_k) Z_j + \xi \geq 0 \right) \\
	& = \bbP\left( \sum_{j\in[r]} \theta_j(\mu_j + N_{\x, j})^\top(I_k^\leftarrow - I_k) (\nu_j + N_{\y, j}) + \xi \geq 0 \right) \\
	& = \bbP\Bigg( \sum_{j\in[r]} \theta_j \mu_j^\top (I_k^\leftarrow - I_k) \nu_j + \sum_{j\in[r]} \theta_j \mu_j (I_k^\leftarrow - I_k) N_{\y, j} + \sum_{j\in[r]} \theta_j N_{\x, j}^\top (I_k^\rightarrow - I_k) \nu_j \\
		& \qquad + \sum_{j\in[r]} \theta_j N_{\x,j}^\top(I_k^\leftarrow-I_k)N_{\y, j} + \xi \geq 0 \Bigg) \\
	& = \bbP\Bigg(
		-\frac{1}{2} \sum_{j\in[r]} \|(I_k^\leftarrow - I_k) u_j\|^2 + \sum_{j\in[r]} \theta_j\mu_j (I_k^\leftarrow - I_k) N_{\y, j} + \sum_{j\in[r]} \theta_j N_{\x, y}^\top (I_k^\leftarrow - I_k) \nu_j \\
		& \qquad + \sum_{j\in[r]} \theta_j N_{\x, j}^\top(I_k^\leftarrow - I_k) N_{\y, j} + \xi \geq 0
	\Bigg),
\end{align*}
where we have used the definition of $u_j$ and Lemma \ref{lemma:inner_prod_and_norm} in the last equality.
Since 
$$
	\theta_j \mu_j^\top (I_k^\leftarrow - I_k) N_{\y,j} \sim N(0, \|(I_k^\leftarrow-I_k)\theta_j \mu_j\|^2),
	\qquad 
	\theta_j N_{\x,j}^\top (I_k^\leftarrow -I_k)\nu_j \sim N(0, \|(I_k^\leftarrow - I_k)\nu_j\|^2),
$$
and they are independent, for any $\gamma \in (0, 1/2)$, we have
$$
	\bbP\left( \sum_{j\in[r]}\theta_j W_j^\top (I_k^\leftarrow - I_k) Z_j + \xi \geq 0 \right) \leq \scrT_1 + \scrT_2,
$$
where
\begin{align*}
	\scrT_1 & := 
	\bbP\left( -(\frac{1}{2}-\gamma) \sum_{j\in[r]} \|(I_k^\leftarrow - I_k) u_j\|^2 
		+ N\left(0, \sum_{j\in[r]} \left( \|(I_k^\leftarrow - I_k) \theta_j\mu_j\|^2 + \|(I_k^\leftarrow-I_k)\theta_j \nu_j\|^2 \right)\right) \geq 0
	\right),\\
	\scrT_2 & := 
	\bbP\left(
		\sum_{j\in[r]} \theta_j N_{\x, j}^\top(I_k^\leftarrow - I_k) N_{\y, j} + \xi \geq \gamma \sum_{j\in[r]} \|(I_k^\leftarrow - I_k)u_j\|^2
	\right).
\end{align*}
Note that
\begin{align*}
	\scrT_1 = 
	\bbP\left(N(0, 1) \geq \frac{\left(\frac{1}{2}-\gamma\right)\sum_{j\in[r]}\|(I_k^\leftarrow - I_k)u_j\|^2}{\sqrt{\sum_{j\in[r]} \left( \|(I_k^\leftarrow - I_k) \theta_j\mu_j\|^2 + \|(I_k^\leftarrow-I_k)\theta_j \nu_j\|^2 \right)}}\right).
\end{align*}
Because
\begin{align*}
	\sum_{j\in[r]} \left( \|(I_k^\leftarrow - I_k) \theta_j\mu_j\|^2 + \|(I_k^\leftarrow-I_k)\theta_j \nu_j\|^2 \right)
	& = \sum_{j\in J_1} (1 + \theta_j^2)\|(I_k^\leftarrow - I_k) \nu_j\|^2 + \sum_{j\in[J_2]} (1+ \theta_j^2) \|(I_k^\leftarrow - I_k) \mu_j\|^2  \\
	& \leq 2 \sum_{j\in[r]} \|(I_k^\leftarrow - I_k)u_j\|^2,
\end{align*}
we have
\begin{align*}
	\scrT_1 & \leq \bbP\left( N(0, 1) \geq \frac{\frac{1}{2}-\gamma}{\sqrt{2}} \sqrt{\sum_{j\in[r]} \|(I_k^\leftarrow - I_k) u_j\|^2}\right)\\
	& \leq \exp\left\{ -\frac{\left(\frac{1}{2} - \gamma\right)^2}{4} \sum_{j\in[r]}\|(I_k^\leftarrow - I_k)u_j\|^2 \right\},
\end{align*}	
where the last inequality is by Gaussian tail bound (see Lemma \ref{lemma:gaussian_tail}) and $\tilde\beta \gg r \geq 1$.
On the other hand, we invoke Markov's inequality to upper bound $\scrT_2$:
\begin{align*}
	\scrT_2 \leq \exp\left\{ t\xi - t\gamma \sum_{j\in[r]} \|(I_k^\leftarrow - I_k)u_j\|^2 \right\} \cdot \bbE\left[ \exp\left\{ t\sum_{j\in[r]} \theta_j N_{\x, j}^\top (I_k^\leftarrow - I_k) N_{\y, j} \right\} \right].
\end{align*}
where $t > 0$ is a constant whose value will be specified later. By nearly identical arguments as those appeared in Case A, we have
\begin{align*}
	\bbE\bigg[\exp\bigg\{t \sum_{j\in[r]}\theta_j N_{\x,j}^\top (I_k^\leftarrow - I_k)N_{\y,j}\bigg\}\bigg]
	& = \prod_{j\in[r]} \frac{1}{|a_j^k - (-b_j)^k|},
\end{align*}
where $a_j, b_j$ are given by \eqref{eq:core_upper_bound_aj_bj}.
Let us set
$$
	t = \sqrt{\frac{1 - \ep^2}{4}}
$$
for some $\ep \in (0, 1)$ whose value will be determined later. Under such a choice of $t$, we have
$
\sqrt{1-4t^2\theta_j^2} = \sqrt{1 - (1-\ep^2)\theta_j^2} \geq \ep.
$
When $k$ is odd, we have
$$
	\frac{1}{|a_j^k - (-b_j)^k|} = \frac{1}{a_j^k + b_j^k} \leq \frac{1}{a_j^k} = \bigg(\frac{1 + \sqrt{1-4t^2\theta_j^2}}{2}\bigg)^{-k} \leq \bigg(\frac{2}{1+\ep}\bigg)^k.
$$
When $k$ is even, we have
\begin{align*}
	\frac{1}{|a_j^k - (-b_j)^k|} & = \frac{1}{a_j^k - b_j^k} = \frac{1}{(a_j - b_j)(a_j^{k-1} + a_j^{k-2}b_j + \cdots + a_j b_j^{k-2} + b_j^{k-1})} \leq \frac{1}{(a_j-b_j)a_j^{k-1}}\\
	& = \frac{1}{\sqrt{1 - 4t^2\theta_j^2} a_j^{k-1}} \leq \frac{1}{\ep} \bigg(\frac{2}{1 + \ep}\bigg)^{k-1}.
\end{align*}	
Thus, for any $k\geq 2$, we have
$$
	\frac{1}{|a_j^k - (-b_j)^k|} \leq \frac{2^k}{\ep}
$$
Hence, we arrive at
\begin{align*}
	\scrT_2 \leq \exp\left\{\frac{\sqrt{1-\ep^2}}{2}\xi - \frac{\sqrt{1-\ep^2}}{2} \gamma \sum_{j\in[r]}  \|(I_k - I_k^\leftarrow)u_j\|^2
	+ \calO\left((r)(k + \log (1/\ep))\right)
	\right\}.
\end{align*}
We choose $\ep = r/\tilde\beta^2 = o(1)$. Then, since
$$
	\sum_{j\in[r]} \|(I_k - I_k^\leftarrow)u_j\|^2 \geq k\tilde\beta^2,
$$
we get
$$
	\frac{r(k+\log(1/\ep))}{\sum_{j\in[r]} \|(I_k - I_k^\leftarrow)u_j\|^2} \leq \frac{r}{\tilde\beta^2} + \frac{1}{k} \cdot \frac{\log\frac{\tilde\beta^2}{r}}{\frac{\tilde\beta^2}{r}} = o(1).
$$
Hence, as long as $\gamma \gtrsim 1$,
$$
	\scrT_2 \leq \exp\left\{\calO(\xi) - \frac{1-o(1)}{2}\gamma\sum_{j\in[r]}  \|(I_k - I_k^\leftarrow)u_j\|^2  \right\},
$$
which further implies
\begin{align*}
	& \bbP\left( \sum_{j\in[r]}\theta_j W_j^\top (I_k^\leftarrow - I_k) Z_j + \xi \geq 0\right)\\
	& \leq \exp\left\{-\frac{(1/2-\gamma)^2}{4}\sum_{j\in[r]} \|(I_k - I_k^\leftarrow)u_j\|^2 \right\}\\
	& \qquad \times \exp\left\{\calO(\xi) - \frac{1-o(1)}{2}\gamma\sum_{j\in[r]} \|(I_k - I_k^\leftarrow)u_j\|^2  \right\}.
\end{align*}	
We choose $\gamma$ such that
$
	\frac{(1/2-\gamma)^2}{4} = \frac{\gamma}{2},
$
namely $\gamma = (3-2\sqrt{2})/2\gtrsim 1$. Under such a choice of $\gamma$, we get
\begin{align*}
	& \bbP\left( \sum_{j\in[r]}\theta_j W_j^\top (I_k^\leftarrow- I_k) Z_j + \xi \geq 0\right)\\
	& \leq (1 + e^{\calO(\xi)}) \exp\left\{-(1-o(1))\frac{3-2\sqrt{2}}{4}\sum_{j\in[r]} \|(I_k - I_k^\leftarrow)u_j\|^2 \right\} \\
	& \leq 2 \exp\left\{-(1-o(1))\frac{3-2\sqrt{2}}{4}\sum_{j\in[r]}  \|(I_k - I_k^\leftarrow)u_j\|^2 + \calO(\xi)\right\} \\
	& = \exp\left\{-(1-o(1))\frac{3-2\sqrt{2}}{4}\sum_{j\in[r]}  \|(I_k - I_k^\leftarrow)u_j\|^2 + \calO(\xi)\right\}.
\end{align*}
The proof is concluded.

\subsubsection{Proof of Proposition \ref{prop:frob_norm_concentration}} \label{prf:prop:frob_norm_concentration}
Recall that we can write
$$
	X_{i_{1:k},\bigcdot} = U_{i_{1:k},\bigcdot} D V^\top + \sigma_\x (N_\x)_{i_{1:k}, \bigcdot}, \qquad 
	Y_{\pi^\star{i_{1:k}},\bigcdot} = U_{i_{1:k},\bigcdot} D V^\top + \sigma_\y (N_\y)_{i_{1:k}, \bigcdot}, 
$$
where $(N_\x)_{i_{1:k}, \bigcdot}, (N_\y)_{i_{1:k}, \bigcdot} \in \bbR^{k\times p}$ are independent and have i.i.d.~$N(0, 1)$ entries. We start by decomposing
\begin{align*}
	& \|X_{i_{1:k},\bigcdot}^\top (I_k^\leftarrow - I_k)Y_{\pi^\star_{i_{1:k}},\bigcdot} \|_F\\
	& = \|[U_{i_{1:k},\bigcdot} D V + \sigma_\x (N_\x)_{i_{1:k}, \bigcdot}]^\top (I_k^\leftarrow - I_k)[U_{i_{1:k},\bigcdot} D V^\top + \sigma_\y (N_\y)_{i_{1:k}, \bigcdot}] \|_F \\
	& \leq \underbrace{\|[U_{i_{1:k},\bigcdot} D V^\top]^\top (I_k^\leftarrow-I_k)U_{i_{1:k},\bigcdot} D V^\top\|_F}_{\scrT_1} + \underbrace{\|\sigma_\x (N_\x)_{i_{1:k}, \bigcdot}^\top (I_k^\leftarrow-I_k)U_{i_{1:k},\bigcdot}D V^\top\|_F}_{\scrT_2} \\
	\label{eq:frob_norm_concentration_decomp}
	& \qquad +\underbrace{\|[U_{i_{1:k},\bigcdot} D V^\top]^\top (I_k^\leftarrow - I_k) \sigma_\y (N_\y)_{i_{1:k}, \bigcdot}\|_F}_{\scrT_3} + \underbrace{\|\sigma_\x (N_\x)_{i_{1:k}, \bigcdot}^\top (I_k^\leftarrow - I_k)\sigma_\y (N_\y)_{i_{1:k}, \bigcdot}\|_F}_{\scrT_4}.\numberthis
\end{align*}
Since $V \in O_{p, r}$, we have
\begin{align*}
	\scrT_1 = \|D U_{i_{1:k},\bigcdot}^\top (I_k^\leftarrow - I_k) U_{i_{1:k},\bigcdot} D \|_F.
\end{align*}
For the second term in the right-hand side of \eqref{eq:frob_norm_concentration_decomp}, we have
\begin{align*}
	\scrT_2^2 
	& = \sigma_\x^2 \tr \left( (N_\x)_{i_{1:k}, \bigcdot}^\top (I_k^\leftarrow - I_k)U_{i_{1:k},\bigcdot} D V^\top V D U_{i_{1:k},\bigcdot}^\top (I^\leftarrow_k - I_k)^\top (N_\x)_{i_{1:k}, \bigcdot} \right) \\
	& = \sigma_\x^2 \tr\left( (N_\x)_{i_{1:k}, \bigcdot}^\top(I_k^\leftarrow - I_k)U_{i_{1:k},\bigcdot} D^2 U_{i_{1:k},\bigcdot}^\top (I_k^\leftarrow - I_k)^\top (N_\x)_{i_{1:k}, \bigcdot} \right) \\
	& = \sigma_\x^2 \vect((N_\x)_{i_{1:k}, \bigcdot})^\top \left[ I_{p} \otimes \left((I_k^\leftarrow-I_k)U_{i_{1:k},\bigcdot} D^2 U_{i_{1:k},\bigcdot}^\top (I_k^\leftarrow - I_k)^\top \right) \right] \vect((N_\x)_{i_{1:k}, \bigcdot}).
\end{align*}
Note that
\begin{align*}
	&\bbE[\scrT_2^2]  = \sigma_\x^2 \tr\left(I_{p} \otimes \left((I_k^\leftarrow-I_k)U_{i_{1:k},\bigcdot} D^2 U_{i_{1:k},\bigcdot}^\top (I_k^\leftarrow - I_k)^\top \right)\right) = p \sigma_\x^2 \|(I_k^\leftarrow-I_k)U_{i_{1:k},\bigcdot} D\|^2_F.
\end{align*}
Meanwhile, we have
\begin{align*}
	\left\|I_{p} \otimes \left((I_k^\leftarrow-I_k)U_{i_{1:k},\bigcdot} D^2 U_{i_{1:k},\bigcdot}^\top (I_k^\leftarrow - I_k)^\top \right)\right\|_F 
	& =  \sqrt{p} \|(I_k^\leftarrow - I_k)U_{i_{1:k},\bigcdot} D^2 U_{i_{1:k},\bigcdot}^\top (I_k^\leftarrow - I_k)^\top\|_F \\
	\left\|I_{p} \otimes \left((I_k^\leftarrow-I_k)U_{i_{1:k},\bigcdot} D^2 U_{i_{1:k},\bigcdot}^\top (I_k^\leftarrow - I_k)^\top \right)\right\|_2 
	& = \|(I_k^\leftarrow - I_k)U_{i_{1:k},\bigcdot} D^2 U_{i_{1:k},\bigcdot}^\top (I_k^\leftarrow - I_k)^\top\|_2.
\end{align*}	
Invoking Hanson-Wright inequality (see Lemma \ref{lemma:hs_ineq}), 
for any $\delta \in(0, 1)$, we have
\begin{align*}
	\scrT_2^2 & \leq \sigma_\x^2 \bigg[ p\|(I_k^\leftarrow - I_k)U_{i_{1:k},\bigcdot} D\|_F^2 \\
	& \qquad + \calO \bigg( \sqrt{p}\|(I_k^\leftarrow - I_k)U_{i_{1:k},\bigcdot}D^2U_{i_{1:k},\bigcdot}^\top(I_k^\leftarrow-I_k)^\top\|_F \sqrt{\log(1/\delta)}\\
	& \qquad\qquad\qquad\qquad\qquad + \|(I_k^\leftarrow-I_k)U_{i_{1:k},\bigcdot}D^2U_{i_{1:k},\bigcdot}^\top(I_k^\leftarrow-I_k)^\top\|_2 \log(1/\delta)\bigg)\bigg]
\end{align*}
with probability at least $1-\delta$. The above inequality translates to 
\begin{align*}
	\scrT_2 & \leq \sigma_\x \bigg[ \sqrt{p}\|(I_k^\leftarrow - I_k)U_{i_{1:k},\bigcdot} D\|_F \\
	& \qquad + \calO \bigg( {p}^{1/4}\|(I_k^\leftarrow - I_k)U_{i_{1:k},\bigcdot}D^2U_{i_{1:k},\bigcdot}^\top(I_k^\leftarrow-I_k)^\top\|_F^{1/2} (\log(1/\delta))^{1/4}\\
	& \qquad\qquad\qquad\qquad\qquad + \|(I_k^\leftarrow-I_k)U_{i_{1:k},\bigcdot}D^2U_{i_{1:k},\bigcdot}^\top(I_k^\leftarrow-I_k)^\top\|_2^{1/2} (\log(1/\delta))^{1/2}\bigg)\bigg]
\end{align*}
with probability at least $1-\delta$.
The treatment for the third term in the right-hand side of \eqref{eq:frob_norm_concentration_decomp} is similar to our treatment for the second term, namely we write $\scrT_3^2$ into a Gaussian quadratic form and invoke Hanson-Wright inequality to get
\begin{align*}
	\scrT_3 & \leq \sigma_\y \bigg[ \sqrt{p}\|(I_k^\leftarrow - I_k)U_{i_{1:k},\bigcdot} D\|_F \\
	& \qquad + \calO \bigg( {p}^{1/4}\|(I_k^\leftarrow - I_k)U_{i_{1:k},\bigcdot}D^2U_{i_{1:k},\bigcdot}^\top(I_k^\leftarrow-I_k)^\top \|_F^{1/2} (\log(1/\delta))^{1/4}\\
	& \qquad\qquad\qquad\qquad\qquad + \|(I_k^\leftarrow-I_k)U_{i_{1:k},\bigcdot}D^2U_{i_{1:k},\bigcdot}^\top(I_k^\leftarrow-I_k)^\top\|_2^{1/2} (\log(1/\delta))^{1/2}\bigg)\bigg]\\
	& = \sigma_\y \bigg[ \sqrt{p}\|(I_k^\leftarrow - I_k)U_{i_{1:k},\bigcdot} D\|_F \\
	& \qquad + \calO \bigg( {p}^{1/4}\|(I_k^\leftarrow - I_k)U_{i_{1:k},\bigcdot}D^2U_{i_{1:k},\bigcdot}^\top(I_k^\leftarrow-I_k)^\top\|_F^{1/2} (\log(1/\delta))^{1/4}\\
	& \qquad\qquad\qquad\qquad\qquad + \|(I_k^\leftarrow-I_k)U_{i_{1:k},\bigcdot}D^2U_{i_{1:k},\bigcdot}^\top(I_k^\leftarrow-I_k)^\top\|_2^{1/2} (\log(1/\delta))^{1/2}\bigg)\bigg],
\end{align*}
where the last equality is by $\|(I_k^\leftarrow - I_k)U_{i_{1:k},\bigcdot} D\|_F = \|(I_k^\leftarrow - I_k)U_{i_{1:k},\bigcdot} D\|_F$ and $\|A^\top A\|_F= \|AA^\top\|_F$, $\|A^\top A\|_2 = \|AA^\top\|_2$ for any matrix $A$ of suitable sizes.
To deal with the fourth term in the right-hand side of \eqref{eq:frob_norm_concentration_decomp}, we proceed by writing it as
\begin{align*}
	\scrT_4  &=\sigma_\x\sigma_\y \sqrt{ \tr \left( (N_\x)_{i_{1:k}, \bigcdot}^\top (I_k^\leftarrow - I_k) (N_\y)_{i_{1:k}, \bigcdot} (N_\y)_{i_{1:k}, \bigcdot}^\top (I_k^\leftarrow - I_k)^\top (N_\x)_{i_{1:k}, \bigcdot} \right)} \\
	& = \sigma_\x\sigma_\y\sqrt{ \vect((N_\x)_{i_{1:k}, \bigcdot})^\top \left[I_p \otimes \left( (I_k^\leftarrow - I_k) (N_\y)_{i_{1:k}, \bigcdot} (N_\y)_{i_{1:k}, \bigcdot}^\top (I_k^\leftarrow - I_k)^\top \right)\right] \vect((N_\x)_{i_{1:k}, \bigcdot})}.
\end{align*}
Conditional on the randomness of $(N_\y)_{i_{1:k}, \bigcdot}$, we invoke Hanson-Wright inequality to conclude that with probability at least $1-\delta$, we have
\begin{align*}
	\scrT_4 & \leq {\sigma_\x\sigma_\y} \bigg[\sqrt{p}\|(I_k^\leftarrow - I_k)(N_\y)_{i_{1:k}, \bigcdot}\|_F 
	+ \calO\bigg( p^{1/4} \|(I_k^\leftarrow-I_k)(N_\y)_{i_{1:k}, \bigcdot} (N_\y)_{i_{1:k}, \bigcdot}^\top (I_k^\leftarrow - I_k)^\top\|_F^{1/2} (\log(1/\delta))^{1/4} \\
	& \qquad\qquad\qquad\qquad\qquad\qquad\qquad\qquad\qquad + \|(I_k^\leftarrow-I_k)(N_\y)_{i_{1:k}, \bigcdot} (N_\y)_{i_{1:k}, \bigcdot}^\top (I_k^\leftarrow - I_k)^\top\|_2^{1/2} (\log(1/\delta))^{1/2}\bigg)\bigg]\\
	& \leq {\sigma_\x\sigma_\y} \bigg[
		2\sqrt{p} \|(N_\y)_{i_{1:k}, \bigcdot}\|_F
		+\calO\bigg(
			p^{1/4} \|(I_k^\leftarrow - I_k)(N_\y)_{i_{1:k}, \bigcdot}\|_2^{1/2} \|(I_k^\leftarrow - I_k)(N_\y)_{i_{1:k}, \bigcdot}\|_F^{1/2} (\log(1/\delta))^{1/4}\\
			& \qquad\qquad\qquad\qquad\qquad\qquad\qquad\qquad\qquad + \|(I_k^\leftarrow - I_k)(N_\y)_{i_{1:k}, \bigcdot}\|_2 (\log(1/\delta))^{1/2}
		\bigg)
	\bigg]\\
	&\leq {\sigma_\x\sigma_\y} \bigg[ 2\sqrt{p}\|(N_\y)_{i_{1:k}, \bigcdot}\|_F 
	+ \calO\bigg(
			p^{1/4} \|(N_\y)_{i_{1:k}, \bigcdot}\|_2^{1/2} \|(N_\y)_{i_{1:k}, \bigcdot}\|_F^{1/2} (\log(1/\delta))^{1/4} + \|(N_\y)_{i_{1:k}, \bigcdot}\|_2 (\log(1/\delta))^{1/2}
		\bigg) 
	\bigg],
\end{align*}
where in the last two inequalities, we have used $\|AB\|_F\leq \|A\|_2 \|B\|_F, \|A B\|_2 \leq \|A\|_2 \|B\|_2$ for matrices $A, B$ of suitable sizes and $\|I_k^\leftarrow - I_k\|_2 \leq \|I_k^\leftarrow\|_2 \|I_k\|_2 \leq 2$.
Since $\|(N_\y)_{i_{1:k}, \bigcdot}\|_F^2\sim \chi^2_{p k}$, by the $\chi^2$ tail bound (see Lemma \ref{lemma:chisq_tail}), we have
$$
	\|(N_\y)_{i_{1:k}, \bigcdot}\|_F \leq \sqrt{p k} + \calO\left( [p k \log(1/\delta)]^{1/4} + (\log(1/\delta))^{1/2} \right)
$$
with probability at least $1-\delta$. Meanwhile, invoking the tail bound for the operator norm of Gaussian Wigner matrices (see Lemma \ref{lemma:gaussian_wigner_op_norm}), we have
$$
	\|(N_\y)_{i_{1:k}, \bigcdot}\|_2 \lesssim \sqrt{k} + \sqrt{p} + \sqrt{\log(1/\delta)}
$$
with probability at least $1-\delta$.
Thus, a union bound gives that with probability at least $1-3\delta$,
\begin{align*}
	\scrT_4 & \lesssim \sigma_\x\sigma_\y \bigg[
		p k^{1/2} + p^{3/4} k^{1/4} (\log(1/\delta))^{1/4} + p^{1/2}(\log(1/\delta))^{1/2} \\
		& \qquad + p^{1/4} (\log(1/\delta))^{1/4} \left( k^{1/4}+p^{1/4}+(\log(1/\delta))^{1/4} \right)
		\left( (p k)^{1/4}+ (p k\log(1/\delta))^{1/8} + (\log(1/\delta))^{1/4} \right)\\
		& \qquad + (\log(1/\delta))^{1/2}\left( k^{1/2}+ p^{1/2}+ (\log(1/\delta))^{1/2} \right)
	\bigg]\\
	& = \sigma_\x\sigma_\y
	\bigg(
		p k^{1/2} + p^{3/4} k^{1/4} (\log(1/\delta))^{1/4} + p^{1/2} (\log(1/\delta))^{1/2}\\
		& \qquad + p^{1/2}k^{1/2} (\log(1/\delta))^{1/4} + p^{3/8} k^{3/8}(\log(1/\delta))^{3/8} + p^{1/4}k^{1/4} (\log(1/\delta))^{1/2} \\
		& \qquad + p^{3/4}k^{1/4} (\log(1/\delta))^{1/4} + p^{5/8}k^{1/8}(\log(1/\delta))^{3/8} + p^{1/2}(\log(1/\delta))^{1/2} \\
		& \qquad  + p^{1/2}k^{1/4}(\log(1/\delta))^{1/2} + p^{3/8}k^{1/8}(\log(1/\delta))^{5/8} + p^{1/4} (\log(1/\delta))^{3/4}\\
		& \qquad k^{1/2}(\log(1/\delta))^{1/2} + p^{1/2}(\log(1/\delta))^{1/2} +(\log(1/\delta)
	\bigg)\\
	& \lesssim  \sigma_\x\sigma_\y
	\bigg(
		p k^{1/2} + p^{3/4}k^{1/4}(\log(1/\delta))^{1/4} + p^{5/8}k^{1/8}(\log(1/\delta))^{3/8} \\
		& \qquad + p^{1/2}k^{1/2}(\log(1/\delta))^{1/4} 
		+ p^{1/2}k^{1/4}(\log(1/\delta))^{1/2} 
		+ p^{3/8}k^{3/8}(\log(1/\delta))^{3/8} + p^{3/8}k^{1/8}(\log(1/\delta))^{5/8} \\
		&\qquad + p^{1/4}k^{1/4}(\log(1/\delta))^{1/2} + p^{1/4} (\log(1/\delta))^{3/4} 
		+ k^{1/2}(\log(1/\delta))^{1/2} + \log(1/\delta)
	\bigg)
\end{align*}
In summary, with probability at least $1-5\delta$, we have
\begin{align*}
	& \frac{\|X_{i_{1:k},\bigcdot}^\top (I_k^\leftarrow - I_k)Y_{\pi^\star_{i_{1:k}},\bigcdot} \|_F}{\sigma_\x\sigma_\y} \\
	& \leq \frac{\scrT_1+\scrT_2+\scrT_3+\scrT_4}{\sigma_\x\sigma_\y}\\
	& \lesssim \frac{\|D U_{i_{1:k},\bigcdot}^\top (I_k^\leftarrow - I_k) U_{i_{1:k},\bigcdot} D \|_F}{\sigma_\x\sigma_\y}  \\
	& \qquad +  \bigg(\frac{1}{\sigma_\x} + \frac{1}{\sigma_\y}\bigg) 
	\bigg(
		p^{1/2} \|(I_k^\leftarrow - I_k)U_{i_{1:k},\bigcdot} D\|_F
		+ p^{1/4} \|(I_k^\leftarrow - I_k)U_{i_{1:k},\bigcdot}D^2U_{i_{1:k},\bigcdot}^\top(I_k^\leftarrow-I_k)^\top\|_F^{1/2} (\log(1/\delta))^{1/4}\\
		& \qquad \qquad 
		+ \|(I_k^\leftarrow-I_k)U_{i_{1:k},\bigcdot}D^2U_{i_{1:k},\bigcdot}^\top(I_k^\leftarrow-I_k)^\top\|_2^{1/2} (\log(1/\delta))^{1/2}
	\bigg)\\
	& \qquad + 
	\bigg( 
		p k^{1/2} + p^{3/4}k^{1/4}(\log(1/\delta))^{1/4} + p^{5/8}k^{1/8}(\log(1/\delta))^{3/8} \\
		& \qquad \qquad + p^{1/2}k^{1/2}(\log(1/\delta))^{1/4} 
		+ p^{1/2}k^{1/4}(\log(1/\delta))^{1/2} 
		+ p^{3/8}k^{3/8}(\log(1/\delta))^{3/8} + p^{3/8}k^{1/8}(\log(1/\delta))^{5/8} \\
		& \qquad \qquad + p^{1/4}k^{1/4}(\log(1/\delta))^{1/2} + p^{1/4} (\log(1/\delta))^{3/4} 
		+ k^{1/2}(\log(1/\delta))^{1/2} + \log(1/\delta)
	\bigg).
\end{align*}
The proof is concluded.

\subsubsection{Proof of Lemma \ref{lemma:negligible_exponent_uniform_bound}} \label{prf:lemma:negligible_exponent_uniform_bound}
For notational simplicity, we let 
$$
\tilde \xi :=  \frac{\sigma_{\min} \xi }{\sigma_{\max}}.
$$
We would like to find conditions under which
\begin{align*}
	& \frac{\|(I_k^\leftarrow - I_k) U_{i_{1:k},\bigcdot} D\|_F^2}{\sigma_{\max}^2}\\
	& \gg \tilde{\xi} \cdot \bigg[
	\frac{\|D U_{i_{1:k},\bigcdot}^\top (I_k^\leftarrow - I_k) U_{i_{1:k},\bigcdot} D \|_F}{\sigma_\x\sigma_\y}  \\
	& \qquad +  \bigg(\frac{1}{\sigma_\x} + \frac{1}{\sigma_\y}\bigg) 
	\bigg(
		p^{1/2} \|(I_k^\leftarrow - I_k)U_{i_{1:k},\bigcdot} D\|_F\\
		& \qquad \qquad 
		+ p^{1/4} \|(I_k^\leftarrow - I_k)U_{i_{1:k},\bigcdot}D^2U_{i_{1:k},\bigcdot}^\top(I_k^\leftarrow-I_k)^\top\|_F^{1/2} (\log(1/\delta_{i_{1:k}}))^{1/4}\\
		& \qquad \qquad 
		+ \|(I_k^\leftarrow-I_k)U_{i_{1:k},\bigcdot}D^2U_{i_{1:k},\bigcdot}^\top(I_k^\leftarrow-I_k)^\top\|_2^{1/2} (\log(1/\delta_{i_{1:k}}))^{1/2}
	\bigg)\\
	& \qquad + 
	\bigg( 
		p k^{1/2} + p^{3/4}k^{1/4}(\log(1/\delta_{i_{1:k}}))^{1/4} + p^{5/8}k^{1/8}(\log(1/\delta_{i_{1:k}}))^{3/8} \\
		& \qquad \qquad + p^{1/2}k^{1/2}(\log(1/\delta_{i_{1:k}}))^{1/4} 
		+ p^{1/2}k^{1/4}(\log(1/\delta_{i_{1:k}}))^{1/2} \\
		& \qquad \qquad + p^{3/8}k^{3/8}(\log(1/\delta_{i_{1:k}}))^{3/8} + p^{3/8}k^{1/8}(\log(1/\delta_{i_{1:k}}))^{5/8} \\
		& \qquad \qquad + p^{1/4}k^{1/4}(\log(1/\delta_{i_{1:k}}))^{1/2} + p^{1/4} (\log(1/\delta_{i_{1:k}}))^{3/4} 
		+ k^{1/2}(\log(1/\delta_{i_{1:k}}))^{1/2} + \log(1/\delta_{i_{1:k}})
	\bigg)
	\bigg]
\end{align*}	
uniformly over all cycles of size at least $k^\star$,
where
$$
	\log(1/\delta_{i_{1:k}}) \asymp \frac{\|(I_k^\leftarrow - I_k) U_{i_{1:k},\bigcdot} D\|_F^2}{\sigma_{\max}^2} \geq k^\star \beta^2.
$$
There are fifteen terms in the right-hand side of the desired inequality. And we deal with each term separately.
\begin{enumerate}
	\item We would like
	$$
		\tilde{\xi} \cdot \frac{\|D U_{i_{1:k},\bigcdot}^\top (I_k^\leftarrow - I_k) U_{i_{1:k},\bigcdot} D \|_F}{\sigma_\x\sigma_\y} \ll 
		\frac{\|(I_k^\leftarrow - I_k)U_{i_{1:k},\bigcdot} D\|_F^2}{\sigma_{\max}^2} 
	$$
	uniformly over all cycles of size at least $k^\star$.
	The left-hand side above is upper bounded by
	$$
		\tilde{\xi} \cdot \frac{d_1}{\sigma_{\min}} \frac{\|(I_k^\leftarrow - I_k)U_{i_{1:k},\bigcdot} D\|_F}{\sigma_{\max}}.
	$$
	Thus, the desired inequality would hold if
	$$
		\tilde{\xi} \cdot \frac{d_1}{\sigma_{\min}} \ll 
		\frac{\|(I_k^\leftarrow - I_k)U_{i_{1:k},\bigcdot} D\|_F}{\sigma_{\max}}.
	$$
	Since the right-hand side above is lower bounded by $\sqrt{k}\beta \geq \sqrt{k^\star}\beta$, it suffices to require
	$$
		\tilde{\xi} \ll \frac{\sqrt{k^\star}\beta}{d_1/\sigma_{\min}}.
	$$
	\item We would like
	$$
		\tilde{\xi} \cdot \bigg(\frac{1}{\sigma_\x} + \frac{1}{\sigma_\y}\bigg) p^{1/2} \|(I_k^\leftarrow - I_k)U_{i_{1:k},\bigcdot} D\|_F \ll
		\frac{\|(I_k^\leftarrow - I_k)U_{i_{1:k},\bigcdot} D\|_F^2}{\sigma_{\max}^2} 
	$$
	uniformly over all cycles of size at least $k^\star$.
	The left-hand side above can be expressed as
	$$
		\tilde \xi \cdot \frac{\sigma_{\min} + \sigma_{\max}}{\sigma_{\min}} \cdot p^{1/2} \cdot \frac{\|(I_k^\leftarrow - I_k)U_{i_{1:k},\bigcdot} D\|_F}{\sigma_{\max}} \leq 2\tilde \xi \cdot \frac{\sigma_{\max}}{\sigma_{\min}} \cdot p^{1/2} \frac{\|(I_k^\leftarrow - I_k)U_{i_{1:k},\bigcdot} D\|_F}{\sigma_{\max}} .
	$$
	Thus, the desired inequality is implied by
	$$
		\tilde{\xi} \ll p^{-1/2} \cdot \frac{\sigma_{\min}}{\sigma_{\max}}  \cdot
		\frac{\|(I_k^\leftarrow - I_k)U_{i_{1:k},\bigcdot} D\|_F}{\sigma_{\max}}.
	$$
	Since ${\|(I_k^\leftarrow - I_k)U_{i_{1:k},\bigcdot} D\|_F}/{\sigma_{\max}}\geq \sqrt{k}\beta \geq \sqrt{k^\star}\beta$, it suffices to require
	$$
		\tilde{\xi} \ll \frac{\sqrt{k^\star}\beta}{\sqrt{p}} \cdot \frac{\sigma_{\min}}{\sigma_{\max}}.
	$$
	\item 
	We would like
	\begin{align*}
		& \tilde{\xi} \cdot \bigg(\frac{1}{\sigma_\x} + \frac{1}{\sigma_\y}\bigg)p^{1/4} \|(I_k^\leftarrow - I_k)U_{i_{1:k},\bigcdot}D^2U_{i_{1:k},\bigcdot}^\top(I_k^\leftarrow-I_k)^\top\|_F^{1/2} (\log(1/\delta_{i_{1:k}}))^{1/4}\\
		& \ll 
		\frac{\|(I_k^\leftarrow - I_k) U_{i_{1:k},\bigcdot} D\|_F^2}{\sigma_{\max}^2}
	\end{align*}	
	uniformly over all cycles of size at least $k^\star$.
	Recalling the definition of $\delta_{i_{1:k}}$, the above inequality is equivalent to
	\begin{align*}
		& \tilde{\xi} \cdot \bigg(\frac{1}{\sigma_\x} + \frac{1}{\sigma_\y}\bigg)p^{1/4} \|(I_k^\leftarrow - I_k)U_{i_{1:k},\bigcdot}D^2U_{i_{1:k},\bigcdot}^\top(I_k^\leftarrow-I_k)^\top\|_F^{1/2}
		\ll 
		\left(\frac{\|(I_k^\leftarrow - I_k) U_{i_{1:k},\bigcdot} D\|_F^2}{\sigma_{\max}^2}\right)^{3/4}.
	\end{align*}	
	We can upper bound the left-hand side above by
	$$
		2 \tilde{\xi} \cdot \frac{\sigma_{\max}}{\sigma_{\min}}\cdot  
		p^{1/4} \cdot
		\frac{\|(I_k^\leftarrow - I_k)U_{i_{1:k},\bigcdot} D\|_F}{\sigma_{\max}}.
	$$
	Thus, the desired inequality is implied by
	$$
		\tilde \xi \ll \frac{\sigma_{\min}}{\sigma_{\max} p^{1/4}} \cdot \left(\frac{\|(I_k^\leftarrow - I_k) U_{i_{1:k},\bigcdot} D\|_F}{\sigma_{\max}}\right)^{1/2}.
	$$
	Since ${\|(I_k^\leftarrow - I_k)U_{i_{1:k},\bigcdot} D\|_F}/{\sigma_{\max}}\geq \sqrt{k}\beta \geq \sqrt{k^\star}\beta$, it suffices to require
	$$
		\tilde{\xi} \ll \frac{\sigma_{\min}}{\sigma_{\max}} \cdot \bigg(\frac{\sqrt{k^\star}\beta}{\sqrt{p}}\bigg)^{1/2}.
	$$
	\item 
	We would like 
	\begin{align*}
		& \tilde{\xi} \cdot \bigg(\frac{1}{\sigma_\x} + \frac{1}{\sigma_\y}\bigg)
		\|(I_k^\leftarrow-I_k)U_{i_{1:k},\bigcdot}D^2U_{i_{1:k},\bigcdot}^\top(I_k^\leftarrow-I_k)^\top\|_2^{1/2} (\log(1/\delta_{i_{1:k}}))^{1/2}\\
		& \ll
		\frac{\|(I_k^\leftarrow - I_k)U_{i_{1:k},\bigcdot} D\|_F^2}{\sigma_{\max}^2}
	\end{align*}	
	uniformly over all cycles of size at least $k^\star$.
	Recalling the definition of $\delta_{i_{1:k}}$, the above inequality is equivalent to
	$$
		\tilde{\xi} \cdot \bigg(\frac{1}{\sigma_\x} + \frac{1}{\sigma_\y}\bigg)
		\|(I_k^\leftarrow-I_k)U_{i_{1:k},\bigcdot}D^2U_{i_{1:k},\bigcdot}^\top(I_k^\leftarrow-I_k)^\top\|_2^{1/2}
		\ll
		\left(\frac{\|(I_k^\leftarrow - I_k)U_{i_{1:k},\bigcdot} D\|_F^2}{\sigma_{\max}^2}\right)^{1/2}
	$$
	The left-hand side is upper bounded by
	$$
		2 \tilde{\xi} \cdot \frac{\sigma_{\max}}{\sigma_{\min}} \cdot 
		\frac{\|(I_k^\leftarrow - I_k)U_{i_{1:k},\bigcdot} D\|_F}{\sigma_{\max}}.
	$$
	Thus, it suffices to require
	$$
		\tilde{\xi} \ll \frac{\sigma_{\min}}{\sigma_{\max}}.
	$$
	\item 
	We would like 
	$$
		\tilde{\xi} \cdot p k^{1/2} 
		\ll \frac{\|(I_k^\leftarrow - I_k)U_{i_{1:k},\bigcdot} D\|_F^2}{\sigma_{\max}^2}
	$$
	uniformly over all cycles of size at least $k^\star$.
	Since the right-hand side is lower bounded by $k \beta^2$, it suffices to require
	$
		\tilde{\xi} p \ll \sqrt{k} \beta^2,
	$
	which is further implied by
	$$
		\tilde{\xi} \ll \frac{\sqrt{k^\star}\beta^2}{p}.
	$$
	\item We would like
	$$
		\tilde{\xi} \cdot p^{3/4} k^{1/4} (\log(1/\delta_{i_{1:k}}))^{1/4} \ll 
		\frac{\|(I_k^\leftarrow - I_k)U_{i_{1:k},\bigcdot} D\|_F^2}{\sigma_{\max}^2}
	$$
	uniformly over all cycles of size at least $k^\star$.
	By the definition of $\delta_{i_{1:k}}$, the above inequality is equivalent to
	$$
		\tilde{\xi} \cdot p^{3/4} k^{1/4} \ll 
		\left(\frac{\|(I_k^\leftarrow - I_k)U_{i_{1:k},\bigcdot} D\|_F^2}{\sigma_{\max}^2}\right)^{3/4},
	$$
	which is further implied by
	$
		\tilde{\xi} \cdot p^{3/4} k^{1/4} \ll k^{3/4} (\beta^2)^{3/4}.
	$
	Thus, it suffices to require
	$$
		\tilde{\xi} \ll \sqrt{k^\star} \left(\frac{\beta^2}{p}\right)^{3/4}.
	$$
	\item We would like
	$$
		\tilde{\xi} \cdot p^{5/8} k^{1/8} (\log(1/\delta_{i_{1:k}}))^{3/8}
		\ll 
		\frac{\|(I_k^\leftarrow - I_k)U_{i_{1:k},\bigcdot} D\|_F^2}{\sigma_{\max}^2}.
	$$
	uniformly over all cycles of size at least $k^\star$.
	By the definition of $\delta_{i_{1:k}}$, the above inequality is equivalent to
	$$
		\tilde{\xi} \cdot p^{5/8} k^{1/8} 
		\ll 
		\left(\frac{\|(I_k^\leftarrow - I_k)U_{i_{1:k},\bigcdot} D\|_F^2}{\sigma_{\max}^2}\right)^{5/8},
	$$
	which is further implied by $\tilde{\xi} \cdot p^{5/8}k^{1/8} \ll k^{5/8}(\beta^2)^{5/8}$.
	Thus, it suffices to require 
	$$
		\tilde{\xi} \ll \sqrt{k^\star}\left(\frac{\beta^2}{p}\right)^{5/8}.
	$$
	\item We would like
	$$
		\tilde{\xi} p^{1/2}k^{1/2}(\log(1/\delta_{i_{1:k}}))^{1/4}
		\ll
		\frac{\|(I_k^\leftarrow - I_k)U_{i_{1:k},\bigcdot} D\|_F^2}{\sigma_{\max}^2}
	$$
	uniformly over all cycles of size at least $k^\star$.
	By the definition of $\delta_{i_{1:k}}$, the above inequality is equivalent to
	$$
		\tilde{\xi} p^{1/2}k^{1/2}
		\ll
		\left(\frac{\|(I_k^\leftarrow - I_k)U_{i_{1:k},\bigcdot} D\|_F^2}{\sigma_{\max}^2}\right)^{3/4},
	$$
	which is further implied by
	$
		\tilde{\xi} \cdot p^{1/2} k^{1/2} \ll  k^{3/4} (\beta^2)^{3/4}.
	$
	Thus, it suffices to require
	$$
		\tilde{\xi} \ll (k^\star)^{1/4} \sqrt{\beta} \bigg(\frac{\beta^2}{p}\bigg)^{1/2} .
	$$
	\item We would like
	$$
		\tilde{\xi} \cdot p^{1/2} k^{1/4} (\log(1/\delta_{i_{1:k}}))^{1/2}
		\ll
		\frac{\|(I_k^\leftarrow - I_k)U_{i_{1:k},\bigcdot} D\|_F^2}{\sigma_{\max}^2}
	$$
	uniformly over all cycles of size at least $k^\star$.
	By the definition of $\delta_{i_{1:k}}$, the above inequality is equivalent to
	$$
		\tilde{\xi} \cdot p^{1/2} k^{1/4}
		\ll
		\left(\frac{\|(I_k^\leftarrow - I_k)U_{i_{1:k},\bigcdot} D\|_F^2}{\sigma_{\max}^2}\right)^{1/2},
	$$
	which is further implied by $\tilde{\xi} \cdot p^{1/2}k^{1/4} \ll k^{1/2} (\beta^2)^{1/2}$. Thus it suffices to require
	$$
		\tilde{\xi} \ll (k^\star)^{1/4} \bigg(\frac{\beta^2}{p}\bigg)^{1/2}.
	$$
	\item We would like
	$$
		\tilde{\xi} \cdot p^{3/8} k^{3/8} (\log(1/\delta_{i_{1:k}}))^{3/8}
		\ll
		\frac{\|(I_k^\leftarrow - I_k)U_{i_{1:k},\bigcdot} D\|_F^2}{\sigma_{\max}^2}
	$$
	uniformly over all cycles of size at least $k^\star$.
	By the definition of $\delta_{i_{1:k}}$, the above inequality is equivalent to
	$$
		\tilde{\xi} \cdot p^{3/8} k^{3/8} 
		\ll
		\left(\frac{\|(I_k^\leftarrow - I_k)U_{i_{1:k},\bigcdot} D\|_F^2}{\sigma_{\max}^2}\right)^{5/8},
	$$
	which is further implied by
	$
		\tilde{\xi} \cdot p^{3/8} k^{3/8} 
		\ll k^{5/8} (\beta^2)^{5/8}.
	$
	Thus, it suffices to require
	$$
		\tilde{\xi} \ll (k^\star)^{1/4} \sqrt{\beta} \bigg(\frac{\beta^2}{p}\bigg)^{3/8}.
	$$
	\item We would like 
	$$
		\tilde{\xi} \cdot p^{3/8} k^{1/8} (\log(1/\delta_{i_{1:k}}))^{5/8}
		\ll
		\frac{\|(I_k^\leftarrow - I_k)U_{i_{1:k},\bigcdot} D\|_F^2}{\sigma_{\max}^2}
	$$
	uniformly over all cycles of size at least $k^\star$.
	By the definition of $\delta_{i_{1:k}}$, the above inequality is equivalent to
	$$
		\tilde{\xi} \cdot p^{3/8} k^{1/8}
		\ll
		\left(\frac{\|(I_k^\leftarrow - I_k)U_{i_{1:k},\bigcdot} D\|_F^2}{\sigma_{\max}^2}\right)^{3/8},
	$$
	which is further implied by
	$
		\tilde{\xi} \cdot p^{3/8} k^{1/8}
		\ll k^{3/8} (\beta^2)^{3/8}.
	$
	Thus, it suffices to require
	$$
		\tilde{\xi} \ll  (k^\star)^{1/4}\bigg(\frac{\beta^2}{p}\bigg)^{3/8}.
	$$
	\item We would like
	$$
		\tilde{\xi} \cdot p^{1/4} k^{1/4}(\log(1/\delta_{i_{1:k}}))^{1/2}
		\ll
		\frac{\|(I_k^\leftarrow - I_k)U_{i_{1:k},\bigcdot} D\|_F^2}{\sigma_{\max}^2}
	$$
	uniformly over all cycles of size at least $k^\star$.
	By the definition of $\delta_{i_{1:k}}$, the above inequality is equivalent to
	$$
		\tilde{\xi} \cdot p^{1/4} k^{1/4}
		\ll
		\left(\frac{\|(I_k^\leftarrow - I_k)U_{i_{1:k},\bigcdot} D\|_F^2}{\sigma_{\max}^2}\right)^{1/2},
	$$
	which is further implied by
	$
		\tilde{\xi} \cdot p^{1/4} k^{1/4} \ll k^{1/2}(\beta^2)^{1/2}.
	$
	Thus, it suffices to require
	$$
		\tilde{\xi} \ll (k^\star)^{1/4} \sqrt{\beta} \bigg(\frac{\beta^2}{p}\bigg)^{1/4}.
	$$
	\item We would like
	$$
		\tilde{\xi} \cdot p^{1/4} (\log(1/\delta_{i_{1:k}}))^{3/4}
		\ll 
		\frac{\|(I_k^\leftarrow - I_k)U_{i_{1:k},\bigcdot} D\|_F^2}{\sigma_{\max}^2}
	$$
	uniformly over all cycles of size at least $k^\star$.
	By the definition of $\delta_{i_{1:k}}$, the above inequality is equivalent to
	$$
		\tilde{\xi} \cdot p^{1/4}
		\ll 
		\left( \frac{\|(I_k^\leftarrow - I_k)U_{i_{1:k},\bigcdot} D\|_F^2}{\sigma_{\max}^2} \right)^{1/4},
	$$
	which is further implied by $\tilde{\xi} \cdot p^{1/4}\ll k^{1/4}(\beta^2)^{1/4}$. Thus, it suffices to require
	$$
		\tilde{\xi} \ll (k^\star)^{1/4} \bigg(\frac{\beta^2}{p}\bigg)^{1/4}.
	$$
	\item We would like
	$$
		\tilde{\xi} \cdot k^{1/2} (\log(1/\delta_{i_{1:k}}))^{1/2}
		\ll
		\frac{\|(I_k^\leftarrow - I_k)U_{i_{1:k},\bigcdot} D\|_F^2}{\sigma_{\max}^2}
	$$
	uniformly over all cycles of size at least $k^\star$.
	By the definition of $\delta_{i_{1:k}}$, the above inequality is equivalent to
	$$
		\tilde{\xi} \cdot k^{1/2} 
		\ll
		\left(\frac{\|(I_k^\leftarrow - I_k)U_{i_{1:k},\bigcdot} D\|_F^2}{\sigma_{\max}^2}\right)^{1/2},
	$$
	which is further implied by
	$$
		\tilde{\xi} \ll \beta.
	$$
	\item We would like
	$$
		\tilde{\xi} \cdot \log(1/\delta_{i_{1:k}}) 
		\ll
		\frac{\|(I_k^\leftarrow - I_k)U_{i_{1:k},\bigcdot} D\|_F^2}{\sigma_{\max}^2}
	$$
	uniformly over all cycles of size at least $k^\star$.
	By the definition of $\delta_{i_{1:k}}$, the above inequality is equivalent to
	$$
		\tilde{\xi} \ll 1.
	$$
\end{enumerate}
Since $\beta^2 \geq 1$ and $\sigma_{\max}\geq \sigma_{\min}$, one can readily check that the above 15 conditions will be satisfied if 
\begin{align*}
	\tilde \xi 
	& \ll
	\left\{(k^\star)^{1/2} \left[ \frac{\beta}{d_1/\sigma_{\min}} \land \frac{\sigma_{\min}}{\sigma_{\max}}\left(\frac{\beta^2}{p}\right)^{1/2} \land \frac{\beta^2}{p} \land \left(\frac{\beta^2}{p}\right)^{3/4} \land \left(\frac{\beta^2}{p}\right)^{5/8} \right]\right\} \\
	& \qquad \land \left\{(k^\star)^{1/4} \left[ \frac{\sigma_{\min}}{\sigma_{\max}}\left(\frac{\beta^2}{p}\right)^{1/4} \land \left(\frac{\beta^2}{p}\right)^{1/2}  \land \left(\frac{\beta^2}{p}\right)^{3/8} \right]  \right\}
	\land \frac{\sigma_{\min}}{\sigma_{\max}} \\
	& = \left\{(k^\star)^{1/2} \left[ \frac{\beta}{d_1/\sigma_{\min}} \land \frac{\sigma_{\min}}{\sigma_{\max}}\left(\frac{\beta^2}{p}\right)^{1/2} \land \frac{\beta^2}{p} \land\left(\frac{\beta^2}{p}\right)^{5/8} \right]\right\} \\
	& \qquad \land \left\{(k^\star)^{1/4} \left[ \frac{\sigma_{\min}}{\sigma_{\max}}\left(\frac{\beta^2}{p}\right)^{1/4} \land \left(\frac{\beta^2}{p}\right)^{1/2}  \land \left(\frac{\beta^2}{p}\right)^{3/8} \right]  \right\}
	\land \frac{\sigma_{\min}}{\sigma_{\max}}.
\end{align*}
The proof is finished by recalling that $\tilde \xi = \sigma_{\min}\xi / \sigma_{\max}$.

\subsubsection{Proof of Lemma \ref{lemma:ub_uniform_and_loco_condition}} \label{prf:lemma:ub_uniform_and_loco_condition}
For notational simplicity, we denote
\begin{align*}
	\scrT_1 & := \frac{\beta}{d_1/\sigma_{\max}} \land \left(\frac{\beta^2}{p}\right)^{1/2} \land \frac{\sigma_{\max}\beta^2}{\sigma_{\min}p} \land \frac{\sigma_{\max}}{\sigma_{\min}}\left(\frac{\beta^2}{p}\right)^{5/8}, \\
	\scrT_2 & := \left(\frac{\beta^2}{p}\right)^{1/4}\land \frac{\sigma_{\max}}{\sigma_{\min}}\left(\frac{\beta^2}{p}\right)^{1/2}  \land \frac{\sigma_{\max}}{\sigma_{\min}}\left(\frac{\beta^2}{p}\right)^{3/8},
\end{align*}	
so that \eqref{eq:ub_prf_uniform_bound_assump} and \eqref{eq:ub_prf_loco_bound_assump} become
$$
	\calE_\unif \ll \left(\sqrt{k^\star} \scrT_1\right) \land \left((k^\star)^{1/4} \scrT_2\right) \land 1,
	\qquad
	\textnormal{and}
	\qquad
	k^\star \calE_\loco \ll \left(\sqrt{k^\star} \scrT_1\right) \land \left((k^\star)^{1/4} \scrT_2\right) \land 1,
$$
respectively. Rearranging terms, the above two inequalities are equivalent to $\calE_\unif \ll 1$ and
$$
	L := \frac{\calE_\unif^2}{\scrT_1^2} \lor \frac{\calE_\unif^4}{\scrT_2^4} \ll k^\star \ll \frac{\scrT_1^2}{\calE_\loco^2} \land \frac{\scrT_2^{4/3}}{\calE_\loco^{4/3}} \land \frac{1}{\calE_\loco} =: U.
$$
We now claim that if the following three conditions hold, then a valid choice of $k^\star \in[n]$ satisfying the above inequality would exist:
\begin{enumerate}
	\item $L \ll U$;
	\item $L \ll n$;
	\item $U \gg 1$.
\end{enumerate}
To show this claim, we consider the following four scenarios:
\begin{enumerate}[label=(\alph*)]
	\item If $U \geq n$ and $L\geq 1$, then we have $1 \leq L \ll n \leq U$. In this case, we can choose $k^\star = \lfloor L \cdot \sqrt{n/L} \rfloor$;
	\item If $U\geq n$ and $L < 1$, then we have $L < 1 \ll n \leq U$. In this case, we choose $k^\star = \lfloor\sqrt{n}\rfloor$. It is clear that $k^\star \in[n]$ and $L \ll k^\star\ll U$.
	\item If $U < n$ and $L\geq 1$, then we have $1 \leq L \ll U < n$. In this case, we choose $k^\star = \lfloor L \cdot \sqrt{U/L}\rfloor$. 
	\item If $U < n$ and $L < 1$, then we have $L < 1 \ll U < n$. In this case, we choose $k^\star = \lfloor \sqrt{U}\rfloor$.
\end{enumerate}
Hence the claim holds. To this end, it suffices to show that the three conditions above are implied by the assumptions imposed in this lemma.
\begin{enumerate}
	\item We first derive a sufficient condition for $L\ll U$ to hold. Note that $L$ is the maximum of two terms and $U$ is the minimum of three terms, so there are six underlying inequalities that we want them to hold. We discuss them in order below.
	\begin{enumerate}
		\item[1.1.] We want $\frac{\calE_\unif^2}{\scrT_1^2} \ll \frac{\scrT_1^2}{\calE_\loco^2}$, or equivalently
		$$
			\scrT_1^4 =  \left[\frac{\beta}{d_1/\sigma_{\max}} \land \left(\frac{\beta^2}{p}\right)^{1/2} \land \frac{\sigma_{\max}\beta^2}{\sigma_{\min}p} \land \frac{\sigma_{\max}}{\sigma_{\min}}\left(\frac{\beta^2}{p}\right)^{5/8}\right]^4\gg \calE_\unif^2\calE_\loco^2.
		$$
		Rearranging terms, we get the following equivalent condition:
		\begin{align*}
			\frac{\beta^2}{p} & \gg \left(\frac{d_1^2/\sigma_{\max}^2}{p} \cdot \calE_\unif\calE_\loco\right) \lor \left(\calE_\unif\calE_\loco\right)  \lor \left(\frac{\sigma_{\min}^{4}}{\sigma_{\max}^{4}} \cdot \calE_\unif^{2}\calE_\loco^{2}\right)^{1/4} \lor \left(\frac{\sigma_{\min}^{4}}{\sigma_{\max}^{4}} \cdot \calE_\unif^{2}\calE_\loco^{2}\right)^{2/5} \\
			& =  \left(\frac{d_1^2/\sigma_{\max}^2}{p} \cdot \calE_\unif\calE_\loco\right) \lor \left(\calE_\unif\calE_\loco\right)  \lor \left(\frac{\sigma_{\min}^{4}}{\sigma_{\max}^{4}} \cdot \calE_\unif^{2}\calE_\loco^{2}\right)^{1/4},
		\end{align*}	
		where the last inequality is by $\calE_\unif^{2}\calE_\loco^{2}\ll 1$.

		\item[1.2.] We want $\frac{\calE_\unif^2}{\scrT_1^2}\ll \frac{\scrT_2^{4/3}}{\calE_\loco^{4/3}}$, or equivalently
		\begin{align*}
			\scrT_1^2 \scrT_2^{4/3} & = 
			\left[\frac{\beta}{d_1/\sigma_{\max}} \land \left(\frac{\beta^2}{p}\right)^{1/2} \land \frac{\sigma_{\max}\beta^2}{\sigma_{\min}p} \land \frac{\sigma_{\max}}{\sigma_{\min}}\left(\frac{\beta^2}{p}\right)^{5/8}\right]^2  \\
			& \qquad \times \left[ \left(\frac{\beta^2}{p}\right)^{1/4}\land \frac{\sigma_{\max}}{\sigma_{\min}}\left(\frac{\beta^2}{p}\right)^{1/2}  \land \frac{\sigma_{\max}}{\sigma_{\min}}\left(\frac{\beta^2}{p}\right)^{3/8} \right]^{4/3} \\
			& \gg \calE_\unif^2 \calE_\loco^{4/3}.
		\end{align*}
		The left-hand side above can be written as
		\begin{align*}
			\scrT_1^2 \scrT_2^{4/3} 
			& = \left(\frac{\beta^{8/3}}{p^{4/3}} \cdot \frac{p}{d_1^2/\sigma_{\max}^2}\right)
			\land \left( \frac{\beta^{10/3}}{p^{5/3}}\cdot \frac{p}{d_1^2/\sigma_{\max}^2}\cdot \frac{\sigma_{\max}^{4/3}}{\sigma_{\min}^{4/3}} \right)
			\land \left( \frac{\beta^3}{p^{3/2}} \cdot \frac{p}{d_1^2/\sigma_{\max}^2} \cdot \frac{\sigma_{\max}^{4/3}}{\sigma_{\min}^{4/3}} \right) \\
			& \qquad 
			\land  \left(\frac{\beta^{8/3}}{p^{4/3}}\right)
			\land \left( \frac{\beta^{10/3}}{p^{5/3}} \cdot \frac{\sigma_{\max}^{4/3}}{\sigma_{\min}^{4/3}} \right)
			\land \left( \frac{\beta^3}{p^{3/2}}\cdot \frac{\sigma_{\max}^{4/3}}{\sigma_{\min}^{4/3}} \right) \\
			& \qquad
			\land \left( \frac{\beta^{14/3}}{p^{7/3}}  \cdot \frac{\sigma_{\max}^2}{\sigma_{\min}^2} \right)
			\land \left( \frac{\beta^{16/3}}{p^{8/3}} \cdot \frac{\sigma_{\max}^{10/3}}{\sigma_{\min}^{10/3}}\right)
			\land \left( \frac{\beta^5}{p^{5/2}} \cdot \frac{\sigma^{10/3}_{\max}}{\sigma_{\min}^{10/3}} \right) \\
			& \qquad
			\land \left( \frac{\beta^{19/6}}{p^{19/12}} \cdot \frac{\sigma_{\max}^2}{\sigma_{\min}^2} \right)
			\land \left( \frac{\beta^{23/6}}{p^{23/12}} \cdot \frac{\sigma_{\max}^{10/3}}{\sigma_{\min}^{10/3}} \right)
			\land \left( \frac{\beta^{7/2}}{p^{7/4}} \cdot \frac{\sigma_{\max}^{10/3}}{\sigma_{\min}^{10/3}} \right).
		\end{align*}
		Thus, $\scrT_1^2 \scrT_2^{4/3} \gg \calE_\unif^2 \calE_\loco^{4/3}$ is equivalent to
		\begin{align*}
			\frac{\beta^2}{p}
			& \gg
			\left( \frac{d_1^2/\sigma_{\max}^2}{p} \cdot \calE_\unif^2 \calE_\loco^{4/3} \right)^{3/4}
			\lor \left(\frac{d_1^2/\sigma_{\max}^2}{p} \cdot \frac{\sigma_{\min}^{4/3}}{\sigma_{\max}^{4/3}} \cdot \calE_\unif^2 \calE_\loco^{4/3}\right)^{3/5}
			\lor \left(\frac{d_1^2/\sigma_{\max}^2}{p}\cdot \frac{\sigma_{\min}^{4/3}}{\sigma_{\max}^{4/3}} \cdot \calE_\unif^2 \calE_\loco^{4/3}\right)^{2/3} \\
			& \qquad 
			\lor \left(\calE_\unif^2 \calE_\loco^{4/3}\right)^{3/4}
			\lor \left(\frac{\sigma_{\min}^{4/3}}{\sigma_{\max}^{4/3}} \cdot \calE_\unif^2 \calE_\loco^{4/3}\right)^{3/5}
			\lor \left( \frac{\sigma_{\min}^{4/3}}{\sigma_{\max}^{4/3}} \cdot \calE_\unif^2 \calE_\loco^{4/3}\right)^{2/3} \\
			& \qquad 
			\lor \left(\frac{\sigma_{\min}^{2}}{\sigma_{\min}^2}\cdot \calE_\unif^2 \calE_\loco^{4/3}\right)^{3/7}
			\lor \left( \frac{\sigma_{\min}^{10/3}}{\sigma_{\max}^{10/3}} \cdot \calE_\unif^2 \calE_\loco^{4/3} \right)^{3/8}
			\lor \left( \frac{\sigma_{\min}^{10/3}}{\sigma_{\max}^{10/3}} \cdot \calE_\unif^2 \calE_\loco^{4/3} \right)^{2/5} \\
			& \qquad
			\lor \left(\frac{\sigma_{\min}^2}{\sigma_{\max}^2} \cdot \calE_\unif^2 \calE_\loco^{4/3} \right)^{12/19}
			\lor \left(\frac{\sigma_{\min}^{10/3}}{\sigma_{\max}^{10/3}} \cdot \calE_\unif^2 \calE_\loco^{4/3}\right)^{12/23}
			\lor \left(\frac{\sigma_{\min}^{10/3}}{\sigma_{\max}^{10/3}}\cdot \calE_\unif^2 \calE_\loco^{4/3}\right)^{4/7} \\
			& = 
			\left( \frac{d_1^2/\sigma_{\max}^2}{p} \cdot \calE_\unif^2 \calE_\loco^{4/3} \right)^{3/4}
			\lor \left(\frac{d_1^2/\sigma_{\max}^2}{p} \cdot \frac{\sigma_{\min}^{4/3}}{\sigma_{\max}^{4/3}} \cdot \calE_\unif^2 \calE_\loco^{4/3}\right)^{3/5}
			\lor \left(\frac{d_1^2/\sigma_{\max}^2}{p}\cdot \frac{\sigma_{\min}^{4/3}}{\sigma_{\max}^{4/3}} \cdot \calE_\unif^2 \calE_\loco^{4/3}\right)^{2/3} \\
			& \qquad 
			\lor \left(\calE_\unif^2 \calE_\loco^{4/3}\right)^{3/4}
			\lor \left(\frac{\sigma_{\min}^{4/3}}{\sigma_{\max}^{4/3}} \cdot \calE_\unif^2 \calE_\loco^{4/3}\right)^{3/5}
			\lor \left( \frac{\sigma_{\min}^{2}}{\sigma_{\max}^{2}} \cdot \calE_\unif^2 \calE_\loco^{4/3}\right)^{3/7}
			\lor \left(\frac{\sigma_{\min}^{10/3}}{\sigma_{\max}^{10/3}}\cdot \calE_\unif^2 \calE_\loco^{4/3}\right)^{3/8}
		\end{align*}
		where the equality is by $\calE_\unif \lor \calE_\loco \ll 1$.

		\item[1.3.] We want $\frac{\calE_\unif^2}{\scrT_1^2} \ll \frac{1}{\calE_\loco}$, or equivalently
		$$
			\scrT_1^2 = \left[\frac{\beta}{d_1/\sigma_{\max}} \land \left(\frac{\beta^2}{p}\right)^{1/2} \land \frac{\sigma_{\max}\beta^2}{\sigma_{\min}p} \land \frac{\sigma_{\max}}{\sigma_{\min}}\left(\frac{\beta^2}{p}\right)^{5/8}\right]^2
			\gg \calE_\unif^2 \calE_\loco.
		$$
		Rearranging terms, the above inequality becomes
		\begin{align*}
			\frac{\beta^2}{p}
			& \gg \left(\frac{d_1^2/\sigma_{\max}^2}{p} \cdot \calE_\unif^2\calE_\loco\right) \lor \left(\calE_\unif^2\calE_\loco \right) \lor \left( \frac{\sigma_{\min}^2}{\sigma_{\max}^2} \cdot \calE_\unif^2\calE_\loco \right)^{1/2} \lor \left(\frac{\sigma_{\min}^2}{\sigma_{\max}^2} \cdot \calE_\unif^2\calE_\loco\right)^{4/5} \\
			& = \left(\frac{d_1^2/\sigma_{\max}^2}{p} \cdot \calE_\unif^2\calE_\loco\right) \lor \left(\calE_\unif^2\calE_\loco \right) \lor \left( \frac{\sigma_{\min}^2}{\sigma_{\max}^2} \cdot \calE_\unif^2\calE_\loco \right)^{1/2},
		\end{align*}
		where the last equality is by $\calE_\unif^2 \calE_\loco \ll 1$.

		\item[1.4.] We want $\frac{\calE_\unif^4}{\scrT_2^4} \ll \frac{\scrT_1^2}{\calE_\loco^2}$, which is equivalent to
		\begin{align*}
			\scrT_1^2 \scrT_2^4 
			& = 
			 \left[\frac{\beta}{d_1/\sigma_{\max}} \land \left(\frac{\beta^2}{p}\right)^{1/2} \land \frac{\sigma_{\max}\beta^2}{\sigma_{\min}p} \land \frac{\sigma_{\max}}{\sigma_{\min}}\left(\frac{\beta^2}{p}\right)^{5/8}\right]^2 \\
			 & \qquad \times \left[\left(\frac{\beta^2}{p}\right)^{1/4}\land \frac{\sigma_{\max}}{\sigma_{\min}}\left(\frac{\beta^2}{p}\right)^{1/2}  \land \frac{\sigma_{\max}}{\sigma_{\min}}\left(\frac{\beta^2}{p}\right)^{3/8}\right]^4\\
			& \gg \calE_\unif^4 \calE_\loco^2.
		\end{align*}
		The left-hand side above can be written as
		\begin{align*}
			\scrT_1^2 \scrT_2^4
			& = 
			\left( \frac{\beta^4}{p^2} \cdot \frac{p}{d_1^2/\sigma_{\max}^2} \right)
			\land \left( \frac{\beta^6}{p^3}\cdot \frac{p}{d_1^2/\sigma_{\max}^2}\cdot \frac{\sigma_{\max}^4}{\sigma_{\min}^4} \right)
			\land \left(\frac{\beta^5}{p^{5/2}}\cdot \frac{p}{d_1^2/\sigma_{\max}^2}\cdot \frac{\sigma_{\max}^4}{\sigma_{\min}^4}\right) \\
			& \qquad
			\land \left(\frac{\beta^4}{p^2}\right)
			\land \left(\frac{\beta^6}{p^3}\cdot \frac{\sigma_{\max}^4}{\sigma_{\min}^4}\right)
			\land \left(\frac{\beta^5}{p^{5/2}}\cdot \frac{\sigma_{\max}^4}{\sigma_{\min}^4}\right) \\
			& \qquad
			\land \left( \frac{\beta^6}{p^3} \cdot \frac{\sigma_{\max}^2}{\sigma_{\min}^2} \right)
			\land \left( \frac{\beta^8}{p^4}\cdot \frac{\sigma_{\max}^6}{\sigma_{\min}^6} \right)
			\land \left( \frac{\beta^7}{p^{7/2}}\cdot \frac{\sigma_{\max}^6}{\sigma_{\min}^6} \right) \\
			& \qquad
			\land \left( \frac{\beta^{9/2}}{p^{9/4}}\cdot \frac{\sigma_{\max}^2}{\sigma_{\min}^2} \right)
			\land \left( \frac{\beta^{13/2}}{p^{13/4}}\cdot \frac{\sigma_{\max}^6}{\sigma_{\min}^6} \right)
			\land \left( \frac{\beta^{11/2}}{p^{11/4}} \cdot \frac{\sigma_{\max}^6}{\sigma_{\min}^6} \right).
		\end{align*}
		Thus, the desired inequality is equivalent to
		\begin{align*}
			\frac{\beta^2}{p}
			& \gg
			\left(\frac{d_1^2/\sigma_{\max}^2}{p}\cdot \calE_\unif^4\calE_\loco^2\right)^{1/2}
			\lor \left(\frac{d_1^2/\sigma_{\max}^2}{p} \cdot \frac{\sigma_{\min}^4}{\sigma_{\max}^4} \cdot \calE_\unif^4 \calE_\loco^2\right)^{1/3}
			\lor \left(\frac{d_1^2/\sigma_{\max}^2}{p} \cdot \frac{\sigma_{\min}^4}{\sigma_{\max}^4} \cdot \calE_\unif^4\calE_\loco^2\right)^{2/5} \\
			& \qquad
			\lor \left(\calE_\unif^4 \calE_\loco^2\right)^{1/2}
			\lor \left(\frac{\sigma_{\min}^4}{\sigma_{\max}^4} \cdot \calE_\unif^4 \calE_\loco^2\right)^{1/3}
			\lor \left(\frac{\sigma_{\min}^4}{\sigma_{\max}^4} \calE_\unif^4 \calE_\loco^2\right)^{2/5} \\
			& \qquad
			\lor \left(\frac{\sigma_{\min}^2}{\sigma_{\max}^2} \cdot \calE_\unif^4\calE_\loco^2\right)^{1/3} 
			\lor \left(\frac{\sigma_{\min}^6}{\sigma_{\max}^6} \cdot \calE_\unif^4 \calE_\loco^2\right)^{1/4} 
			\lor \left(\frac{\sigma_{\min}^6}{\sigma_{\max}^6} \cdot \calE_\unif^4 \calE_\loco^2\right)^{2/7} \\
			& \qquad 
			\lor \left(\frac{\sigma_{\min}^2}{\sigma_{\max}^2} \cdot \calE_\unif^4 \calE_\loco^2\right)^{4/9} 
			\lor \left(\frac{\sigma_{\min}^6}{\sigma_{\max}^6} \cdot \calE_\unif^4 \calE_\loco^2 \right)^{4/13}
			\lor \left(\frac{\sigma_{\min}^6}{\sigma_{\max}^6} \cdot \calE_\unif^4 \calE_\loco^2 \right)^{4/11} \\
			& = \left(\frac{d_1^2/\sigma_{\max}^2}{p}\cdot \calE_\unif^4\calE_\loco^2\right)^{1/2}
			\lor \left(\frac{d_1^2/\sigma_{\max}^2}{p} \cdot \frac{\sigma_{\min}^4}{\sigma_{\max}^4} \cdot \calE_\unif^4 \calE_\loco^2\right)^{1/3}
			\lor \left(\frac{d_1^2/\sigma_{\max}^2}{p} \cdot \frac{\sigma_{\min}^4}{\sigma_{\max}^4} \cdot \calE_\unif^4\calE_\loco^2\right)^{2/5} \\
			& \qquad 
			\lor \left(\calE_\unif^4\calE_\loco^2\right)^2
			\lor \left(\frac{\sigma_{\min}^2}{\sigma_{\max}^2}\cdot \calE_\unif^4 \calE_\loco^2\right)^{1/3}
			\lor \left(\frac{\sigma_{\min}^4}{\sigma_{\max}^4}\cdot \calE_\unif^4 \calE_\loco^2\right)^{1/3}
			\lor \left(\frac{\sigma_{\min}^6}{\sigma_{\max}^6}\cdot \calE_\unif^4 \calE_\loco^2\right)^{1/4} \\
			& = \left(\frac{d_1^2/\sigma_{\max}^2}{p}\cdot \calE_\unif^4\calE_\loco^2\right)^{1/2}
			\lor \left(\frac{d_1^2/\sigma_{\max}^2}{p} \cdot \frac{\sigma_{\min}^4}{\sigma_{\max}^4} \cdot \calE_\unif^4 \calE_\loco^2\right)^{1/3}
			\lor \left(\frac{d_1^2/\sigma_{\max}^2}{p} \cdot \frac{\sigma_{\min}^4}{\sigma_{\max}^4} \cdot \calE_\unif^4\calE_\loco^2\right)^{2/5} \\
			& \qquad 
			\lor \left(\calE_\unif^4\calE_\loco^2\right)^2
			\lor \left(\frac{\sigma_{\min}^2}{\sigma_{\max}^2}\cdot \calE_\unif^4 \calE_\loco^2\right)^{1/3}
			\lor \left(\frac{\sigma_{\min}^6}{\sigma_{\max}^6}\cdot \calE_\unif^4 \calE_\loco^2\right)^{1/4}
		\end{align*}
		where the last two equalities are by $\calE_\unif^4\calE_\loco^2 \ll 1$.

		\item[1.5.] We want $\frac{\calE_\unif^4}{\scrT_2^4} \ll \frac{\scrT_2^{4/3}}{\calE_\loco^{4/3}}$, or equivalently
		\begin{align*}
			\scrT_2^{16/3} = 
			\left[\left(\frac{\beta^2}{p}\right)^{1/4}\land \frac{\sigma_{\max}}{\sigma_{\min}}\left(\frac{\beta^2}{p}\right)^{1/2}  \land \frac{\sigma_{\max}}{\sigma_{\min}}\left(\frac{\beta^2}{p}\right)^{3/8}\right]^{16/3}
			\gg \calE_\unif^{4} \calE_\loco^{4/3}.
		\end{align*}
		Rearranging terms, the above inequality becomes
		\begin{align*}
			\frac{\beta^2}{p}
			& \gg
			\left( \calE_\unif^{4} \calE_\loco^{4/3} \right)^{3/4} 
			\lor \left(\frac{\sigma_{\min}^{16/3}}{\sigma_{\max}^{16/3}} \cdot \calE_\unif^{4} \calE_\loco^{4/3}\right)^{3/8}
			\lor \left(\frac{\sigma_{\min}^{16/3}}{\sigma_{\max}^{16/3}} \cdot \calE_\unif^{4} \calE_\loco^{4/3}\right)^{1/2}\\
			& = 
			\left( \calE_\unif^{4} \calE_\loco^{4/3} \right)^{3/4} 
			\lor \left(\frac{\sigma_{\min}^{16/3}}{\sigma_{\max}^{16/3}} \cdot \calE_\unif^{4} \calE_\loco^{4/3}\right)^{3/8},
		\end{align*}
		where the last equality is by $\calE_\unif^4\calE_\loco^{4/3} \ll 1$.
		\item[1.6.] We want $\frac{\calE_\unif^4}{\scrT_2^4} \ll \frac{1}{\calE_\loco}$, which is equivalent to
		\begin{align*}
			\scrT^4_2 = \left[\left(\frac{\beta^2}{p}\right)^{1/4}\land \frac{\sigma_{\max}}{\sigma_{\min}}\left(\frac{\beta^2}{p}\right)^{1/2}  \land \frac{\sigma_{\max}}{\sigma_{\min}}\left(\frac{\beta^2}{p}\right)^{3/8}\right]^{4} \gg \calE_\unif^4 \calE_\loco.
		\end{align*}
		Rearranging terms, the above inequality becomes
		\begin{align*}
			\frac{\beta^2}{p} \gg
			& \left(\calE_\unif^4 \calE_\loco\right) 
			\lor \left(\frac{\sigma_{\min}^4}{\sigma_{\max}^4} \cdot \calE_\unif^4 \calE_\loco\right)^{1/2}
			\lor \left(\frac{\sigma_{\min}^4}{\sigma_{\max}^4} \cdot \calE_\unif^4 \calE_\loco\right)^{2/3} \\
			& = 
			\left(\calE_\unif^4 \calE_\loco\right) 
			\lor \left(\frac{\sigma_{\min}^4}{\sigma_{\max}^4} \cdot \calE_\unif^4 \calE_\loco\right)^{1/2},
		\end{align*}
		where the last equality is by $\calE_\unif^4 \calE_\loco \ll 1$.
	\end{enumerate}

	\item The second condition is $L \ll n$, i.e., both
	$$
		\frac{\calE_\unif^2}{\scrT_1^2} \ll n \iff \scrT_1^2 = \left[\frac{\beta}{d_1/\sigma_{\max}} \land \left(\frac{\beta^2}{p}\right)^{1/2} \land \frac{\sigma_{\max}\beta^2}{\sigma_{\min}p} \land \frac{\sigma_{\max}}{\sigma_{\min}}\left(\frac{\beta^2}{p}\right)^{5/8}\right]^2 \gg \frac{\calE_\unif^2}{n},
	$$
	and
	$$
		\frac{\calE_\unif^4}{\scrT_2^4} \ll n \iff
		\scrT_2^4 =  \left[\left(\frac{\beta^2}{p}\right)^{1/4}\land \frac{\sigma_{\max}}{\sigma_{\min}}\left(\frac{\beta^2}{p}\right)^{1/2}  \land \frac{\sigma_{\max}}{\sigma_{\min}}\left(\frac{\beta^2}{p}\right)^{3/8}\right]^{4} \gg  \frac{\calE_\unif^4}{n}.
	$$
	We can equivalently express the above inequality as
	\begin{align*}
		\frac{\beta^2}{p}
		& \gg \left(\frac{d_1^2/\sigma_{\max}^2}{p} \cdot \frac{\calE_\unif^2}{n}\right) 
		\lor \left( \frac{\calE_\unif^2}{n}\right)
		\lor \left(\frac{\sigma_{\min}^2}{\sigma_{\max}^2} \cdot \frac{\calE_\unif^2}{n}\right)^{1/2} 
		\lor \left( \frac{\sigma_{\min}^2}{\sigma_{\max}^2} \cdot \frac{\calE_\unif^2}{n} \right)^{4/5} \\
		& \qquad \lor \left(\frac{\calE_\unif^4}{n}\right) 
		\lor \left(\frac{\sigma_{\min}^4}{\sigma_{\max}^4} \cdot \frac{\calE_\unif^4}{n}\right)^{1/2}
		\lor \left(\frac{\sigma_{\min}^4}{\sigma_{\max}^4} \cdot \frac{\calE_\unif^4}{n}\right)^{2/3} \\
		& = \left(\frac{d_1^2/\sigma_{\max}^2}{p} \cdot \frac{\calE_\unif^2}{n}\right) 
		\lor \left( \frac{\calE_\unif^2}{n}\right)
		\lor \left(\frac{\sigma_{\min}^2}{\sigma_{\max}^2} \cdot \frac{\calE_\unif^2}{n}\right)^{1/2},
	\end{align*}
	where the last equality is by $\calE_\unif \ll 1$.
	\item The third condition is $U\gg 1$, i.e., 
	$$
		\frac{\scrT_1^2}{\calE_\loco^2} \land \frac{\scrT_2^{4/3}}{\calE_\loco^{4/3}} \land \frac{1}{\calE_\loco} \gg 1 ,
	$$
	which holds if and only if
	$$
		\scrT_1 \gg \calE_\loco, \scrT_2 \gg \calE_\loco, \calE_\loco \ll 1.
	$$
	The above inequality holds when $\calE_\loco \ll 1$	and
	\begin{align*}
		\frac{\beta^2}{p}
		& \gg 
		\left(\frac{d_1/\sigma_{\max}}{\sqrt{p}}\cdot \calE_\loco\right)^2
		\lor \left(\calE_\loco\right)^2
		\lor \left(\frac{\sigma_{\min}}{\sigma_{\max}} \cdot \calE_\loco\right)
		\lor \left(\frac{\sigma_{\min}}{\sigma_{\max}} \cdot \calE_\loco\right)^{8/5} \\
		& \qquad
		\lor \left(\calE_\loco\right)^4
		\lor \left( \frac{\sigma_{\min}}{\sigma_{\max}} \cdot \calE_\loco \right)^2
		\lor \left( \frac{\sigma_{\min}}{\sigma_{\max}} \cdot \calE_\loco \right)^{8/3} \\
		& = 
		\left(\frac{d_1/\sigma_{\max}}{\sqrt{p}}\cdot \calE_\loco\right)^2
		\lor \left(\calE_\loco\right)^2
		\lor \left( \frac{\sigma_{\min}}{\sigma_{\max}} \cdot \calE_\loco \right),
	\end{align*}
	where the last equality is by $\calE_\loco\ll 1$.
\end{enumerate}
Summarize the above discussion, a sufficient condition for $L\ll U, L \ll n, U\gg 1$ is given by $\calE_\loco \lor \calE_\unif \ll 1$ and
\begin{align*}
	\frac{\beta^2}{p}
	& \gg
	\left(\frac{d_1^2/\sigma_{\max}^2}{p} \cdot \calE_\unif\calE_\loco\right) \lor \left(\calE_\unif\calE_\loco\right)  \lor \left(\frac{\sigma_{\min}^{4}}{\sigma_{\max}^{4}} \cdot \calE_\unif^{2}\calE_\loco^{2}\right)^{1/4} \\
	&\qquad
	\lor \left( \frac{d_1^2/\sigma_{\max}^2}{p} \cdot \calE_\unif^2 \calE_\loco^{4/3} \right)^{3/4}
	\lor \left(\frac{d_1^2/\sigma_{\max}^2}{p} \cdot \frac{\sigma_{\min}^{4/3}}{\sigma_{\max}^{4/3}} \cdot \calE_\unif^2 \calE_\loco^{4/3}\right)^{3/5}
	\lor \left(\frac{d_1^2/\sigma_{\max}^2}{p}\cdot \frac{\sigma_{\min}^{4/3}}{\sigma_{\max}^{4/3}} \cdot \calE_\unif^2 \calE_\loco^{4/3}\right)^{2/3} \\
	& \qquad 
	\lor \left(\calE_\unif^2 \calE_\loco^{4/3}\right)^{3/4}
	\lor \left(\frac{\sigma_{\min}^{4/3}}{\sigma_{\max}^{4/3}} \cdot \calE_\unif^2 \calE_\loco^{4/3}\right)^{3/5}
	\lor \left( \frac{\sigma_{\min}^{2}}{\sigma_{\max}^{2}} \cdot \calE_\unif^2 \calE_\loco^{4/3}\right)^{3/7}
	\lor \left(\frac{\sigma_{\min}^{10/3}}{\sigma_{\max}^{10/3}}\cdot \calE_\unif^2 \calE_\loco^{4/3}\right)^{3/8} \\
	& \qquad
	\lor \left(\frac{d_1^2/\sigma_{\max}^2}{p} \cdot \calE_\unif^2\calE_\loco\right) \lor \left(\calE_\unif^2\calE_\loco \right) \lor \left( \frac{\sigma_{\min}^2}{\sigma_{\max}^2} \cdot \calE_\unif^2\calE_\loco \right)^{1/2} \\
	& \qquad
	\lor \left(\frac{d_1^2/\sigma_{\max}^2}{p}\cdot \calE_\unif^4\calE_\loco^2\right)^{1/2}
	\lor \left(\frac{d_1^2/\sigma_{\max}^2}{p} \cdot \frac{\sigma_{\min}^4}{\sigma_{\max}^4} \cdot \calE_\unif^4 \calE_\loco^2\right)^{1/3}
	\lor \left(\frac{d_1^2/\sigma_{\max}^2}{p} \cdot \frac{\sigma_{\min}^4}{\sigma_{\max}^4} \cdot \calE_\unif^4\calE_\loco^2\right)^{2/5} \\
	& \qquad 
	\lor \left(\calE_\unif^4\calE_\loco^2\right)^2
	\lor \left(\frac{\sigma_{\min}^2}{\sigma_{\max}^2}\cdot \calE_\unif^4 \calE_\loco^2\right)^{1/3}
	\lor \left(\frac{\sigma_{\min}^6}{\sigma_{\max}^6}\cdot \calE_\unif^4 \calE_\loco^2\right)^{1/4} \\
	& \qquad
	\lor \left( \calE_\unif^{4} \calE_\loco^{4/3} \right)^{3/4} 
	\lor \left(\frac{\sigma_{\min}^{16/3}}{\sigma_{\max}^{16/3}} \cdot \calE_\unif^{4} \calE_\loco^{4/3}\right)^{3/8} \\
	& \qquad
	\lor \left(\calE_\unif^4 \calE_\loco\right) 
	\lor \left(\frac{\sigma_{\min}^4}{\sigma_{\max}^4} \cdot \calE_\unif^4 \calE_\loco\right)^{1/2} \\
	& \qquad
	\lor \left(\frac{d_1^2/\sigma_{\max}^2}{p} \cdot \frac{\calE_\unif^2}{n}\right) 
	\lor \left( \frac{\calE_\unif^2}{n}\right)
	\lor \left(\frac{\sigma_{\min}^2}{\sigma_{\max}^2} \cdot \frac{\calE_\unif^2}{n}\right)^{1/2} \\
	& \qquad
	\lor \left(\frac{d_1/\sigma_{\max}}{\sqrt{p}}\cdot \calE_\loco\right)^2
	\lor \left(\calE_\loco\right)^2
	\lor \left( \frac{\sigma_{\min}}{\sigma_{\max}} \cdot \calE_\loco \right) \\
	& = 
	\left( \frac{d_1^2/\sigma_{\max}^2}{p} \cdot \calE_\unif \calE_\loco\right)
	\lor \left( \frac{d_1^2/\sigma_{\max}^2}{p} \cdot \calE_\unif^2 \calE_\loco\right)
	\lor \left( \frac{d_1^2/\sigma_{\max}^2}{p} \cdot \frac{\calE_\unif^2}{n}\right) 
	\lor \left( \frac{d_1^2/\sigma_{\max}^2}{p} \cdot \calE_\loco^2\right)\\
	& \qquad
	\lor \left( \frac{d_1^2/\sigma_{\max}^2}{p} \cdot \calE_\unif^2 \calE_\loco^{4/3}\right)^{3/4}
	\lor \left( \frac{d_1^2/\sigma_{\max}}{p} \cdot \calE_\unif^2 \calE_\loco^{4/3} \cdot \frac{\sigma_{\min}^{4/3}}{\sigma_{\max}^{4/3}}\right)^{2/3}
	\lor \left( \frac{d_1^2\sigma_{\max}}{p} \cdot \calE_\unif^2 \calE_\loco^{4/3} \cdot \frac{\sigma_{\min}^{4/3}}{\sigma_{\max}^{4/3}}\right)^{3/5}\\
	& \qquad
	\lor \left( \frac{d_1^2\sigma_{\max}}{p} \cdot \calE_\unif^4 \calE_\loco^2\right)^{1/2}
	\lor \left( \frac{d_1^2\sigma_{\max}}{p} \cdot \calE_\unif^4 \calE_\loco^2 \cdot \frac{\sigma_{\min}^4}{\sigma_{\max}^4}\right)^{2/5}
	\lor \left( \frac{d_1^2}{p} \cdot \calE_\unif^4 \calE_\loco^2 \cdot \frac{\sigma_{\min}^4}{\sigma_{\max}^4}\right)^{1/3}\\
	& \qquad
	\lor \left(\calE_\unif^8 \calE_\loco^4\right)
	\lor \left(\calE_\unif^4 \calE_\loco\right)
	\lor \left(\calE_\unif^3 \calE_\loco\right)
	\lor \left(\calE_\unif^2 \calE_\loco\right)
	\lor \left(\calE_\unif^2 \calE_\loco^{1/2} \cdot \frac{\sigma_{\min}^2}{\sigma_{\max}^2}\right)\\
	& \qquad
	\lor \left(\frac{\calE_\unif^2}{n}\right)
	\lor \left(\calE_\unif^{3/2} \calE_\loco\right)
	\lor \left(\calE_\unif^{3/2} \calE_\loco^{1/2} \cdot \frac{\sigma_{\min}^2}{\sigma_{\max}^2}\right)
	\lor \left(\calE_\unif^{4/3} \calE_\loco^{2/3} \cdot \frac{\sigma_{\min}^{2/3}}{\sigma_{\max}^{2/3}}\right) \\
	&\qquad
	\lor \left(\calE_\unif^{6/5} \calE_\loco^{4/5} \cdot \frac{\sigma_{\min}^{4/5}}{\sigma_{\max}^{4/5}}\right)
	\lor \left(\calE_\unif \calE_\loco\right)
	\lor \left(\calE_\unif \calE_\loco^{1/2} \cdot \frac{\sigma_{\min}}{\sigma_{\max}}\right)
	\lor \left(\calE_\unif \calE_\loco^{1/2} \cdot \frac{\sigma_{\min}^{3/2}}{\sigma_{\max}^{3/2}}\right) \\
	& \qquad
	\lor \left(\frac{\calE_\unif^2}{\sqrt{n}} \cdot \frac{\sigma_{\min}}{\sigma_{\max}}\right)
	\lor \left(\calE_\unif^{6/7} \calE_\loco^{4/7} \cdot \frac{\sigma_{\min}^{6/7}}{\sigma_{\max}^{6/7}}\right)
	\lor \left(\calE_\unif^{3/4} \calE_\loco^{1/2} \cdot \frac{\sigma_{\min}^{5/4}}{\sigma_{\max}^{5/4}}\right)
	\lor \left(\calE_\unif^{1/2} \calE_\loco^{1/2} \cdot \frac{\sigma_{\min}}{\sigma_{\max}}\right) \\
	& \qquad
	\lor \left(\calE_\loco^2\right)
	\lor \left(\calE_\loco \cdot \frac{\sigma_{\min}}{\sigma_{\max}}\right) \\
	& = 
	\left( \frac{d_1^2/\sigma_{\max}^2}{p} \cdot \calE_\unif \calE_\loco\right)
	\lor \left( \frac{d_1^2/\sigma_{\max}^2}{p} \cdot \frac{\calE_\unif^2}{n}\right) 
	\lor \left( \frac{d_1^2/\sigma_{\max}^2}{p} \cdot \calE_\loco^2\right)\\
	& \qquad
	\lor \left( \frac{d_1^2/\sigma_{\max}^2}{p} \cdot \calE_\unif^2 \calE_\loco^{4/3}\right)^{3/4}
	\lor \left( \frac{d_1^2/\sigma_{\max}}{p} \cdot \calE_\unif^2 \calE_\loco^{4/3} \cdot \frac{\sigma_{\min}^{4/3}}{\sigma_{\max}^{4/3}}\right)^{2/3}
	\lor \left( \frac{d_1^2/\sigma_{\max}}{p} \cdot \calE_\unif^2 \calE_\loco^{4/3} \cdot \frac{\sigma_{\min}^{4/3}}{\sigma_{\max}^{4/3}}\right)^{3/5}\\
	& \qquad
	\lor \left( \frac{d_1^2/\sigma_{\max}}{p} \cdot \calE_\unif^4 \calE_\loco^2\right)^{1/2}
	\lor \left( \frac{d_1^2/\sigma_{\max}}{p} \cdot \calE_\unif^4 \calE_\loco^2 \cdot \frac{\sigma_{\min}^4}{\sigma_{\max}^4}\right)^{2/5}
	\lor \left( \frac{d_1^2/\sigma_{\max}}{p} \cdot \calE_\unif^4 \calE_\loco^2 \cdot \frac{\sigma_{\min}^4}{\sigma_{\max}^4}\right)^{1/3}\\
	& \qquad
	\lor \left(\frac{\calE_\unif^2}{n}\right)
	\lor \left(\calE_\unif^{4/3} \calE_\loco^{2/3} \cdot \frac{\sigma_{\min}^{2/3}}{\sigma_{\max}^{2/3}}\right) \\
	&\qquad
	\lor \left(\calE_\unif^{6/5} \calE_\loco^{4/5} \cdot \frac{\sigma_{\min}^{4/5}}{\sigma_{\max}^{4/5}}\right)
	\lor \left(\calE_\unif \calE_\loco\right) \\
	& \qquad
	\lor \left(\frac{\calE_\unif}{\sqrt{n}} \cdot \frac{\sigma_{\min}}{\sigma_{\max}}\right)
	\lor \left(\calE_\unif^{6/7} \calE_\loco^{4/7} \cdot \frac{\sigma_{\min}^{6/7}}{\sigma_{\max}^{6/7}}\right)
	\lor \left(\calE_\unif^{1/2} \calE_\loco^{1/2} \cdot \frac{\sigma_{\min}}{\sigma_{\max}}\right) \\
	& \qquad
	\lor \left(\calE_\loco^2\right)
	\lor \left(\calE_\loco \cdot \frac{\sigma_{\min}}{\sigma_{\max}}\right) \\
	& =
	\left( \omega_n \cdot \frac{\calE_\loco}{\calE_\unif}\right)
	\lor \left( \frac{\omega_n}{n}\right) 
	\lor \left( \omega_n \cdot \frac{\calE_\loco^2}{\calE_\unif^2}\right)\\
	& \qquad
	\lor \left( \omega_n^{3/4} \cdot \calE_\loco \right)
	\lor \left( \omega_n^{2/3} \cdot \calE_\loco^{8/9} \cdot \frac{\sigma_{\min}^{8/9}}{\sigma_{\max}^{8/9}}\right)
	\lor \left( \omega_n^{3/5} \cdot \calE_\loco^{4/5} \cdot \frac{\sigma_{\min}^{4/5}}{\sigma_{\max}^{4/5}}\right)\\
	& \qquad
	\lor \left( \omega_n^{1/2} \cdot \calE_\unif\calE_\loco\right)
	\lor \left( \omega_n^{2/5} \cdot \calE_\unif^{4/5}\calE_\loco^{4/5} \cdot \frac{\sigma_{\min}^{8/5}}{\sigma_{\max}^{8/5}}\right)
	\lor \left( \omega_n^{1/3} \cdot \calE_\loco^{2/3}\cdot \frac{\sigma_{\min}^{4/3}}{\sigma_{\max}^{4/3}}\right)\\
	& \qquad
	\lor \left(\frac{\calE_\unif^2}{n}\right)
	\lor \left(\calE_\unif^{4/3} \calE_\loco^{2/3} \cdot \frac{\sigma_{\min}^{2/3}}{\sigma_{\max}^{2/3}}\right) \\
	&\qquad
	\lor \left(\calE_\unif^{6/5} \calE_\loco^{4/5} \cdot \frac{\sigma_{\min}^{4/5}}{\sigma_{\max}^{4/5}}\right)
	\lor \left(\calE_\unif \calE_\loco\right) \\
	& \qquad
	\lor \left(\frac{\calE_\unif}{\sqrt{n}} \cdot \frac{\sigma_{\min}}{\sigma_{\max}}\right)
	\lor \left(\calE_\unif^{6/7} \calE_\loco^{4/7} \cdot \frac{\sigma_{\min}^{6/7}}{\sigma_{\max}^{6/7}}\right)
	\lor \left(\calE_\unif^{1/2} \calE_\loco^{1/2} \cdot \frac{\sigma_{\min}}{\sigma_{\max}}\right) \\
	& \qquad
	\lor \left(\calE_\loco^2\right)
	\lor \left(\calE_\loco \cdot \frac{\sigma_{\min}}{\sigma_{\max}}\right)
\end{align*}
where we rearranged terms in the first equality, cancelled redundant terms in the second equality (by $\calE_\loco \lor \calE_\unif \ll 1$ and $\sigma_{\min}\leq \sigma_{\max}$), and plugged in $\omega_n = \frac{d_1^2/\sigma_{\max}^2}{p} \cdot \calE_\unif^2$ in the last equality. The proof is concluded.

\subsection{Proof of Corollary \ref{cor:ub_mismatch_proportion_under_further_assumptions}} \label{prf:cor:ub_mismatch_proportion_under_further_assumptions}

Define 
$
\Delta_{i,i'} := \|(U_{i, \bigcdot} - U_{i', \bigcdot})D\|^2/\sigma_{\max}^2.
$
With a slight abuse of notation we let
$$
	\calE_k := \frac{1}{n}\sum_{i\neq i'}  \exp\left\{ \frac{-(1-o(1))C_k \|(U_{i, \bigcdot} - U_{i', \bigcdot})D\|^2}{\sigma_{\max}^2} \right\},
$$
where the exact value of the $o(1)$ term may change from line by line.
We start by computing
\begin{align*}
	& \sum_{i_1\neq \cdots\neq i_k} \exp\left\{\frac{-(1-o(1))C_k\left\|(I_k^\leftarrow -I_k)U_{i_{1:k},\bigcdot} D\right\|_F^2}{\sigma_{\max}^2}\right\}\\
	& = \sum_{i_1\neq \cdots \neq i_k} \exp\left\{-(1-o(1))C_k \Delta_{i_{k},i_1}\right\} \cdot \exp\left\{-(1-o(1)) C_k \left(\Delta_{i_1, i_2} + \cdots \Delta_{i_{k-2},i_k}\right)\right\} \\
	& \leq \sum_{i_1\neq \cdots \neq i_k} \max_{i'\neq i_1}\left(\exp\left\{-(1-o(1))C_k \Delta_{i_1,i'}\right\} \right)\cdot \exp\left\{-(1-o(1)) C_k \left(\Delta_{i_1, i_2} + \cdots \Delta_{i_{k-1},i_k}\right)\right\} \\
	& \leq \sum_{i_1\neq \cdots \neq i_{k-1}} \max_{i'\neq i_1}\left(\exp\left\{-(1-o(1))C_k \Delta_{i_1,i'}\right\} \right)\cdot  \exp\left\{-(1-o(1)) C_k \left(\Delta_{i_1, i_2} + \cdots \Delta_{i_{k-2},i_{k-1}}\right)\right\} \\
	& \qquad \times \sum_{i_k\in [n]\setminus\{i_{k-1}\}}  \exp\left\{-(1-o(1)) C_k  \Delta_{i_{k-1},i_k}\right\} \\
	& \leq \calE_k \sum_{i_1\neq \cdots \neq i_{k-1}} \max_{i'\neq i_1}\left(\exp\left\{-(1-o(1))C_k \Delta_{i_1,i'}\right\} \right)\cdot  \exp\left\{-(1-o(1)) C_k \left(\Delta_{i_1, i_2} + \cdots \Delta_{i_{k-2},i_{k-1}}\right)\right\},
\end{align*}
where the last inequality is by \eqref{eq:ub_further_assumption_1}. 
Recursively applying the above arguments, we arrive at
\begin{align*}
	& \sum_{i_1\neq \cdots\neq i_k} \exp\left\{\frac{-(1-o(1))C_k\left\|(I_k^\leftarrow -I_k)U_{i_{1:k},\bigcdot} D\right\|_F^2}{\sigma_{\max}^2}\right\}\\
	& \leq \calE_k^{k-1} \sum_{i_1\in[n]} \max_{i'\neq i_1} \exp\left\{-(1-o(1))C_k \Delta_{i_1,i'}\right\} \\
	& \leq \calE_k^{k-1} \sum_{i \neq i'} \exp\left\{-(1-o(1))C_k \Delta_{i,i'}\right\} \\
	& \leq n \calE_k^k.
\end{align*}

To show \eqref{eq:ub_mismatch_proportion_under_futher_assumption_1}, we proceed by
\begin{align*}
	& \frac{1}{n} \sum_{k=2}^n \sum_{i_1\neq \cdots\neq i_k} \exp\left\{\frac{-(1-o(1))C_k\left\|(I_k^\leftarrow -I_k)U_{i_{1:k},\bigcdot} D\right\|_F^2}{\sigma_{\max}^2}\right\}\\
	& \leq \sum_{k=2}^n \calE_k^k\\
	& \leq \sum_{k=2}^n \left( \frac{1}{n} \sum_{i\neq i'} \exp\left\{ \frac{-(1-o(1)) C_6 \|(U_{i,\bigcdot} - U_{i', \bigcdot})D\|^2 }{\sigma_{\max}^2} \right\}\right)^k\\
	& \leq \frac{1}{n} \sum_{i\neq i'} \exp\left\{ \frac{-(1-o(1)) C_6 \|(U_{i,\bigcdot} - U_{i', \bigcdot})D\|^2 }{\sigma_{\max}^2} \right\},
\end{align*}
where the last inequality is by summing over geometric series. 
By repeating the arguments in the proof of Theorem \ref{thm:ub} (in particular, the arguments starting from \eqref{eq:ub_exponential_plus_polynomial}), we conclude that\eqref{eq:ub_mismatch_proportion_under_futher_assumption_1} holds.

To show \eqref{eq:ub_mismatch_proportion_under_futher_assumption_2}, we proceed by
\begin{align*}
	& \frac{1}{n} \sum_{k=2}^n \sum_{i_1\neq \cdots\neq i_k} \exp\left\{\frac{-(1-o(1))C_k\left\|(I_k^\leftarrow -I_k)U_{i_{1:k},\bigcdot} D\right\|_F^2}{\sigma_{\max}^2}\right\}\\
	& \leq \sum_{k=2}^{n} \calE_k^k\\
	& \leq \sum_{k=2}^n \left[\max_{i\in[n]} \sum_{i'\in[n]\setminus\{i\}} \exp\left\{-\frac{(1-o(1)) C_k \|(U_{i,\bigcdot} - U_{i',\bigcdot})D\|^2}{\sigma_{\max}^2}\right\}\right]^k \\
	& = \sum_{k=2}^n  \max_{i\in[n]} \left[\sum_{i'\in[n]\setminus\{i\}} \exp\left\{-\frac{(1-o(1)) C_k \|(U_{i,\bigcdot} - U_{i',\bigcdot})D\|^2}{\sigma_{\max}^2}\right\}\right]^k \\
	& \leq \sum_{k=2}^n \max_{i\in[n]} \sum_{i'\in[n]\setminus\{i\}}\exp\left\{-\frac{(1-o(1)) k C_k \|(U_{i,\bigcdot}-U_{i',\bigcdot})D\|^2}{\sigma_{\max}^2}\right\},
\end{align*}
where the last inequality is by \eqref{eq:ub_further_assumption_2}. For the summands with $2\leq k\leq 5$, since $kC_k \geq 1/4$, we have
\begin{align*}
	& \sum_{k=2}^5 \max_{i\in[n]} \sum_{i'\in[n]\setminus\{i\}}\exp\left\{-\frac{(1-o(1)) k C_k \|(U_{i,\bigcdot}-U_{i',\bigcdot})D\|^2}{\sigma_{\max}^2}\right\} \\
	& \leq \sum_{k=2}^k \max_{i\in[n]} \sum_{i'\in[n]\setminus\{i\}}\exp\left\{-\frac{(1-o(1)) \|(U_{i,\bigcdot}-U_{i',\bigcdot})D\|^2}{4\sigma_{\max}^2}\right\} \\
	& \leq \max_{i\in[n]} \sum_{i'\in[n]\setminus\{i\}}\exp\left\{-\frac{(1-o(1)) \|(U_{i,\bigcdot}-U_{i',\bigcdot})D\|^2}{4\sigma_{\max}^2}\right\} \\
	& \leq \frac{1}{n} \sum_{i\neq i'}\exp\left\{-\frac{(1-o(1)) \|(U_{i,\bigcdot}-U_{i',\bigcdot})D\|^2}{4\sigma_{\max}^2}\right\},
\end{align*}
where the last line is by \eqref{eq:ub_further_assumption_1}.
For the summands with $k\geq 6$, we have
\begin{align*}
	& \sum_{k=6}^n \max_{i\in[n]} \sum_{i'\in[n]\setminus\{i\}}\exp\left\{-\frac{(1-o(1)) k C_k \|(U_{i,\bigcdot}-U_{i',\bigcdot})D\|^2}{\sigma_{\max}^2}\right\} \\
	& = \sum_{k=6}^n \max_{i\in[n]} \sum_{i'\in[n]\setminus\{i\}}\exp\left\{-\frac{(1-o(1)) k C_6 \|(U_{i,\bigcdot}-U_{i',\bigcdot})D\|^2}{\sigma_{\max}^2}\right\} \\
	& \leq \sum_{k=6}^n \max_{i\in[n]} \left[\sum_{i'\in[n]\setminus\{i\}}\exp\left\{-\frac{(1-o(1)) C_6 \|(U_{i,\bigcdot}-U_{i',\bigcdot})D\|^2}{\sigma_{\max}^2}\right\} \right]^k\\
	& = \sum_{k=6}^n \left[\max_{i\in[n]}\sum_{i'\in[n]\setminus\{i\}}\exp\left\{-\frac{(1-o(1)) C_6 \|(U_{i,\bigcdot}-U_{i',\bigcdot})D\|^2}{\sigma_{\max}^2}\right\} \right]^k,
\end{align*}
which is a summation of geometric series.
Note that
\begin{align*}
	& \left[\max_{i\in[n]}\sum_{i'\in[n]\setminus\{i\}}\exp\left\{-\frac{(1-o(1)) C_6 \|(U_{i,\bigcdot}-U_{i',\bigcdot})D\|^2}{\sigma_{\max}^2}\right\}\right]^6 \\
	& \leq 
	\max_{i\in[n]}\sum_{i'\in[n]\setminus\{i\}}\exp\left\{-\frac{(1-o(1)) 6C_6 \|(U_{i,\bigcdot}-U_{i',\bigcdot})D\|^2}{\sigma_{\max}^2}\right\}\\
	& \leq 
	\max_{i\in[n]}\sum_{i'\in[n]\setminus\{i\}}\exp\left\{-\frac{(1-o(1)) \|(U_{i,\bigcdot}-U_{i',\bigcdot})D\|^2}{4\sigma_{\max}^2}\right\}\\
	& \leq \frac{1}{n} \sum_{i\neq i'}\exp\left\{-\frac{(1-o(1)) \|(U_{i,\bigcdot}-U_{i',\bigcdot})D\|^2}{4\sigma_{\max}^2}\right\} \\
	& = o(1),
\end{align*}	
where the penultimate line is by \eqref{eq:ub_further_assumption_1}. Thus, we have
\begin{align*}
	& \sum_{k=6}^n \max_{i\in[n]} \sum_{i'\in[n]\setminus\{i\}}\exp\left\{-\frac{(1-o(1)) k C_k \|(U_{i,\bigcdot}-U_{i',\bigcdot})D\|^2}{\sigma_{\max}^2}\right\} \\
	& \leq \frac{1}{n} \sum_{i\neq i'}\exp\left\{-\frac{(1-o(1)) \|(U_{i,\bigcdot}-U_{i',\bigcdot})D\|^2}{4\sigma_{\max}^2}\right\}.
\end{align*}
In summary, we have shown that
\begin{align*}
	& \frac{1}{n} \sum_{k=2}^n \sum_{i_1\neq \cdots\neq i_k} \exp\left\{\frac{-(1-o(1))C_k\left\|(I_k^\leftarrow -I_k)U_{i_{1:k},\bigcdot} D\right\|_F^2}{\sigma_{\max}^2}\right\}\\
	& \lesssim \frac{1}{n} \sum_{i\neq i'}\exp\left\{-\frac{(1-o(1)) \|(U_{i,\bigcdot}-U_{i',\bigcdot})D\|^2}{4\sigma_{\max}^2}\right\}.
\end{align*}
Again by repeating the arguments in the proof of Theorem \ref{thm:ub},
we conclude that\eqref{eq:ub_mismatch_proportion_under_futher_assumption_2} holds.

\section{Auxiliary results}

\begin{lemma}[Gaussian tail bound]
\label{lemma:gaussian_tail}
For any $t>0$, we have
$$
	\frac{1}{\sqrt{2\pi}}\left(\frac{1}{t} - \frac{1}{t^2}\right) e^{-t^2/2} \leq \Phi(-t) = \bbP\left(N(0, 1)\geq t\right) \leq \frac{1}{\sqrt{2\pi}}\cdot \frac{1}{t} e^{-t^2/2}.
$$
\end{lemma}

\begin{lemma}[$\chi^2$ tail bound]
\label{lemma:chisq_tail}
Let $X$ be a $\chi$-squared random variable with $k$ degrees of freedom and fix $\delta \in (0, 1)$. Then with probability at least $1-\delta$, we have
$$
	X \leq k + 2\sqrt{k \log(1/\delta)} + 2 \log(1/\delta).
$$
\end{lemma}
\begin{proof}
	This is a direct consequence of Lemma 1 in \cite{laurent2000adaptive}.
\end{proof}

\begin{lemma}[Operator norm of Gaussian Wigner matrices]
\label{lemma:gaussian_wigner_op_norm}
Let $X\in \bbR^{n\times p}$ be a matrix with i.i.d.~$N(0, 1)$ entries and fix $\delta\in (0, 1)$. Then with probability at least $1-\delta$, we have
$$
	\|X\| \lesssim \sqrt{n} + \sqrt{p} + \sqrt{\log(1/\delta)}.
$$
\end{lemma}
\begin{proof}
	By Corollary 3.11 in \cite{bandeira2016sharp}, for any $t\geq0$ and $\ep\in(0, 1/2]$, we have
	$$
		\bbP\left(\|X\|\geq (1+\ep)(\sqrt{n}+\sqrt{p}) + t\right) \leq (n\land p)e^{-t^2/c_\ep},
	$$
	where $c_\ep$ is an absolute constant depending only on $\ep$. Setting the right-hand side above to be $\delta$ gives the desired result.
\end{proof}

\begin{lemma}[Hanson-Wright inequality]
\label{lemma:hs_ineq}
	Let $X\in\bbR^n$ be a random vector with i.i.d.~$N(0, 1)$ entries and let $A$ be an $n\times n$ matrix. Then for any $\delta\in(0, 1)$, with probability at least $1-\delta$, we have
	$$
		|X^\top A X - \bbE[X^\top A X]| \lesssim \|A\|_F \sqrt{\log(1/\delta)} + \|A\| {\log(1/\delta)}.
	$$
\end{lemma}
\begin{proof}
	This is a direct consequence of Theorem 1.1 in \cite{rudelson2013hanson}.
\end{proof}

\begin{lemma}
\label{lemma:inner_prod_and_norm}
	For any $k, \ell\in \bbN$ and $A\in \bbR^{k\times k}$, we have
	$$
		\la A, [I_k - (I_k^\rightarrow)^\ell]A \ra = \la A, [I_k - (I_k^\leftarrow)^\ell]A \ra = \frac{1}{2} \|[I_k - (I_k^\leftarrow)^\ell] A\|_F^2 = \frac{1}{2} \|[I_k - (I_k^\rightarrow)^\ell] A\|_F^2.
	$$
\end{lemma}
\begin{proof}
	The first equality follows from 
	$$
	\la A, (I_k^\rightarrow)^\ell A \ra = \la [(I_k^\rightarrow)^\ell]^\top A, A \ra = \la (I_k^\leftarrow)^\ell A, A \ra = \la A, (I_k^\leftarrow)^\ell A \ra .
	$$
	For the second equality, note that
	\begin{align*}
		\|[I_k - (I_k^\leftarrow)^\ell] A\|_F^2 
		& = \la A, [I_k - (I_k^\rightarrow)^\ell] [I_k - (I_k^\leftarrow)^\ell] A \ra \\
		& = 2 \la A, A \ra - \la A, (I_k^\leftarrow)^\ell A \ra - \la A, (I_k^\rightarrow)^\ell A \ra \\
		& = 2 \la A, A \ra - 2 \la A , (I_k^\leftarrow)^\ell A \ra.
	\end{align*}
	Dividing both sides by two, we get the desired result. The third equality follows similarly.
\end{proof}

\end{appendices}

\end{document}